\documentclass{amsart} 

\usepackage[abs]{overpic}		

\usepackage{amsmath,amssymb,amsthm} 
\usepackage{graphicx} 
\usepackage[arrow, matrix, curve]{xy} 
\usepackage{color} 
\usepackage{bm} 
\usepackage{hyperref} 

\definecolor{darkgreen}{cmyk}{.854,.2549, 1,.1347}
\definecolor{pink}{rgb}{.9961,0, .9804}

\usepackage{rotating}

\DeclareMathOperator{\tb}{tb}
\DeclareMathOperator{\rot}{rot}
\DeclareMathOperator{\lk}{lk}

\DeclareMathOperator{\de}{d}
\DeclareMathOperator{\e}{e}
\DeclareMathOperator{\PD}{PD}

\DeclareMathOperator{\su}{s}
\DeclareMathOperator{\cs}{cs}

\newcommand{\R}{\mathbb{R}}
\newcommand{\Z}{\mathbb{Z}}
\newcommand{\Q}{\mathbb{Q}}
\newcommand{\N}{\mathbb{N}}

\newcommand{\xist}{\xi_{\mathrm{st}}}

\newcommand{\Pushoff}{{\def\svgwidth{1,6ex}\,\,
\begingroup%
  \makeatletter%
  \providecommand\color[2][]{%
    \errmessage{(Inkscape) Color is used for the text in Inkscape, but the package 'color.sty' is not loaded}%
    \renewcommand\color[2][]{}%
  }%
  \providecommand\transparent[1]{%
    \errmessage{(Inkscape) Transparency is used (non-zero) for the text in Inkscape, but the package 'transparent.sty' is not loaded}%
    \renewcommand\transparent[1]{}%
  }%
  \providecommand\rotatebox[2]{#2}%
  \ifx\svgwidth\undefined%
    \setlength{\unitlength}{49.54732089bp}%
    \ifx\svgscale\undefined%
      \relax%
    \else%
      \setlength{\unitlength}{\unitlength * \real{\svgscale}}%
    \fi%
  \else%
    \setlength{\unitlength}{\svgwidth}%
  \fi%
  \global\let\svgwidth\undefined%
  \global\let\svgscale\undefined%
  \makeatother%
  \begin{picture}(1,1.00059637)%
    \put(0,0){\includegraphics[width=\unitlength,page=1]{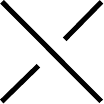}}%
  \end{picture}%
\endgroup%
\,\,}} 

\newtheoremstyle{thm}{}{}{\itshape}{}{\bfseries}{}{ }{}
\newtheoremstyle{definition}{}{}{}{}{\bfseries}{}{ }{}

\theoremstyle{plain}
\newtheorem{Theorem}{Theorem}[section]
\newtheorem{thm}[Theorem]{Theorem}
\newtheorem{lem}[Theorem]{Lemma}
\newtheorem{prop}[Theorem]{Proposition}
\newtheorem{cor}[Theorem]{Corollary}
\newtheorem*{Theorem-ohne}{Theorem}

\newtheorem{ques}[Theorem]{Question}

\newtheorem*{thm1}{Theorem~\ref{thm:continequality}}
\newtheorem*{thm2}{Theorem~\ref{thm.overtwisted}}

\newtheorem*{thm3}{Theorem~\ref{thm:S3rational}}
\newtheorem*{cor1}{Corollary~\ref{cor:S3}}
\newtheorem*{cor2}{Corollary~\ref{cor:S3char}}

\newtheorem*{thm4}{Theorem~\ref{thm:LegendrianSurgeryNumber}}
\newtheorem*{thm5}{Theorem~\ref{thm:P}}
\newtheorem*{cor3}{Corollary~\ref{cor:P}}
\newtheorem*{cor4}{Corollary~\ref{cor:Pstd}}
\newtheorem*{thm6}{Theorem~\ref{thm:E}}
\newtheorem*{cor5}{Corollary~\ref{cor:Est}}
\newtheorem*{cor6}{Corollary~\ref{cor:Echar}}
\newtheorem*{thm7}{Theorem~\ref{thm:S1S2lowerBound}}
\newtheorem*{cor7}{Corollary~\ref{cor:S1xS2}}
\newtheorem*{prop1}{Proposition~\ref{prop:xist}}

\theoremstyle{definition}
\newtheorem{defi}[Theorem]{Definition}
\newtheorem{rem}[Theorem]{Remark}
\newtheorem{ex}[Theorem]{Example}

\begingroup\expandafter\expandafter\expandafter\endgroup
\expandafter\ifx\csname pdfsuppresswarningpagegroup\endcsname\relax
\else
\fi

\begin{document}


\title[Contact surgery numbers]{Contact surgery numbers}

\author{John Etnyre}
\address{School of Mathematics, Georgia Institute of Technology, Atlanta, GA}
\email{etnyre@math.gatech.edu}

\author{Marc Kegel}
\address{Humboldt-Universit\"at zu Berlin, Rudower Chaussee 25, 12489 Berlin, Germany.}
\email{kegemarc@math.hu-berlin.de, kegelmarc87@gmail.com}

\author{Sinem Onaran}
\address{Department of Mathematics, Hacettepe University, 06800 Beytepe-Ankara, Turkey.}
\email{sonaran@hacettepe.edu.tr}

\date{\today}


\begin{abstract}
It is known that any contact $3$-manifold can be obtained by rational contact Dehn surgery along a Legendrian link $L$ in the standard tight contact $3$-sphere. We define and study various versions of contact surgery numbers, the minimal number of components of a surgery link $L$ describing a given contact $3$-manifold under consideration. 

In the first part of the paper, we relate contact surgery numbers to other invariants in terms of various inequalities. In particular, we show that the contact surgery number of a contact manifold is bounded from above by the topological surgery number of the underlying topological manifold plus three. 

In the second part, we compute contact surgery numbers of all contact structures on the $3$-sphere. Moreover, we completely classify the contact structures with contact surgery number one on $S^1\times S^2$, the Poincar\'e homology sphere and the Brieskorn sphere $\Sigma(2,3,7)$. We conclude that there exist infinitely many non-isotopic contact structures on each of the above manifolds which cannot be obtained by a single rational contact surgery from the standard tight contact $3$-sphere. We further obtain results for the $3$-torus and lens spaces.

As one ingredient of the proofs of the above results we generalize computations of the homotopical invariants of contact structures to contact surgeries with more general surgery coefficients which might be of independent interest. 
\end{abstract}

\makeatletter
\@namedef{subjclassname@2020}{%
  \textup{2020} Mathematics Subject Classification}
\makeatother



\maketitle


\section{Introduction}

A fundamental result due to Lickorish--Wallace says that any $3$-manifold can be obtained from $S^3$ by a finite sequence of Dehn surgeries. Moreover, it is known that any $3$-manifold can be obtained from $S^3$ by performing only \textbf{integer} Dehn surgeries ({\em i.e.\ }the surgery coefficients are integers) or by only \textbf{even} Dehn surgeries~\cite{Ka79} ({\em i.e.\ }the surgery coefficients are all even integers) or by only \textbf{simple} Dehn surgeries ({\em i.e.\ }all components of the surgery link in $S^3$ are unknots). These special surgery diagrams have interesting geometric meanings. A surgery diagram with only integer surgery coefficients specifies a simply connected closed $4$-manifold bounded by the surgered $3$-manifold and a surgery diagram with only even coefficients specifies a parallelization of the surgered $3$-manifold~\cite{Ka79}, {\em cf.}~\cite{DGGK18}. 

For a given $3$-manifold $M$, Auckly~\cite{Au97} defined the \textbf{surgery number} $\su(M)$ of $M$ as the minimal number of components of a link in $S^3$ needed to describe $M$ as a rational Dehn surgery along that link. It is natural to define surgery numbers for more restricted classes of surgeries. Let $*$ be an extra property of a Dehn surgery. In this work, we will consider integer surgeries ($*=\Z$), simple surgeries ($*=U$) or both ($*=\Z,U$). Then the $*$ \textbf{surgery number} $\su_*(M)$ is defined to be the minimal number of components of a link in $S^3$ needed to describe $M$ as Dehn surgery with property $*$ along the link. 

While upper bounds on surgery numbers can be provided by constructing explicit Dehn surgery presentations of a given manifold, it is much harder to find lower bounds. Lower bounds on surgery numbers can be given for example by the rank of the first homology or the fundamental group or induced from the linking pairing~\cite{Au97}. More advanced obstructions (that can yield lower bounds on homology spheres) can be obtained from gauge theory~\cite{Au97}, from Heegaard Floer homology~\cite{HoKaLi14,HoLi16} or from the $SU(2)$ character variety of the fundamental group~\cite{SZ19}. In general, the known lower bounds do not coincide with the number of components in explicit surgery diagrams and therefore it is in general hard to compute surgery numbers. At the moment there is no homology sphere known to have a surgery number larger than $2$.

In the present paper, we consider analogous questions in contact geometry and study contact surgery numbers of contact $3$-manifolds. To define contact surgery numbers let us first briefly recall contact surgery. Let $K$ be a Legendrian knot in a contact $3$-manifold $(M,\xi)$. Then there is the classical construction to do Dehn surgery along $K$ with respect to the contact structure. The result is that for any non-vanishing \textbf{contact surgery coefficient} ({\em i.e.\ }measured with respect to the contact longitude of $K$, obtained by pushing $K$ into the Reeb-direction), there exist finitely many tight contact structures on the newly glued-in solid torus that fit together with the old contact structure to give a global contact structure on the surgered manifold~\cite{Gir00, Ho00}. 

If the contact surgery coefficient is of the form $1/n$, for $n\in\Z$, the contact structure on the newly glued-in solid torus is unique. Therefore, one often restricts to contact $(\pm1)$-surgeries. However, these diagrams usually tend to be very complicated. Here we want to allow more general coefficients and adopt the convention that if we say that some contact manifold $(M,\xi)$ can be obtained by contact $r$-surgery, we mean there exist one choice of the above finitely many tight contact structures on the newly glued-in solid torus such that we get $(M,\xi)$ after the surgery. 

The following is the generalization of the Lickorish--Wallace theorem to contact geometry.
\begin{thm} [Ding--Geiges~\cite{DiGe04}]\label{thm:DingGeiges}
Let $(M,\xi)$ be a contact $3$-manifold. Then $(M,\xi)$ can be obtained by rational contact Dehn surgery along a Legendrian link in $(S^3,\xist)$. Moreover, one can assume all contact surgery coefficients to be of the form $\pm1$ and all components of the Legendrian link to be unknots with Thurston--Bennequin numbers $\tb=-1$ or $-2$. 
\end{thm}
The addendum that the Legendrian link can be assumed to consist of Legendrian unknots with Thurston--Bennequin invariants $\tb=-1$ or $-2$ is due to Avdek~\cite[Theorem~1.7]{Av13}. A further generalization, observed in~\cite{DiGeSt04}, is that one can present any contact manifold in a surgery diagram with only one $+1$ surgery coefficient and all other coefficients negative.

With this Theorem~\ref{thm:DingGeiges} in mind, it is natural to ask what is the simplest Legendrian link describing a given contact $3$-manifold. In this paper we propose to study various versions of contact surgery numbers, measuring the complexity of surgery links $L$ of a given contact $3$-manifold in terms of the number of components of $L$.

\begin{defi}  \label{defi:contactsurgerynumbers} 
Let $(M,\xi)$ be a contact $3$-manifold. We define the \textbf{contact surgery number} $\cs(M,\xi)$ to be the minimal number of components of a Legendrian link in $(S^3,\xist)$ needed to describe $(M,\xi)$ as a rational contact surgery along the link (with non-vanishing contact surgery coefficients). 

The definition can be extended by requiring additional properties of the surgeries. Let $*$ be an additional property of a contact surgery. In this article, we will consider the following properties:
\begin{itemize}
	\item $*=\Z$ : All contact surgery coefficients are non-vanishing integers.
	\item $*=1/\Z$ : All contact surgery coefficients are of the form $\pm 1/n$.
	\item $*=\pm1$ : All contact surgery coefficients are $\pm1$.
	\item $*=U$ : Each component of the Legendrian surgery link is a Legendrian realization of the unknot.
	\item $*=L$ : All contact surgery coefficients are negative except for at most one surgery coefficient which is $+1$. ($L$ stands for Legendrian surgery.)
\end{itemize}
 Sometimes we will consider combinations of these properties. For any of the above properties, we define the $*$ \textbf{contact surgery number} $\cs_*(M,\xi)$ to be the minimal number of components of a Legendrian link with property $*$ in $(S^3,\xist)$ needed to describe $(M,\xi)$ as contact Dehn surgery along that link. 
\end{defi}

\subsection{Inequalities}
Upper bounds to contact surgery numbers and relations to other invariants in terms of inequalities can be obtained by explicit constructions of surgery diagrams. In Section~\ref{sec:Kirby} and~\ref{sec:tight} we discuss contact Kirby moves (that is, modifications of contact surgery diagrams not changing the contactomorphism type of the surgered manifold), which we then use to relate various versions of contact surgery numbers to each other.

Our first result in that direction roughly says that if we can obtain a contact manifold via a single contact surgery then we can also obtain it via a surgery along a $3$-component link in which every component is a Legendrian realization of the unknot.

\begin{thm1}  Let $(M,\xi)$ be a contact $3$-manifold. Then the following inequalities hold true
	\begin{align*}
	\cs_{U}(M,\xi)&\leq3\cs(M,\xi),\\
	\cs_{U, \Z}(M,\xi)&\leq3\cs_{\Z}(M,\xi),\\
	\cs_{U, 1/\Z}(M,\xi)&\leq3\cs_{1/\Z}(M,\xi),\\
	\cs_{U, \pm1}(M,\xi)&\leq3\cs_{\pm1}(M,\xi).
	\end{align*}
\end{thm1}

We emphasize that contact surgery numbers are bounded from below by topological surgery numbers since every contact surgery induces a topological surgery on the underlying manifolds. In Section~\ref{upperboundsoncsn} we present various general upper bounds on contact surgery numbers of contact manifolds that depend only on the topological surgery numbers of the underlying topological manifold. Our main result is as follows. 

\begin{thm2} \label{thm.overtwisted0}
Let $(M,\xi)$ be a contact manifold. Then
\begin{align*}
\cs_{\pm1}(M,\xi)&\leq \su_\Z(M)+3,\\
\cs_{L,\pm1}(M,\xi)&\leq \su_\Z(M)+4,\\
\cs(M,\xi)&\leq \su(M)+3.
\end{align*}
\end{thm2}

Similar bounds hold true for the $U$-versions of the surgery numbers, see Section~\ref{upperboundsoncsn} for the details.

Our method of proof is constructive, but in concrete situations, the bounds can often be improved. In the following we will describe examples of contact manifolds $(M,\xi)$ where the difference $\cs(M,\xi)-\su(M)$ is $0$, $1$ and $2$, we do not know if there exist examples where the difference is $3$.
 
 \begin{ques}\label{ques:difference}
 	Does there exists a contact manifold $(M,\xi)$ such that $$\cs(M,\xi)-\su(M)=3?$$
 \end{ques}

\subsection{Computations of contact surgery numbers}
In the second part of the article, we explicitly compute contact surgery numbers for some contact manifolds. In particular, we compute the contact surgery numbers of all contact structures on $S^3$ and we classify all contact structures with $\cs=1$ on the Poincar\'e homology sphere $P$, on the Brieskorn homology sphere $\Sigma(2,3,7)$, and on $S^1\times S^2$. 

\subsubsection{The \texorpdfstring{$3$}{3}-sphere}
We first state the results for $S^3$. To do so, we first need to recall the classification of contact structures on $S^3$ by Eliashberg~\cite{El89,El92} and introduce some notation.

Let $\xist$ be the unique tight contact structure on $S^3$. This is the unique contact manifold with a vanishing contact surgery number. Overtwisted contact structures on $S^3$, up to isotopy, are in one-to-one correspondence with plane fields, up to homotopy; which in turn can be indexed by the integers. This indexing can be done with Gompf's $\de_3$-invariant from~\cite{Go98} (where it is denoted by $\theta$ but is normalized differently). Here we choose the normalization of the $\de_3$-invariant such that $\de_3(\xist)=0$. 

With this normalization, the $\de_3$-invariants of contact structures on homology spheres take values in the integers. We denote the unique overtwisted contact structure on a homology sphere $M$ with $\de_3$-invariant equal to $n\in\Z$ by $\xi_n$. 

\begin{thm3}
An overtwisted contact structure on $S^3$ has $\cs=1$ if and only if its $\de_3$-invariant is of the form
\begin{equation*}
k(q+qk-2z)
\end{equation*}
for $q\geq1$, $k\geq1$ and $z=0,1,\ldots ,q-1$, or
\begin{equation*}
qk(k+1)+2k+1
\end{equation*}
for $q\leq-1$, $k\geq0$, or
\begin{equation*}
qk(k-1)+1
\end{equation*}
for $q\leq-1$, $k\geq0$. All other overtwisted contact structures on $S^3$ have $\cs=2$.
\end{thm3}

In Section~\ref{sec:comp} we also derive similar results for $\cs_{\pm1}$, $\cs_{1/\Z}$ and $\cs_\Z$. We formulate some direct corollaries, see Section~\ref{sec:comp} for more applications.

\begin{cor1}
There exist infinitely many non-isotopic contact structures on $S^3$ which cannot be obtained by a single rational contact surgery from $(S^3,\xist)$. As a concrete example we see that $\cs_{\pm1}(S^3,\xi_0)=2$. See Figure~\ref{fig:S310} (ii) for a surgery description of $(S^3,\xi_0)$ along a $2$-component link.
\end{cor1}

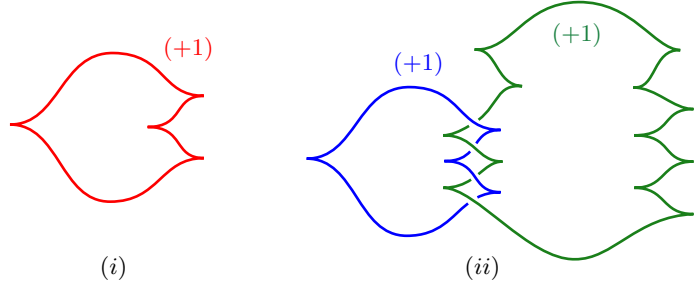
\begin{figure}[htbp]{\small
		\begin{overpic}
			{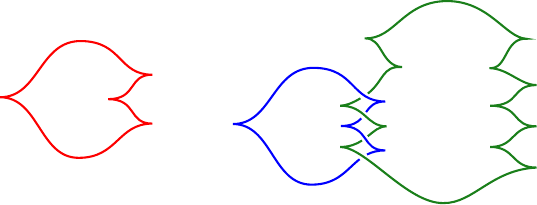}
			\put(58, 78){\color{red} $(+1)$}
			\put(145, 72){\color{blue} $(+1)$}
			\put(205, 83){\color{darkgreen} $(+1)$}
			\put(34, -5){$(i)$}
			\put(172, -5){$(ii)$}
	\end{overpic}}
	\caption{(i) This is the only contact structure on $S^3$ with $\cs_{\pm1}=1$ and it is the unique contact $(\pm1)$-surgery description along a single Legendrian knot of $(S^3,\xi_1)$. (ii)~ This surgery diagram is a contact $(\pm1)$-surgery description of $(S^3,\xi_0)$ along a $2$-component link with a vanishing $\de_3$-invariant.}
	\label{fig:S310}
\end{figure}

The next corollary say that $(S^3,\xi_1)$ has a unique contact $(\pm1)$-surgery diagram along a single Legendrian knot in $(S^3,\xist)$.

\begin{cor2}
If $(S^3,\xi_1)$ is obtained by a single contact $(\pm1)$-surgery along a Legendrian knot $K$ in $(S^3,\xist)$, then $K$ has to be the (unoriented) Legendrian unknot with Thurston--Bennequin invariant $\tb=-2$ and rotation number $|\rot|=1$ and the contact surgery coefficient has to be $(+1)$. See Figure~\ref{fig:S310} (i).
\end{cor2}

We also have results about the Legendrian contact surgery numbers of contact structures on $S^3$. 

\begin{thm4}
A contact structure on $S^3$ has $\cs_L=1$ if and only if it is isotopic to $\xi_1$. Moreover, there exist infinite families of contact structures on $S^3$ with Legendrian contact surgery number equal to two. As concrete examples we have
\begin{itemize}
	\item [(1)] $\cs_{L,\pm1}(S^3,\xi_{2k})=2$ for $k\in \Z$, and
	\item [(2)] $\cs_L(S^3,\xi_{1-2k})=2$ for $k\in \N$. 
\end{itemize}
\end{thm4}

\subsubsection{The Poincar\'e homology sphere and the Brieskorn sphere \texorpdfstring{$\Sigma(2,3,7)$}{E(2,3,7)}}
In Sections~\ref{sec:P} and~\ref{sec:E},  we study the contact surgery numbers of the  the Poincar\'e homology sphere $P= \Sigma(2,3,5)$ and the Brieskorn homology sphere $\Sigma(2,3,7)$.

\begin{thm5}
A contact structure on $P$ has $\cs=1$ if and only if its $\de_3$-invariant is of the form
\begin{equation*}
m(3-m)-1
\end{equation*}
where $m$ is an arbitrary integer with $m\geq3$. Furthermore, we have $\cs_\Z(P,\xi)\leq3$ for any contact structure on $P$.
\end{thm5}

\begin{cor3}
There exists infinitely many contact structures on $P$ with $\cs=2$.
\end{cor3}

\begin{cor4}
The unique tight contact structure $\xist$ on $P$ cannot be obtained by a rational contact surgery along a single Legendrian knot from $(S^3,\xist)$. More concretely, 
\begin{equation*}
2\leq\cs(P,\xist)\leq\cs_{\pm1}(P,\xist)\leq3.
\end{equation*}
However, $(P,\xist)$ can be obtained by a single Legendrian surgery (i.e. contact $(-1)$-surgery) along a Legendrian knot in an overtwisted $S^3$.
\end{cor4}

In particular, we have an example of a tight manifold, $(P,\xist)$, whose contact surgery number differs from the topological surgery number of the underlying topological manifold. We note that tight contact structures are often attained in simple surgery diagrams. For example, all tight contact structures on lens spaces occur from a rational contact surgery along Legendrian realizations of the unknot~\cite{Gir00, Ho00}. However, the example of $(P,\xist)$ shows that this simplicity is not always reflected in the number of components of the surgery link. Nonetheless, $(P,\xist)$ can be obtained by a contact surgery along a $3$-chain link.

\begin{thm6}
The unique tight contact structure $\xist$ on $\Sigma(2,3,7)$ can be obtained by a single Legendrian surgery along a right-handed Legendrian trefoil and thus $\cs_{\pm1}(\Sigma(2,3,7),\xist)=1$. An overtwisted contact structure on $\Sigma(2,3,7)$ has $\cs=1$ if and only if its $\de_3$-invariant is of the form
\begin{equation*}
l(3-l)-1\,\text{ or }\, m(m-1)
\end{equation*}
where $m,l$ are arbitrary integers with $l\geq0$ and $m\geq2$. Furthermore, we have $\cs_\Z(\Sigma(2,3,7),\xi)\leq3$ for any contact structure on $\Sigma(2,3,7)$.
\label{thm6}
\end{thm6}

It is known that $\Sigma(2,3,7)$ can be obtained by a topological $(+1)$-surgery on the figure eight knot. However, we see that is not the case for $(\Sigma(2,3,7),\xist)$.
\begin{cor5}
One cannot obtain $(\Sigma(2,3,7),\xist)$ by a rational contact surgery along a Legendrian realization of the figure eight knot. 
\end{cor5}
From Theorem~\ref{thm:E}, $(\Sigma(2,3,7),\xist)$ can be obtained by surgery on the right-handed trefoil which is the unique knot with this property. 
\begin{cor6}
There is a unique Legendrian surgery description of $(\Sigma(2,3,7),\xist)$ along a single Legendrian knot, the contact $(-1)$-surgery along the unique (unoriented) Legendrian right-handed trefoil knot with $\tb=0$ and $|\rot|=1$. 
\end{cor6}

\subsubsection{\texorpdfstring{$S^1\times S^2$}{S1xS2}}
We now consider contact $3$-manifolds with non-trivial first homology. Specifically, we study the contact surgery numbers of $S^1\times S^2$, which has a unique tight contact structure $\xist$. All the remaining contact structures on $S^1\times S^2$ are overtwisted and by Eliashberg's classification of overtwisted contact structures, they only depend on the algebraic topology of the underlying tangential $2$-plane field. Since $H_1(S^1\times S^2)=\Z$, there is another invariant of the tangential $2$-plane field, namely its $spin^c$ structure, see Sections~\ref{sec:d3} and~\ref{sec:S1xS2} for more discussion.

\begin{thm7}
There exists exactly one contact structure in every $spin^c$ structure of $S^1\times S^2$ which can be obtained by a contact surgery along a single Legendrian knot in $(S^3,\xist)$. In particular, no overtwisted contact structure on $S^1\times S^2$ with trivial Euler class can be obtained by a surgery along a single Legendrian knot in $(S^3,\xist)$.
\end{thm7}

\begin{prop1}
A contact structure $\xi$ on $S^1\times S^2$ has $\cs_{\pm1}(S^1\times S^2,\xi)=1$ if and only if $(S^1\times S^2,\xi)$ is contactomorphic to $(S^1\times S^2,\xist)$. 

Moreover, contact $(+1)$-surgery along the Legendrian unknot with $\tb=-1$ and $\rot=0$ is the unique contact $(\pm1)$-surgery diagram of $(S^1\times S^2,\xist)$ along a single Legendrian knot in $(S^3,\xist)$.
\end{prop1}

\begin{cor7}
	There exist infinitely many contact structures on $S^1\times S^2$ with $\cs=2$.
\end{cor7}

\subsection{Miscellaneous results}
In Section~\ref{sec:T3} we study the contact surgery number of the infinite family of tight structures on $T^3$. We will see that any tight contact structure $\xi$ on $T^3$ satisfies 
\[
cs_{\pm}(T^3,\xi)\leq 4.
\]
As a curious corollary, we find an infinite family of distinct non-loose Legendrian knots with the same classical invariants in a fixed overtwisted contact structure $\eta$ on $S^1\times S^2\# S^1\times S^2$. Recall, a Legendrian knot in an overtwisted contact structure is non-loose if the complement of a standard neighborhood of the Legendrian knot is tight. We will show there is a family $L_n, n\in \N$ of Legendrian knots in $(S^1\times S^2\# S^1\times S^2, \eta)$ having the same classical invariants, such that Legendrian surgery on $L_n$ yields the tight contact structure on $T^3$ with Giroux torsion $n-1$. This is the aforementioned infinite family. 

In the following subsection, we study contact surgery numbers of contact lens spaces. We give specific examples of tight contact structures with $\cs_\pm=1$ and criteria for a tight contact structure to have $\cs_\pm\geq 2$. 

In the last section of the paper, we study Legendrian surgeries between overtwisted contact structures on $S^3$ and show that there are many constraints on such surgeries.

\section*{Conventions}
Throughout this paper, we assume the reader to be familiar with Dehn surgery and contact topology on the level of~\cite{PrSo97,GoSt99,Ge08,OzSt04}.

We work in the smooth category. All manifolds, maps, etc. are assumed to be smooth. We assume all $3$-manifolds to be connected, closed, oriented and all contact structures to be positive and co-orientable. Since a contact surgery diagram only determines a contact manifold up to contactomorphism, we only consider contact manifolds up to contactomorphism. Whenever we say that a contact structure is unique, we mean that it is unique up to contactomorphism. Legendrian links in $(S^3,\xist)$ are always presented in their front projection.

We choose the normalization of the $\de_3$-invariant such that $\de_3(S^3,\xist)=0$. Although this normalization differs from the normalization used for example in~\cite{Go98,DiGeSt04,DuKe16}, we believe it to be more natural since then the $\de_3$-invariants of contact structures on homology spheres take values in the integers and since the $\de_3$-invariant behaves additively under connected sum. The referee informed us that there is another point in favor of the normalization used in this work: with our normalization, we have that $-d_3$ agrees with the degree of the contact invariants in gauge theory. 

\section*{Acknowledgments}
Our collaboration started at the \textit{ICERM workshop Perspectives on Dehn Surgery}, July 15--19, 2019 after we had worked for some time independently on this topic. We thank ICERM (the
Institute for Computational and Experimental Research in Mathematics in Providence, RI) for the productive environment.

A large part of this project was carried out when M.K. was \textit{Oberwolfach Research Fellow} at the \textit{Mathematisches Forschungsinstitut Oberwolfach} in August 2020. We thank both institutions for their support. 
 
M.K.~would also like to thank Irena Matkovi\u{c} for detailed answers to questions about the classifications of contact structures on Seifert fibred spaces at a very early stage of this project and Marco Golla for his interest and for providing the alternative argument for proving tightness of some contact surgery diagrams via the contact class in Heegaard Floer homology.
 
We are happy to thank the referee, for his careful reading and detailed feedback, which helped to improve, clarify, and correct several statements in this work.
 
 J.E.~was partially supported by NSF grant DMS-1906414. 
 S.O.~was partially supported by T\"UB{\.I}TAK 1001-119F411.

\section{Contact Kirby moves}\label{sec:Kirby}

To establish explicit upper bounds on contact surgery numbers we will need several modifications on contact surgery diagrams not changing the contactomorphism type of the resulting contact manifold. In this section, we will discuss these modifications but first, we introduce the following notations. Let $K$ be a Legendrian knot in $(M,\xi)$, then we will write $(M_K(r),\xi_K(r))$ for a contact manifold obtained by contact $r$-surgery along $K$ where it is important to remember that the contact structure $\xi_K(r)$ is in general not unique. We often write as short notation
\begin{equation*}
K(r):=(M_K(r),\xi_K(r)).
\end{equation*}
For $K$ together with a Legendrian push-off of $K$ we write $K{\def\svgwidth{1,6ex}\,\,
\begingroup%
  \makeatletter%
  \providecommand\color[2][]{%
    \errmessage{(Inkscape) Color is used for the text in Inkscape, but the package 'color.sty' is not loaded}%
    \renewcommand\color[2][]{}%
  }%
  \providecommand\transparent[1]{%
    \errmessage{(Inkscape) Transparency is used (non-zero) for the text in Inkscape, but the package 'transparent.sty' is not loaded}%
    \renewcommand\transparent[1]{}%
  }%
  \providecommand\rotatebox[2]{#2}%
  \ifx\svgwidth\undefined%
    \setlength{\unitlength}{49.54732089bp}%
    \ifx\svgscale\undefined%
      \relax%
    \else%
      \setlength{\unitlength}{\unitlength * \real{\svgscale}}%
    \fi%
  \else%
    \setlength{\unitlength}{\svgwidth}%
  \fi%
  \global\let\svgwidth\undefined%
  \global\let\svgscale\undefined%
  \makeatother%
  \begin{picture}(1,1.00059637)%
    \put(0,0){\includegraphics[width=\unitlength,page=1]{PushOff.pdf}}%
  \end{picture}%
\endgroup%
\,\,} K$. We denote a copy of $K$ with $n$ extra stabilizations by $K_n$, if this knot $K_n$ is again stabilized $m$ times this is denoted by $K_{n,m}$. For example, $K{\def\svgwidth{1,6ex}\,\,
\begingroup%
  \makeatletter%
  \providecommand\color[2][]{%
    \errmessage{(Inkscape) Color is used for the text in Inkscape, but the package 'color.sty' is not loaded}%
    \renewcommand\color[2][]{}%
  }%
  \providecommand\transparent[1]{%
    \errmessage{(Inkscape) Transparency is used (non-zero) for the text in Inkscape, but the package 'transparent.sty' is not loaded}%
    \renewcommand\transparent[1]{}%
  }%
  \providecommand\rotatebox[2]{#2}%
  \ifx\svgwidth\undefined%
    \setlength{\unitlength}{49.54732089bp}%
    \ifx\svgscale\undefined%
      \relax%
    \else%
      \setlength{\unitlength}{\unitlength * \real{\svgscale}}%
    \fi%
  \else%
    \setlength{\unitlength}{\svgwidth}%
  \fi%
  \global\let\svgwidth\undefined%
  \global\let\svgscale\undefined%
  \makeatother%
  \begin{picture}(1,1.00059637)%
    \put(0,0){\includegraphics[width=\unitlength,page=1]{PushOff.pdf}}%
  \end{picture}%
\endgroup%
\,\,} K_n{\def\svgwidth{1,6ex}\,\,
\begingroup%
  \makeatletter%
  \providecommand\color[2][]{%
    \errmessage{(Inkscape) Color is used for the text in Inkscape, but the package 'color.sty' is not loaded}%
    \renewcommand\color[2][]{}%
  }%
  \providecommand\transparent[1]{%
    \errmessage{(Inkscape) Transparency is used (non-zero) for the text in Inkscape, but the package 'transparent.sty' is not loaded}%
    \renewcommand\transparent[1]{}%
  }%
  \providecommand\rotatebox[2]{#2}%
  \ifx\svgwidth\undefined%
    \setlength{\unitlength}{49.54732089bp}%
    \ifx\svgscale\undefined%
      \relax%
    \else%
      \setlength{\unitlength}{\unitlength * \real{\svgscale}}%
    \fi%
  \else%
    \setlength{\unitlength}{\svgwidth}%
  \fi%
  \global\let\svgwidth\undefined%
  \global\let\svgscale\undefined%
  \makeatother%
  \begin{picture}(1,1.00059637)%
    \put(0,0){\includegraphics[width=\unitlength,page=1]{PushOff.pdf}}%
  \end{picture}%
\endgroup%
\,\,} K_{n,m}$ denotes a $3$-component Legendrian link consisting of the following three Legendrian knots: one copy of $K$, a Legendrian push-off $K_n$ of $K$ that is $n$ times stabilized, where the $n$-stabilizations are in general not specified but fixed, and a Legendrian push-off $K_{n,m}$ of $K_n$ that is stabilized (in addition to the $n$ fixed stabilizations of $K_n$) stabilized another $m$ times, where again the $m$ extra stabilizations are in general not specified but fixed.

The following three lemmas due to Ding and Geiges are fundamental in the study of contact Dehn surgeries along Legendrian links.

\begin{lem}[Cancellation Lemma]\label{lem:cancelation}
	Contact $(1/n)$-surgery ($n\in\Z\setminus\{0\}$) along a Legendrian knot $K$ in $(M,\xi)$ and contact $(-1/n)$-surgery along a Legendrian push-off of $K$ cancel each other, i.e.\ the result is contactomorphic to $(M,\xi)$. In the notation introduced above this reads
	\begin{equation*}
	K\left(\frac{1}{n}\right){\def\svgwidth{1,6ex}\,\,
\begingroup%
  \makeatletter%
  \providecommand\color[2][]{%
    \errmessage{(Inkscape) Color is used for the text in Inkscape, but the package 'color.sty' is not loaded}%
    \renewcommand\color[2][]{}%
  }%
  \providecommand\transparent[1]{%
    \errmessage{(Inkscape) Transparency is used (non-zero) for the text in Inkscape, but the package 'transparent.sty' is not loaded}%
    \renewcommand\transparent[1]{}%
  }%
  \providecommand\rotatebox[2]{#2}%
  \ifx\svgwidth\undefined%
    \setlength{\unitlength}{49.54732089bp}%
    \ifx\svgscale\undefined%
      \relax%
    \else%
      \setlength{\unitlength}{\unitlength * \real{\svgscale}}%
    \fi%
  \else%
    \setlength{\unitlength}{\svgwidth}%
  \fi%
  \global\let\svgwidth\undefined%
  \global\let\svgscale\undefined%
  \makeatother%
  \begin{picture}(1,1.00059637)%
    \put(0,0){\includegraphics[width=\unitlength,page=1]{PushOff.pdf}}%
  \end{picture}%
\endgroup%
\,\,} \,K\left(-\frac{1}{n}\right)=(M,\xi).
	\end{equation*}
\end{lem}

\begin{lem}[Replacement Lemma]\label{lem:replacemenet}
	Contact $(\pm1/n)$-surgery ($n\in\N$) along a Legendrian knot $K$ in $(M,\xi)$ yields the same contact manifold as contact $(\pm1)$-surgeries along $n$ Legendrian push-offs of $K$, i.e.
	\begin{equation*}
	K\left(\pm\frac{1}{n}\right)\cong K(\pm1){\def\svgwidth{1,6ex}\,\,
\begingroup%
  \makeatletter%
  \providecommand\color[2][]{%
    \errmessage{(Inkscape) Color is used for the text in Inkscape, but the package 'color.sty' is not loaded}%
    \renewcommand\color[2][]{}%
  }%
  \providecommand\transparent[1]{%
    \errmessage{(Inkscape) Transparency is used (non-zero) for the text in Inkscape, but the package 'transparent.sty' is not loaded}%
    \renewcommand\transparent[1]{}%
  }%
  \providecommand\rotatebox[2]{#2}%
  \ifx\svgwidth\undefined%
    \setlength{\unitlength}{49.54732089bp}%
    \ifx\svgscale\undefined%
      \relax%
    \else%
      \setlength{\unitlength}{\unitlength * \real{\svgscale}}%
    \fi%
  \else%
    \setlength{\unitlength}{\svgwidth}%
  \fi%
  \global\let\svgwidth\undefined%
  \global\let\svgscale\undefined%
  \makeatother%
  \begin{picture}(1,1.00059637)%
    \put(0,0){\includegraphics[width=\unitlength,page=1]{PushOff.pdf}}%
  \end{picture}%
\endgroup%
\,\,}\cdots{\def\svgwidth{1,6ex}\,\,
\begingroup%
  \makeatletter%
  \providecommand\color[2][]{%
    \errmessage{(Inkscape) Color is used for the text in Inkscape, but the package 'color.sty' is not loaded}%
    \renewcommand\color[2][]{}%
  }%
  \providecommand\transparent[1]{%
    \errmessage{(Inkscape) Transparency is used (non-zero) for the text in Inkscape, but the package 'transparent.sty' is not loaded}%
    \renewcommand\transparent[1]{}%
  }%
  \providecommand\rotatebox[2]{#2}%
  \ifx\svgwidth\undefined%
    \setlength{\unitlength}{49.54732089bp}%
    \ifx\svgscale\undefined%
      \relax%
    \else%
      \setlength{\unitlength}{\unitlength * \real{\svgscale}}%
    \fi%
  \else%
    \setlength{\unitlength}{\svgwidth}%
  \fi%
  \global\let\svgwidth\undefined%
  \global\let\svgscale\undefined%
  \makeatother%
  \begin{picture}(1,1.00059637)%
    \put(0,0){\includegraphics[width=\unitlength,page=1]{PushOff.pdf}}%
  \end{picture}%
\endgroup%
\,\,} K(\pm1).
	\end{equation*}
\end{lem}

\begin{lem}[Transformation Lemma]\label{lem:algo}
	Let $K$ be a Legendrian knot in $(M,\xi)$ with given contact surgery coefficient $r\in\Q\setminus\{0\}$. 
	\begin{itemize}
		\item [(1)]	Then 
		\begin{equation*}
		K(r)\cong K\left(\frac{1}{k}\right){\def\svgwidth{1,6ex}\,\,
\begingroup%
  \makeatletter%
  \providecommand\color[2][]{%
    \errmessage{(Inkscape) Color is used for the text in Inkscape, but the package 'color.sty' is not loaded}%
    \renewcommand\color[2][]{}%
  }%
  \providecommand\transparent[1]{%
    \errmessage{(Inkscape) Transparency is used (non-zero) for the text in Inkscape, but the package 'transparent.sty' is not loaded}%
    \renewcommand\transparent[1]{}%
  }%
  \providecommand\rotatebox[2]{#2}%
  \ifx\svgwidth\undefined%
    \setlength{\unitlength}{49.54732089bp}%
    \ifx\svgscale\undefined%
      \relax%
    \else%
      \setlength{\unitlength}{\unitlength * \real{\svgscale}}%
    \fi%
  \else%
    \setlength{\unitlength}{\svgwidth}%
  \fi%
  \global\let\svgwidth\undefined%
  \global\let\svgscale\undefined%
  \makeatother%
  \begin{picture}(1,1.00059637)%
    \put(0,0){\includegraphics[width=\unitlength,page=1]{PushOff.pdf}}%
  \end{picture}%
\endgroup%
\,\,} \,K\left(\frac{1}{\frac{1}{r}-k}\right)
		\end{equation*}
		holds for all integers $k\in \Z$. 
		\item [(2)] If the contact surgery coefficient $r$ is negative, one can write $r$ uniquely as 
		\begin{equation}\label{expand}
		r=[r_1+1,r_2,\ldots,r_n]:=r_1+1-\cfrac{1}{r_2-\cfrac{1}{\dotsb -\cfrac{1}{r_n} }}
		\end{equation}
		with integers $r_1,\ldots ,r_n\leq -2$ and we have
		\begin{equation*}
		K(r)\cong K_{|2+r_1|}(-1){\def\svgwidth{1,6ex}\,\,
\begingroup%
  \makeatletter%
  \providecommand\color[2][]{%
    \errmessage{(Inkscape) Color is used for the text in Inkscape, but the package 'color.sty' is not loaded}%
    \renewcommand\color[2][]{}%
  }%
  \providecommand\transparent[1]{%
    \errmessage{(Inkscape) Transparency is used (non-zero) for the text in Inkscape, but the package 'transparent.sty' is not loaded}%
    \renewcommand\transparent[1]{}%
  }%
  \providecommand\rotatebox[2]{#2}%
  \ifx\svgwidth\undefined%
    \setlength{\unitlength}{49.54732089bp}%
    \ifx\svgscale\undefined%
      \relax%
    \else%
      \setlength{\unitlength}{\unitlength * \real{\svgscale}}%
    \fi%
  \else%
    \setlength{\unitlength}{\svgwidth}%
  \fi%
  \global\let\svgwidth\undefined%
  \global\let\svgscale\undefined%
  \makeatother%
  \begin{picture}(1,1.00059637)%
    \put(0,0){\includegraphics[width=\unitlength,page=1]{PushOff.pdf}}%
  \end{picture}%
\endgroup%
\,\,} K_{|2+r_1|,|2+r_2|}(-1){\def\svgwidth{1,6ex}\,\,
\begingroup%
  \makeatletter%
  \providecommand\color[2][]{%
    \errmessage{(Inkscape) Color is used for the text in Inkscape, but the package 'color.sty' is not loaded}%
    \renewcommand\color[2][]{}%
  }%
  \providecommand\transparent[1]{%
    \errmessage{(Inkscape) Transparency is used (non-zero) for the text in Inkscape, but the package 'transparent.sty' is not loaded}%
    \renewcommand\transparent[1]{}%
  }%
  \providecommand\rotatebox[2]{#2}%
  \ifx\svgwidth\undefined%
    \setlength{\unitlength}{49.54732089bp}%
    \ifx\svgscale\undefined%
      \relax%
    \else%
      \setlength{\unitlength}{\unitlength * \real{\svgscale}}%
    \fi%
  \else%
    \setlength{\unitlength}{\svgwidth}%
  \fi%
  \global\let\svgwidth\undefined%
  \global\let\svgscale\undefined%
  \makeatother%
  \begin{picture}(1,1.00059637)%
    \put(0,0){\includegraphics[width=\unitlength,page=1]{PushOff.pdf}}%
  \end{picture}%
\endgroup%
\,\,}\cdots{\def\svgwidth{1,6ex}\,\,
\begingroup%
  \makeatletter%
  \providecommand\color[2][]{%
    \errmessage{(Inkscape) Color is used for the text in Inkscape, but the package 'color.sty' is not loaded}%
    \renewcommand\color[2][]{}%
  }%
  \providecommand\transparent[1]{%
    \errmessage{(Inkscape) Transparency is used (non-zero) for the text in Inkscape, but the package 'transparent.sty' is not loaded}%
    \renewcommand\transparent[1]{}%
  }%
  \providecommand\rotatebox[2]{#2}%
  \ifx\svgwidth\undefined%
    \setlength{\unitlength}{49.54732089bp}%
    \ifx\svgscale\undefined%
      \relax%
    \else%
      \setlength{\unitlength}{\unitlength * \real{\svgscale}}%
    \fi%
  \else%
    \setlength{\unitlength}{\svgwidth}%
  \fi%
  \global\let\svgwidth\undefined%
  \global\let\svgscale\undefined%
  \makeatother%
  \begin{picture}(1,1.00059637)%
    \put(0,0){\includegraphics[width=\unitlength,page=1]{PushOff.pdf}}%
  \end{picture}%
\endgroup%
\,\,} K_{|2+r_1|,\ldots,|2+r_n|}(-1).
		\end{equation*}
	\end{itemize}
\end{lem}

For a proof of the Cancellation Lemma and the Replacement Lemma, we refer to~\cite{DiGe01}, cf.~\cite{DiGeSt04}. The proof of the Transformation Lemma is given in~\cite{DiGe04}, cf.~\cite{DiGeSt04}. 

\begin{rem} We make the following two remarks.
	\begin{itemize}
		\item [(1)] The first part of the Transformation Lemma in~\cite{DiGe04} and~\cite[Section~1]{DiGeSt04} is only formulated for natural numbers $k\in\N$ and negative contact surgery coefficients $r<0$, but the proof given there works exactly the same for all integers $k\in \Z$ and all surgery coefficients $r\in\Q\setminus\{0\}$.
		\item [(2)] Observe that by choosing $k$ such that $\frac{1}{r} -k$ is negative one can use the Transformation Lemma to change any contact $r$-surgery (for $r\neq0$) into a sequence of contact $(\pm1)$-surgeries along a different Legendrian link. Moreover in~\cite{DiGe04} and~\cite{DiGeSt04} it is shown that all the different choices of stabilizations in this Legendrian link correspond exactly to the different tight contact structures one can choose on the new glued-in solid torus in the contact $r$-surgery. So any choice of contact structure $\xi_K(r)$ on $M_K(r)$ corresponds to exactly one choice of stabilizations in the Transformation Lemma.
	\end{itemize}
\end{rem}

For any negative rational number $r<0$ there exists a unique continued fraction expansion as in part $(2)$ of Lemma~\ref{lem:algo}. Moreover, we observe that if $-(n+1)\leq r<-n$ for some $n\in\N$, then $n+1+r\in[0,1)$, and it follows that the first coefficient $r_1$ in Equation~\eqref{expand} is $r_1=-(n+2)$. Using this observation we can prove the following well-known lemma~\cite[Proposition~7]{DiGe04},~\cite{LS07},~\cite[Proposition~11.2.12]{OzSt04}, which we will use later.

\begin{lem}\label{lem:transfoappli}
Let $r>0$ be a positive contact surgery coefficient, which is not the reciprocal of an integer. Then any contact $r$-surgery $K(r)$ along $K$ is equivalent to a contact $(1/n)$-surgery along $K$, for some $n\in\N$, followed by a negative contact $r'$-surgery, for some $r'<0$, along a Legendrian push-off $K'$ of $K$ with at least one extra stabilization, as in Figure~\ref{fig:transformation}. 
\end{lem}

\begin{figure}[htbp]{\small
		\begin{overpic}
			{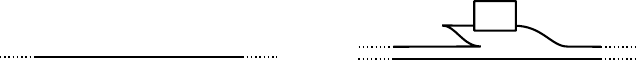}
			\put(25, 16){$r>0$}
			\put(100, 16){$K$}
			\put(230.5,  28){stab}
			\put(186, 23){$r'<0$}
			\put(278, 21){$K'$}
			\put(278, -3){$K$}
			\put(194, -3){$\frac 1n$}
			\put(148, 12){$\cong$}
	\end{overpic}}
\caption{An application of the Transformation Lemma. The knot $K'$ is a Legendrian push-off of $K$ away from the shown segment. The box represents an unspecified number of stabilizations of arbitrary signs.}
\label{fig:transformation}
\end{figure}
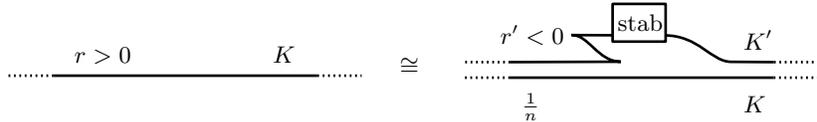

\begin{proof} 
With the above discussion, it is enough to find an $n\in\N$ such that
\begin{equation*}
-2\leq r'=\frac{1}{\frac{1}{r}-n}<0.
\end{equation*}
Then the first coefficient $r_1$ in the continued fraction expansion of $r'$ is smaller than $-2$ and therefore, a stabilization is needed. An easy calculation shows that any $n\in\N$ with 
\begin{equation*}
\frac{1}{r}+\frac{1}{2}\leq n
\end{equation*}
suffice.
\end{proof}

From Lemma~\ref{lem:replacemenet} and~\ref{lem:algo} the following specification for integer contact surgeries can be deduced.

\begin{lem}\label{lem:surgerylemma}
	Let $K$ be a Legendrian knot and $n$ be a positive integer. Then the following contactomorphisms hold:
	\begin{align}
	K(-n)\cong&K_{n-1}(-1)\label{eq1}\\
	K(+n)\cong&K(+1){\def\svgwidth{1,6ex}\,\,
\begingroup%
  \makeatletter%
  \providecommand\color[2][]{%
    \errmessage{(Inkscape) Color is used for the text in Inkscape, but the package 'color.sty' is not loaded}%
    \renewcommand\color[2][]{}%
  }%
  \providecommand\transparent[1]{%
    \errmessage{(Inkscape) Transparency is used (non-zero) for the text in Inkscape, but the package 'transparent.sty' is not loaded}%
    \renewcommand\transparent[1]{}%
  }%
  \providecommand\rotatebox[2]{#2}%
  \ifx\svgwidth\undefined%
    \setlength{\unitlength}{49.54732089bp}%
    \ifx\svgscale\undefined%
      \relax%
    \else%
      \setlength{\unitlength}{\unitlength * \real{\svgscale}}%
    \fi%
  \else%
    \setlength{\unitlength}{\svgwidth}%
  \fi%
  \global\let\svgwidth\undefined%
  \global\let\svgscale\undefined%
  \makeatother%
  \begin{picture}(1,1.00059637)%
    \put(0,0){\includegraphics[width=\unitlength,page=1]{PushOff.pdf}}%
  \end{picture}%
\endgroup%
\,\,} K\left(-\frac{n}{n-1}\right)\label{eq2}\\
	K\left(-\frac{n}{n-1}\right)\cong&K_1\left(-\frac{1}{n-1}\right)\label{eq3}\\
	K(+n)\cong&K(+1){\def\svgwidth{1,6ex}\,\,
\begingroup%
  \makeatletter%
  \providecommand\color[2][]{%
    \errmessage{(Inkscape) Color is used for the text in Inkscape, but the package 'color.sty' is not loaded}%
    \renewcommand\color[2][]{}%
  }%
  \providecommand\transparent[1]{%
    \errmessage{(Inkscape) Transparency is used (non-zero) for the text in Inkscape, but the package 'transparent.sty' is not loaded}%
    \renewcommand\transparent[1]{}%
  }%
  \providecommand\rotatebox[2]{#2}%
  \ifx\svgwidth\undefined%
    \setlength{\unitlength}{49.54732089bp}%
    \ifx\svgscale\undefined%
      \relax%
    \else%
      \setlength{\unitlength}{\unitlength * \real{\svgscale}}%
    \fi%
  \else%
    \setlength{\unitlength}{\svgwidth}%
  \fi%
  \global\let\svgwidth\undefined%
  \global\let\svgscale\undefined%
  \makeatother%
  \begin{picture}(1,1.00059637)%
    \put(0,0){\includegraphics[width=\unitlength,page=1]{PushOff.pdf}}%
  \end{picture}%
\endgroup%
\,\,} K_1\left(-\frac{1}{n-1}\right)\label{eq4}
	\end{align}
\end{lem}

\begin{proof}
	Equation~(\ref{eq1}) follows from Lemma~\ref{lem:algo} Part~(2) and the trivial continued fraction expansion $-n =[-n]$.
	
	To prove the other equations, we first notice that Equation~(\ref{eq2}) is a direct consequence of Lemma~\ref{lem:algo} Part~(1) (for $k=1$) and that Equation~(\ref{eq4}) follows from 
	Equation~(\ref{eq2}) and ~(\ref{eq3}). A straightforward induction argument verifies the continued fraction expansion
	\begin{equation*}
	-\frac{n}{n-1}=[\underbrace{-2,\ldots,-2}_{n-1}].
	\end{equation*}
	Equation~(\ref{eq2}) follows from Lemma~\ref{lem:algo} Part~(2) together with Lemma~\ref{lem:replacemenet}.
\end{proof}

We will often concentrate on positive integer contact surgeries, i.e. contact surgeries for surgery coefficient $n\in\N$. The above lemma says that there are exactly two contact surgeries corresponding to an integer contact surgery, depending on the sign of the stabilization.

We will also need the following version of a handle slide for contact surgery which is due to Ding and Geiges~\cite{DiGe09}, cf.~\cite{CEK21} for the generalization to other contact surgery coefficients.

\begin{lem}[Contact handle slide]\label{lem:contactHandleSlide} 
	The two blue Legendrian knots in Figure~\ref{fig:contactHandleSlide} in the exterior of the red Legendrian knot along which we perform a contact $(-1)$-surgery are isotopic.
	
	If the blue Legendrian knot comes equipped with a contact surgery coefficient $r\in\Q\setminus\{0\}$ then the contact surgery coefficient does not change under the isotopy.
\end{lem}

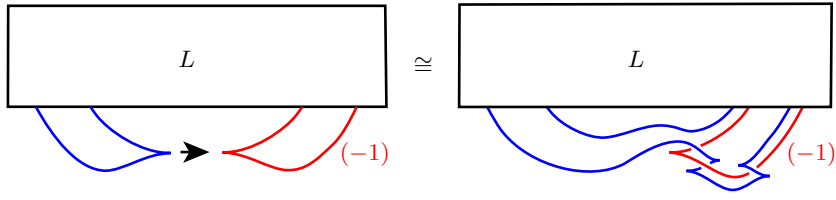
\begin{figure}[htbp]{\small
		\begin{overpic}
			{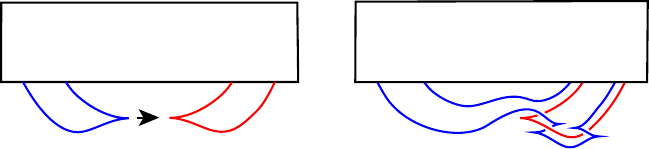}
			\put(65, 47){$L$}
			\put(235, 47){$L$}
			\put(126,  12){\color{red} $(-1)$}
			\put(295, 12){\color{red} $(-1)$}
			\put(154, 45){$\cong$}
	\end{overpic}}
	\caption{A version of the contact handle slide}
	\label{fig:contactHandleSlide}
\end{figure}

As a direct application of a contact handle slide, we obtain the following generalization of a slam dunk~\cite{DiGe09}.

\begin{lem}[Contact slam dunk]\label{lem:contactSlamDunk}
	Let $K$ be a Legendrian knot and $U$ be a Legendrian meridian of $K$ with $\tb=-1$. Then contact $(-1)$-surgery along $K$ followed by contact $(+1)$-surgery along $U$ cancel each other.
	
\end{lem}

All moves so far preserve the number of $(\pm1)$-framed Legendrian knots mod $2$. But in fact, there are moves where this number is not preserved. The most important one which we want to use later is the following, first obtained by Lisca and Stipsciz for contact $(\pm1)$-surgery in~\cite{LS11} and then generalized in~\cite[Lemma~3.7]{MT18}. 

\begin{lem}[Lantern destabilization]\label{lem:LanternDestabilization}
	Let $K$ be a Legendrian knot. Then
	\begin{equation*}
	K_1(+1){\def\svgwidth{1,6ex}\,\,
\begingroup%
  \makeatletter%
  \providecommand\color[2][]{%
    \errmessage{(Inkscape) Color is used for the text in Inkscape, but the package 'color.sty' is not loaded}%
    \renewcommand\color[2][]{}%
  }%
  \providecommand\transparent[1]{%
    \errmessage{(Inkscape) Transparency is used (non-zero) for the text in Inkscape, but the package 'transparent.sty' is not loaded}%
    \renewcommand\transparent[1]{}%
  }%
  \providecommand\rotatebox[2]{#2}%
  \ifx\svgwidth\undefined%
    \setlength{\unitlength}{49.54732089bp}%
    \ifx\svgscale\undefined%
      \relax%
    \else%
      \setlength{\unitlength}{\unitlength * \real{\svgscale}}%
    \fi%
  \else%
    \setlength{\unitlength}{\svgwidth}%
  \fi%
  \global\let\svgwidth\undefined%
  \global\let\svgscale\undefined%
  \makeatother%
  \begin{picture}(1,1.00059637)%
    \put(0,0){\includegraphics[width=\unitlength,page=1]{PushOff.pdf}}%
  \end{picture}%
\endgroup%
\,\,} K_{1,1}\left(-\frac{1}{n}\right){\def\svgwidth{1,6ex}\,\,
\begingroup%
  \makeatletter%
  \providecommand\color[2][]{%
    \errmessage{(Inkscape) Color is used for the text in Inkscape, but the package 'color.sty' is not loaded}%
    \renewcommand\color[2][]{}%
  }%
  \providecommand\transparent[1]{%
    \errmessage{(Inkscape) Transparency is used (non-zero) for the text in Inkscape, but the package 'transparent.sty' is not loaded}%
    \renewcommand\transparent[1]{}%
  }%
  \providecommand\rotatebox[2]{#2}%
  \ifx\svgwidth\undefined%
    \setlength{\unitlength}{49.54732089bp}%
    \ifx\svgscale\undefined%
      \relax%
    \else%
      \setlength{\unitlength}{\unitlength * \real{\svgscale}}%
    \fi%
  \else%
    \setlength{\unitlength}{\svgwidth}%
  \fi%
  \global\let\svgwidth\undefined%
  \global\let\svgscale\undefined%
  \makeatother%
  \begin{picture}(1,1.00059637)%
    \put(0,0){\includegraphics[width=\unitlength,page=1]{PushOff.pdf}}%
  \end{picture}%
\endgroup%
\,\,} K_{1,1,s}\left(r\right)\cong K(+1){\def\svgwidth{1,6ex}\,\,
\begingroup%
  \makeatletter%
  \providecommand\color[2][]{%
    \errmessage{(Inkscape) Color is used for the text in Inkscape, but the package 'color.sty' is not loaded}%
    \renewcommand\color[2][]{}%
  }%
  \providecommand\transparent[1]{%
    \errmessage{(Inkscape) Transparency is used (non-zero) for the text in Inkscape, but the package 'transparent.sty' is not loaded}%
    \renewcommand\transparent[1]{}%
  }%
  \providecommand\rotatebox[2]{#2}%
  \ifx\svgwidth\undefined%
    \setlength{\unitlength}{49.54732089bp}%
    \ifx\svgscale\undefined%
      \relax%
    \else%
      \setlength{\unitlength}{\unitlength * \real{\svgscale}}%
    \fi%
  \else%
    \setlength{\unitlength}{\svgwidth}%
  \fi%
  \global\let\svgwidth\undefined%
  \global\let\svgscale\undefined%
  \makeatother%
  \begin{picture}(1,1.00059637)%
    \put(0,0){\includegraphics[width=\unitlength,page=1]{PushOff.pdf}}%
  \end{picture}%
\endgroup%
\,\,} K_1\left(-\frac{1}{n-1}\right){\def\svgwidth{1,6ex}\,\,
\begingroup%
  \makeatletter%
  \providecommand\color[2][]{%
    \errmessage{(Inkscape) Color is used for the text in Inkscape, but the package 'color.sty' is not loaded}%
    \renewcommand\color[2][]{}%
  }%
  \providecommand\transparent[1]{%
    \errmessage{(Inkscape) Transparency is used (non-zero) for the text in Inkscape, but the package 'transparent.sty' is not loaded}%
    \renewcommand\transparent[1]{}%
  }%
  \providecommand\rotatebox[2]{#2}%
  \ifx\svgwidth\undefined%
    \setlength{\unitlength}{49.54732089bp}%
    \ifx\svgscale\undefined%
      \relax%
    \else%
      \setlength{\unitlength}{\unitlength * \real{\svgscale}}%
    \fi%
  \else%
    \setlength{\unitlength}{\svgwidth}%
  \fi%
  \global\let\svgwidth\undefined%
  \global\let\svgscale\undefined%
  \makeatother%
  \begin{picture}(1,1.00059637)%
    \put(0,0){\includegraphics[width=\unitlength,page=1]{PushOff.pdf}}%
  \end{picture}%
\endgroup%
\,\,} K_{1,s}\left(r\right)
	\end{equation*}
	holds for all $n,s\in\N_0$ and $r\in\Q\setminus\{0\}$, where the stabilizations of $K_1$ and $K_{1,1}$ are chosen with the same sign, the remaining stabilizations are arbitrary, see Figure~\ref{fig:lanternDestabilization}.
\end{lem}
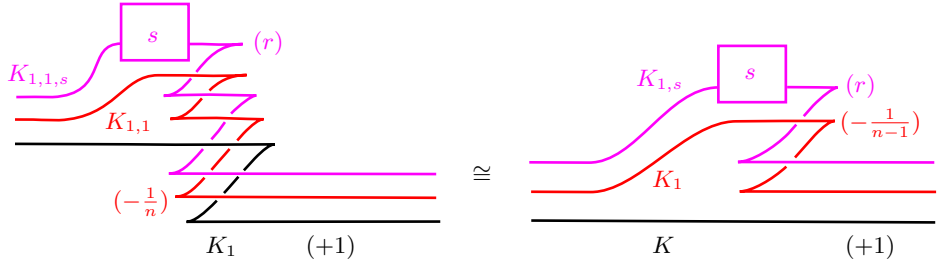
\begin{figure}[htbp]{\small
		\begin{overpic}
			{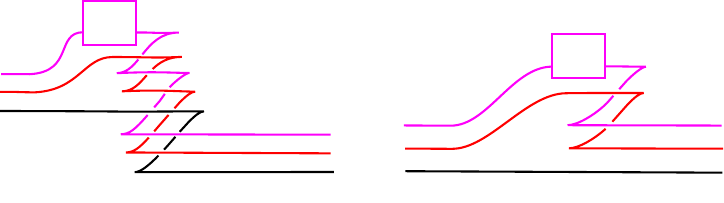}
			\put(50, 83){\color{pink}$s$}
			\put(-3, 68){\color{pink}$K_{1,1,s}$}
			\put(90, 80){\color{pink}$(r)$}
			\put(34,  50){\color{red} $K_{1,1}$}
			\put(37, 20){\color{red} $(-\frac 1n)$}
			\put(72, 3){$K_1$}
			\put(110, 3){$(+1)$}
			\put(173, 30){$\cong$}
			\put(276, 69){\color{pink}$s$}
			\put(235, 66){\color{pink}$K_{1,s}$}
			\put(314, 64){\color{pink}$(r)$}
			\put(241,  28){\color{red} $K_{1}$}
			\put(312, 50){\color{red} $(-\frac 1{n-1})$}
			\put(241, 3){$K$}
			\put(314, 3){$(+1)$}
	\end{overpic}}
		\caption{Lantern destabilization.}
		\label{fig:lanternDestabilization}
\end{figure}

For a proof, we refer to~\cite{LS11} where the equivalence of the surgery diagrams is shown by translating them into open books and then using the lantern relation and a destabilization (which explains the name). In~\cite{LS11} the proof is only given for the case of $r=\infty$. It is straightforward that their proof also works with the additional surgery curve $K_{1,1,s}$.

\section{Tightness and overtwistedness}\label{sec:tight}

Next, we will need to understand how the properties of the contact structures behave under contact surgery. Since contact $(-1)$-surgery, also called Legendrian surgery, corresponds to a symplectic handle attachment (see~\cite[Section~3]{DiGe04}), contact $(-1)$-surgery preserves Stein fillability~\cite{El90}, strong symplectic fillability~\cite{We91}, weak symplectic fillability~\cite[Lemma~2.4]{EtHo01} and the non-vanishing of the contact element in Heegaard Floer homology~\cite{OzSz05}. 

Wand~\cite{Wa15} showed that Legendrian surgery also preserves tightness for closed contact manifolds. And because of the Replacement Lemma~\ref{lem:replacemenet} and the Transformation Lemma~\ref{lem:algo} all these properties are also preserved for any negative contact surgery coefficient. For this reason, in this section, we will concentrate on positive contact surgery. First, we study contact $(1/n)$-surgeries, for $n\in\N$. Note that in this case, the resulting contact structure is unique.

\begin{lem}\label{lem:over}
	If $K(1/n_0)$ is overtwisted for $n_0\in\N$, then $K(1/n)$ is overtwisted for every $n\geq n_0$.
\end{lem}

\begin{proof}
	Follows directly from Wand's theorem and the Cancellation and Replacement Lemmas.
\end{proof}

\begin{lem}\label{lem:tight}
	If $K(1/n)$ is tight, then $K(r)$ is tight for any $r$ with 
	\begin{equation*}
	n-1<\frac{1}{r}\leq n.
	\end{equation*}
\end{lem}

\begin{proof}
	The Transformation Lemma yields
	\begin{equation*}
	K(r)=K\left(\frac{1}{n}\right){\def\svgwidth{1,6ex}\,\,
\begingroup%
  \makeatletter%
  \providecommand\color[2][]{%
    \errmessage{(Inkscape) Color is used for the text in Inkscape, but the package 'color.sty' is not loaded}%
    \renewcommand\color[2][]{}%
  }%
  \providecommand\transparent[1]{%
    \errmessage{(Inkscape) Transparency is used (non-zero) for the text in Inkscape, but the package 'transparent.sty' is not loaded}%
    \renewcommand\transparent[1]{}%
  }%
  \providecommand\rotatebox[2]{#2}%
  \ifx\svgwidth\undefined%
    \setlength{\unitlength}{49.54732089bp}%
    \ifx\svgscale\undefined%
      \relax%
    \else%
      \setlength{\unitlength}{\unitlength * \real{\svgscale}}%
    \fi%
  \else%
    \setlength{\unitlength}{\svgwidth}%
  \fi%
  \global\let\svgwidth\undefined%
  \global\let\svgscale\undefined%
  \makeatother%
  \begin{picture}(1,1.00059637)%
    \put(0,0){\includegraphics[width=\unitlength,page=1]{PushOff.pdf}}%
  \end{picture}%
\endgroup%
\,\,} \,K\left(\frac{1}{\frac{1}{r}-n}\right).
	\end{equation*}
	Since $K(1/n)$ is tight and $\frac{1}{\frac{1}{r}-n}<0$ the claim follows from Wand's theorem.
\end{proof}

Next we will analyze the case of stabilized Legendrian knots, for which we can completely determine the overtwistedness/tightness of the manifolds obtained by contact Dehn surgery along them.

\begin{thm}[Ozbagci~\cite{Oz06}]\label{thm:ozbagci}
	Let $K$ in $(M,\xi)$ be a Legendrian knot.
	\begin{itemize}
		\item [(1)] If $K$ is a stabilization, then for any surgery coefficient $r>0$ there exists at least one contact $r$-surgery such that $K(r)$ is overtwisted. Specifically, if the first stabilization in the Transformation Lemma has an opposite sign to the stabilization of $K$ then every positive contact $r$-surgery is overtwisted. Otherwise, the result may be overtwisted or tight.
		\item [(2)] If $K$ is $2$ times stabilized, once positively and once negatively, then, for $r>0$, every contact $r$-surgery is overtwisted.
	\end{itemize}
\end{thm}

Theorem~\ref{thm:ozbagci} (1) is due to Ozbagci who obtained the statement by translating the contact surgery diagrams into open books~\cite{Oz06}. Lisca and Stipsicz~\cite{LS11} gave a direct proof of Theorem~\ref{thm:ozbagci} (1) for contact $(+n)$-surgeries. Here we extend the proof from~\cite{LS11} to slightly more general statements.

\begin{proof} [Proof of Theorem~\ref{thm:ozbagci}]
	Let $K$ be a stabilized Legendrian knot. First, we want to show that $K(1/n)$ is overtwisted for all $n\in\N$. For this consider Figure~\ref{fig:overtwistedDisk} (or in case of the stabilization with the other sign consider the mirrored figure) where the Legendrian knot $K$ is presented in a neighborhood of the stabilization. Now consider a Legendrian knot $L$ in the complement of $K$ which is given by $n$ parallel Legendrian push-offs of $K$ away from the stabilization, but looks near the stabilization as in Figure~\ref{fig:overtwistedDisk}. The Legendrian knot $L$ represents, seen as a curve on a tubular neighborhood $\nu K$ of $K$, a curve isotopic to $\mu+n\lambda_C$ and therefore bounds a meridional disk of the newly glued-in solid torus in the surgered manifold. 
	
\begin{figure}[htbp]{\small
		\begin{overpic}
			{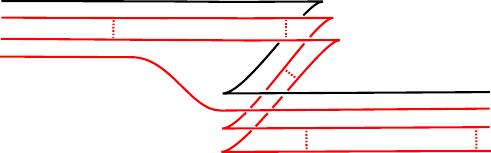}
			\put(30, 35){\color{red} $L$}
			\put(170, 35){$K$}
			\put(210, 37){$(\frac 1n)$}
	\end{overpic}}
		\caption{$L$ bounds an overtwisted disk in the surgered manifold.}
		\label{fig:overtwistedDisk}
\end{figure}	
	
	Next, we want to show that the Legendrian unknot $L$ seen as Legendrian knot in $K(1/n)$ has vanishing Thurston--Bennequin invariant $\tb_{new}$ and therefore bounds an overtwisted disk. To this end, we use the formula from~\cite{Ke18,Ke17} to compute the new Thurston--Bennequin invariant $\tb_{new}$ of $L$ in $K(1/n)$ out of the algebraic surgery data to be vanishing.

	The case of contact $r$-surgery, with $r>0$, follows similarly. By Lemma~\ref{lem:transfoappli} there exists a choice of stabilizations in the Transformation Lemma, such that $K(r)$ is equivalent to the second contact surgery diagram shown in Figure~\ref{fig:GeneralOTDisk} (or the mirrored figure in case of the stabilization with opposite sign). 
	\begin{figure}[htbp]{\small
		\begin{overpic}
			{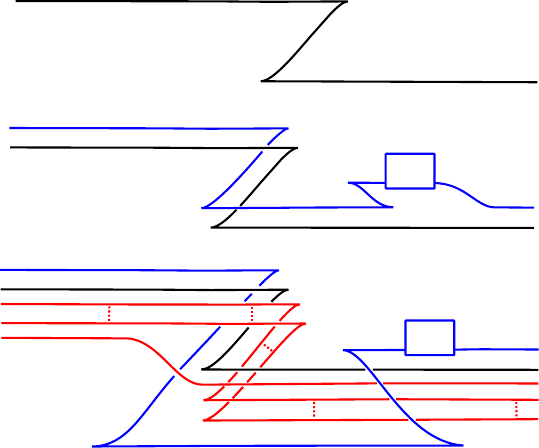}
			\put(20, 42){\color{red} $L$}
			\put(29, 10){\color{blue} $(r')<0$}
			\put(240, 50){\color{blue} $K'$}
			\put(240, 120){\color{blue} $K'$}
			\put(199, 50){\color{blue} stab}
			\put(189, 130){\color{blue} stab}
			\put(144, 44){$(\frac 1n)$}
			\put(134, 120){\color{blue} $(r')<0$}
			\put(70, 134){$(\frac 1n)$}
			\put(30, 134){$K$}
			\put(30, 205){$K$}
			\put(70, 205){$(r)$}
			\put(130, 99){\begin{rotate}{-90}$\cong$\end{rotate}}
			\put(130, 167){\begin{rotate}{-90}$\cong$\end{rotate}}
	\end{overpic}}
		\caption{$L$ bounds an overtwisted disk in the surgered manifold.}
		\label{fig:GeneralOTDisk}
\end{figure}	
	
	The third surgery diagram in Figure~\ref{fig:GeneralOTDisk} is by a Legendrian isotopy equivalent to the second one. Consider the Legendrian knot $L$ in this surgery diagram, as in the first case. If we see $K'$ as a Legendrian knot in $K(1/n)$, we see that $K'$ does not intersect the Seifert disk of $L$ in $K(1/n)$ and therefore the overtwisted disk bounded by $L$ survives after the contact $r'$-surgery along $K'$ in $K(1/n)$. So, the first part of Theorem~\ref{thm:ozbagci} follows. The contact surgeries where the stabilizations of $K'$ are all chosen with the same sign as that of $K$ can be tight follows from the next theorem.
	
	Part (2) follows directly. 
\end{proof}

The case of integer contact surgeries along knots stabilized only with one sign is summarized in the following theorem due to Lisca and Stipsicz, whose proof follows from the lantern destabilization (Lemma~\ref{lem:LanternDestabilization}). We refer to~\cite[Section~2]{LS11} for the details.

\begin{thm}[Lisca--Stipsicz~\cite{LS11}]\label{thm:LiscaStipsicz}
	Let $L$ be a Legendrian knot and let $K=L_s$ be the $s$-fold positive stabilization of $L$. The contact $n$-surgeries (again with all stabilizations positive) along $K$ are determined as follows.
	\begin{align*}
	\text{For } n<s+1:\,&K(+n)=L_{s-n+1}(+1) \text{ and} \\
	\text{for } n\geq s+1:\,&K(+n)=L(+1){\def\svgwidth{1,6ex}\,\,
\begingroup%
  \makeatletter%
  \providecommand\color[2][]{%
    \errmessage{(Inkscape) Color is used for the text in Inkscape, but the package 'color.sty' is not loaded}%
    \renewcommand\color[2][]{}%
  }%
  \providecommand\transparent[1]{%
    \errmessage{(Inkscape) Transparency is used (non-zero) for the text in Inkscape, but the package 'transparent.sty' is not loaded}%
    \renewcommand\transparent[1]{}%
  }%
  \providecommand\rotatebox[2]{#2}%
  \ifx\svgwidth\undefined%
    \setlength{\unitlength}{49.54732089bp}%
    \ifx\svgscale\undefined%
      \relax%
    \else%
      \setlength{\unitlength}{\unitlength * \real{\svgscale}}%
    \fi%
  \else%
    \setlength{\unitlength}{\svgwidth}%
  \fi%
  \global\let\svgwidth\undefined%
  \global\let\svgscale\undefined%
  \makeatother%
  \begin{picture}(1,1.00059637)%
    \put(0,0){\includegraphics[width=\unitlength,page=1]{PushOff.pdf}}%
  \end{picture}%
\endgroup%
\,\,} L_1\left(-\frac{1}{n-1-s}\right).
	\end{align*}
	In particular, it follows that for $n<s+1$ the contact manifold $K(n)$ is overtwisted and that for $n\geq s+1$ the contact manifold $K(n)$ is tight if $L(+1)$ is tight. The same statement also holds with all stabilizations negative instead of positive.
\end{thm}

For more general knot types, one can use Theorem~\ref{thm:LiscaStipsicz} to trace most information about tightness back to understanding the tightness/overtwistedness of surgeries along its non-destabilizable Legendrian realizations.  As an example, we can determine tightness/overtwistedness for most surgery slopes along Legendrian unknots.

\begin{cor}Let $U$ be a Legendrian realization of the unknot in $(S^3,\xist)$.
	\begin{itemize}
		\item [(1)] If $U$ is twice stabilized, once positively and once negatively, then any positive surgery along $U$, is overtwisted.
		\item [(2)] We denote by $U_s$ the Legendrian unknot with $\tb=-s$ and $\vert\rot\vert=s-1$, i.e. $U_s$ is obtained by $(s-1)$ stabilizations of the same sign from the unique Legendrian unknot $U=U_1$ with $\tb=-1$. If $n<s+1$ any contact $(+n)$-surgery along $U_s$ is overtwisted. If $n\geq s+1$ then exactly one contact $(+n)$-surgery along $U_s$ is overtwisted and the other is tight.
	\end{itemize}
\end{cor}

\begin{proof}
	Part (1) follows from Theorem~\ref{thm:ozbagci} Part~(2) and part (2) from Theorem~\ref{thm:LiscaStipsicz}.
\end{proof}

\begin{rem} \label{rem:HF}
	An alternative approach to determine if a given (rational) contact surgery diagram represents a tight manifold is to use the contact class in Heegaard Floer homology, see for example~\cite{LS04,Go15,MT18}. 
\end{rem}

\section{Computing the homotopical invariants after surgery} \label{sec:d3}

Later we will also need to compute the algebraic invariants of the underlying tangential $2$-plane field of a contact structure. We are mainly interested in the Euler class and Gompf's $\de_3$-invariant~\cite{Go98}. For contact $(\pm1/n)$-surgeries this can be done in a relatively straightforward way with the following result from~\cite{DuKe16}. 

\begin{lem}\label{lem:d3}
	Let $L = L_1 \sqcup \ldots \sqcup L_k$ be an oriented Legendrian link in
	$(S^3, \xist)$ and denote by $(M, \xi)$ the contact manifold obtained
	from $S^3$ by contact $(\pm {1}/{n_i})$-surgeries along $L$ ($n_i \in \Z_{>0}$). 
	We write $\tb_i$, $\rot_i$ for the Thurston--Bennequin invariant and the rotation number of $L_i$, and $l_{ij}$ for the linking between $L_i$ and $L_j$. We denote the topological surgery coefficient of $L_i$ by $r_i=p_i/q_i=\pm1/n + \tb_i$ and define the generalized linking matrix as
	\begin{align*}
	Q:=\begin{pmatrix}
	p_1&q_2 l_{12} &\cdots&q_n l_{1n}\\
	q_1 l_{21} & p_2&&\\
	\vdots&&\ddots\\
	q_1 l_{n1}&&& p_n
	\end{pmatrix}.
	\end{align*}	
	
	\begin{itemize}
		\item [(0)] The first homology group
		$H_1(M)$ of $M$ is generated by the meridians $\mu_i$
		and the relations are 
		$Q{\mathbf\mu}=0$ where $\mathbf\mu$ is the vector with entries $\mu_i$. 
		\item [(1)] The Poincar\'{e} dual of the Euler class is given by
		\begin{equation*}
		\operatorname{PD}\big(\textrm{e}(\xi)\big)=\sum_{i=1}^k n_i\rot_i\mu_i\in H_1(M).
		\end{equation*}
		
		\item [(2)] The Euler class $\textrm{e}(\xi)$ is torsion if and only if there exists
		a rational solution $\mathbf b\in\Q^k$ of $Q\mathbf b=\mathbf{rot}$.
		In this case, the $\de_3$-invariant is a well-defined rational number and computed as
		\begin{equation*}
		\de_3 = \frac{1}{4} \left(\sum_{i=1}^k n_i b_i \rot_i  +  (3-n_i)\operatorname{sign}_i\right)   - \frac{3}{4} \sigma (Q) 
		\end{equation*}
		where $\operatorname{sign}_i$ denotes the sign of the contact surgery coefficient of $L_i$ and $\sigma(Q)$ denotes the signature of $Q$. (In the proof of Theorem~5.1. in~\cite{DuKe16} it is shown that the eigenvalues of $Q$ are all real and thus the signature of $Q$ is well-defined although $Q$ is in general a non-symmetric matrix.)
	\end{itemize}
\end{lem}

In the rest of this chapter, we will present similar formulas for integer surgeries and then consider more general rational cases. Later, in the presence of $2$-torsion, we will also need to consider Gompf's half Euler class $\Gamma$~\cite{Go98}. Thus, we will shortly discuss how to compute $\Gamma$ for more general contact surgeries.

\subsection{\texorpdfstring{$\de_3$}{d3}-invariants of integer surgeries} 
Combining Lemmas~\ref{lem:d3} and~\ref{lem:surgerylemma} we present a formula for computing the possible $\de_3$-invariants of integer contact surgeries along a single Legendrian knot. In this paper, we mainly use its Corollaries~\ref{cor:+1first} and~\ref{cor:-1first} in the case that the underlying topological surgery coefficient is $\pm1$ to give lower bounds on contact surgery numbers, we hope that this more general formulation will be useful independently.

\begin{thm}\label{thm:integerD3}
	Let $K$ be a Legendrian knot in $(S^3,\xist)$ with Thurston--Bennequin invariant $t$ and rotation number $r$ and let $n\in\N$ be a positive integer. Then the possible values for the $\de_3$-invariants of the integer contact surgeries along $K$ can be computed as follows.
	
	If $t-n\neq 0$, $K(-n)$ is a rational homology sphere. Thus the $\de_3$-invariants of its contact structures are well-defined rational numbers and take the values
	\begin{equation*}
	\de_3\big( K(-n) \big)=\de_3\big( K_{n-1}(-1) \big)=\frac{1}{4}\left( \frac{{r'}^2}{t-n} -2 -3\operatorname{sign}(t-n)\right),
	\end{equation*}
	where $r'$ denotes the rotation number $\rot(K_{n-1})$ of the stabilized knot.
	
	If $t+n\neq 0$, $K(+n)$ is a rational homology sphere. Thus the $\de_3$-invariants of its contact structures are well-defined rational numbers and take the values
	\begin{align*}
	\de_3\big( K(+n) \big)=&\de_3\left( K(+1){\def\svgwidth{1,6ex}\,\,
\begingroup%
  \makeatletter%
  \providecommand\color[2][]{%
    \errmessage{(Inkscape) Color is used for the text in Inkscape, but the package 'color.sty' is not loaded}%
    \renewcommand\color[2][]{}%
  }%
  \providecommand\transparent[1]{%
    \errmessage{(Inkscape) Transparency is used (non-zero) for the text in Inkscape, but the package 'transparent.sty' is not loaded}%
    \renewcommand\transparent[1]{}%
  }%
  \providecommand\rotatebox[2]{#2}%
  \ifx\svgwidth\undefined%
    \setlength{\unitlength}{49.54732089bp}%
    \ifx\svgscale\undefined%
      \relax%
    \else%
      \setlength{\unitlength}{\unitlength * \real{\svgscale}}%
    \fi%
  \else%
    \setlength{\unitlength}{\svgwidth}%
  \fi%
  \global\let\svgwidth\undefined%
  \global\let\svgscale\undefined%
  \makeatother%
  \begin{picture}(1,1.00059637)%
    \put(0,0){\includegraphics[width=\unitlength,page=1]{PushOff.pdf}}%
  \end{picture}%
\endgroup%
\,\,} K_1\left(-\frac{1}{n-1}\right) \right)\\
	=&\frac{1}{4}\left( \frac{(1\pm r)^2 +(t-n)(1\pm r)-tn}{t+n}\mp r -2+n -3\sigma\right),
	\end{align*}
	where the upper line of signs in the equation corresponds to the positive stabilization of $K_1$ and the lower line of signs corresponds to the negative stabilization of $K_1$ and $\sigma$ is the signature of the generalized linking matrix, which takes the values
	\begin{equation*}
	\sigma=\begin{cases}
		0,\,&\text{ if }\, \operatorname{sign}(t+n)=+1,\\
		-2,\,&\text{ if }\, \operatorname{sign}(t+n)=-1.
	\end{cases}
\end{equation*}
\end{thm}

\begin{proof}
	First, we remark that the condition $t-n\neq 0$ in the case of negative surgeries and the condition $t+n\neq 0$ in the case of positive surgeries ensure that the topological surgery coefficient (measured with respect to the Seifert framing) is non-vanishing implying that the surgered manifold has torsion first homology. Thus~\cite{Go98} shows that the $\de_3$-invariant is a well-defined rational number. Next we use the formula from Lemma~\ref{lem:d3} to compute the $\de_3$-invariants.
	
	We start with the case of a negative contact surgery $K(-n) = K_{n-1}(-1)$. The topological surgery coefficient is $t-n$ and thus the linking matrix is $Q=(t-n)$ with signature $\operatorname{sign}(t-n)$. The solution $b$ of $Qb=r'$ is given by
	\begin{equation*}
	b=\frac{r'}{t-n}.
	\end{equation*}
	Plugging everything into the formula from Lemma~\ref{lem:d3} yields the claimed result.
	
	The case of positive surgery is more complicated since 
	\begin{equation*}
	K(+n) = K(+1){\def\svgwidth{1,6ex}\,\,
\begingroup%
  \makeatletter%
  \providecommand\color[2][]{%
    \errmessage{(Inkscape) Color is used for the text in Inkscape, but the package 'color.sty' is not loaded}%
    \renewcommand\color[2][]{}%
  }%
  \providecommand\transparent[1]{%
    \errmessage{(Inkscape) Transparency is used (non-zero) for the text in Inkscape, but the package 'transparent.sty' is not loaded}%
    \renewcommand\transparent[1]{}%
  }%
  \providecommand\rotatebox[2]{#2}%
  \ifx\svgwidth\undefined%
    \setlength{\unitlength}{49.54732089bp}%
    \ifx\svgscale\undefined%
      \relax%
    \else%
      \setlength{\unitlength}{\unitlength * \real{\svgscale}}%
    \fi%
  \else%
    \setlength{\unitlength}{\svgwidth}%
  \fi%
  \global\let\svgwidth\undefined%
  \global\let\svgscale\undefined%
  \makeatother%
  \begin{picture}(1,1.00059637)%
    \put(0,0){\includegraphics[width=\unitlength,page=1]{PushOff.pdf}}%
  \end{picture}%
\endgroup%
\,\,} K_1\left(-\frac{1}{n-1}\right).
	\end{equation*}
	We first build the generalized linking matrix as
	\begin{equation*}
	Q=\begin{pmatrix}
	1+t & tn-t\\
	t & tn-t-n
	\end{pmatrix}
	\end{equation*}
	with determinant $\det(Q)=-(t+n)$ and trace $\operatorname{tr}(Q)=1+n(t-1)$. If $\det(Q)<0$ the signature $\sigma$ of $Q$ is $0$, and in the case of positive determinant the signature $\sigma$ is $\pm2$ depending on the sign of the trace. By observing that when $\det(Q)>0$, we have $t+n<0$, and since $n>0$ we see that $t<0$ and the trace is negative. Thus the claimed formula for the signature follows. 
	
	Next, we solve the equation
	\begin{equation*}
	Q\mathbf{b}=\begin{pmatrix}
	r\\r\pm1
	\end{pmatrix}
	\end{equation*}
	for $\mathbf b$ and obtain
	\begin{equation*}\mathbf b=
	\begin{pmatrix}
	b_1\\b_2
	\end{pmatrix}=\begin{pmatrix}
	\mp 1 +\frac{n(r\pm1)-tnr+tn(r\pm1)}{t+n}\\
	-\frac{r\pm1\pm t}{t+n}
	\end{pmatrix}.
	\end{equation*}
	Plugging it in the formula from Lemma~\ref{lem:d3} we get the claimed values for the $\de_3$-invariants.
\end{proof}

We will mostly use this result in the case that the topological surgery coefficient is $\pm 1$ and thus the surgered manifold is an integer homology sphere. For easier readability we formulate these important cases in Corollaries~\ref{cor:+1first} and~\ref{cor:-1first} as independent results.

\begin{cor}\label{cor:+1first}
	Let $K$ in $(S^3,\xist)$ be a Legendrian knot with Thurston--Bennequin invariant $t$ and rotation number $r$. Then we can compute the possible $\de_3$-invariants of $K(1-t)$ as follows.
	
	If $t\geq2$, the $\de_3$-invariant of its contact structures take the values
	\begin{equation*}
	\de_3\big( K(1-t) \big)=\de_3\big( K_{t-2}(-1) \big)=\frac{1}{4}\left({r'}^2-5\right)\in2\Z+1,
	\end{equation*}
	where $r'$ denotes the rotation number $\rot(K_{t-2})$ of the stabilized knot. 
	
	If $t\leq0$, the $\de_3$-invariant of its contact structures take the values
	\begin{align*}
	\de_3\big( K(1-t) \big)=&\de_3\left( K(+1){\def\svgwidth{1,6ex}\,\,
\begingroup%
  \makeatletter%
  \providecommand\color[2][]{%
    \errmessage{(Inkscape) Color is used for the text in Inkscape, but the package 'color.sty' is not loaded}%
    \renewcommand\color[2][]{}%
  }%
  \providecommand\transparent[1]{%
    \errmessage{(Inkscape) Transparency is used (non-zero) for the text in Inkscape, but the package 'transparent.sty' is not loaded}%
    \renewcommand\transparent[1]{}%
  }%
  \providecommand\rotatebox[2]{#2}%
  \ifx\svgwidth\undefined%
    \setlength{\unitlength}{49.54732089bp}%
    \ifx\svgscale\undefined%
      \relax%
    \else%
      \setlength{\unitlength}{\unitlength * \real{\svgscale}}%
    \fi%
  \else%
    \setlength{\unitlength}{\svgwidth}%
  \fi%
  \global\let\svgwidth\undefined%
  \global\let\svgscale\undefined%
  \makeatother%
  \begin{picture}(1,1.00059637)%
    \put(0,0){\includegraphics[width=\unitlength,page=1]{PushOff.pdf}}%
  \end{picture}%
\endgroup%
\,\,} K_1\left(-\frac{1}{|t|}\right) \right)\\
	=&\frac{1}{4}\left( (t\pm r)^2-1\right)\in2\Z,
	\end{align*}
	where the upper line of signs in the equation corresponds to the positive stabilization of $K_1$ and the lower line of signs corresponds to the negative stabilization of $K_1$. 
\end{cor}

\begin{cor}\label{cor:-1first}
	Let $K$ in $(S^3,\xist)$ be a Legendrian knot with Thurston--Bennequin invariant $t$ and rotation number $r$. Then we can compute the $\de_3$-invariants of $K(-1-t)$ as follows.
	
	If $t\geq0$, the $\de_3$-invariant of its contact structures take the values
	\begin{equation*}
	\de_3\big( K(-1-t) \big)=\de_3\big( K_{t}(-1) \big)=\frac{1}{4}\left(1-{r'}^2\right)\in2\Z,
	\end{equation*}
	where $r'$ denotes the rotation number $\rot(K_{t})$ of the stabilized knot. 
	
	If $t\leq-2$, the $\de_3$-invariant of its contact structures take the values
	\begin{align*}
	\de_3\big( K(-1-t) \big)=&\de_3\left( K(+1){\def\svgwidth{1,6ex}\,\,
\begingroup%
  \makeatletter%
  \providecommand\color[2][]{%
    \errmessage{(Inkscape) Color is used for the text in Inkscape, but the package 'color.sty' is not loaded}%
    \renewcommand\color[2][]{}%
  }%
  \providecommand\transparent[1]{%
    \errmessage{(Inkscape) Transparency is used (non-zero) for the text in Inkscape, but the package 'transparent.sty' is not loaded}%
    \renewcommand\transparent[1]{}%
  }%
  \providecommand\rotatebox[2]{#2}%
  \ifx\svgwidth\undefined%
    \setlength{\unitlength}{49.54732089bp}%
    \ifx\svgscale\undefined%
      \relax%
    \else%
      \setlength{\unitlength}{\unitlength * \real{\svgscale}}%
    \fi%
  \else%
    \setlength{\unitlength}{\svgwidth}%
  \fi%
  \global\let\svgwidth\undefined%
  \global\let\svgscale\undefined%
  \makeatother%
  \begin{picture}(1,1.00059637)%
    \put(0,0){\includegraphics[width=\unitlength,page=1]{PushOff.pdf}}%
  \end{picture}%
\endgroup%
\,\,} K_1\left(-\frac{1}{|t|}\right) \right)\\
	=&\frac{1}{4}\left( 1-(t\pm r)^2\right)-(t\pm r)\in2\Z+1,
	\end{align*}
	where the upper line of signs in the equation corresponds to the positive stabilization of $K_1$ and the lower line of signs corresponds to the negative stabilization of $K_1$. 
\end{cor}

\begin{rem}
	We observe, that in Corollary~\ref{cor:+1first}, $\tb(K_{t-2})=2$ is even and thus $r'$ is odd, so ${r'}^2-5$ is always divisible by $4$ but not by $8$. In fact, the first formula in Corollary~\ref{cor:+1first} yields odd $\de_3$-invariants. Similarly, we see that $t\pm r$ is always odd, thus $(t\pm r)^2-1$ is divisible by $8$ and thus the second equation in Corollary~\ref{cor:+1first} takes values in the even integers.
	
	Similarly, we observe that in Corollary~\ref{cor:-1first}, $1-{r'}^2$ is always divisible by $8$ and that $t\pm r$ is always odd, thus $1-(t\pm r)^2$ is divisible by $8$ and we get odd $\de_3$-invariants in the second equation of Corollary~\ref{cor:-1first}.
\end{rem}

In particular, we get the following two key results completely characterizing the possible values of the $\de_3$-invariants obtained by contact surgery along different Legendrian realizations of the same smooth knot type. To formulate the results compactly we first introduce some notation.

Let $K$ in $S^3$ be a smooth, unoriented knot with two oriented Legendrian realizations $L^+_{max}$ and $L^-_{max}$ of $K$ (which are allowed to be isotopic) such that $\operatorname{TB}=\tb(L^+_{max})=\tb(L^-_{max})$ and $\operatorname{ROT}=\rot(L^+_{max})=-\rot(L^-_{max})\geq 0$ such that any other oriented Legendrian realization $L$ of $K$ has classical invariants $\tb(L)$ and $\rot(L)$ such that $\tb(L)$ and $\rot(L)$ are also obtained as classical invariants of a stabilization of $L^+_{max}$ or $L^-_{max}$. We define the integer
\begin{equation*}
T:=\frac{\operatorname{TB}+\operatorname{ROT}-1}{2}.
\end{equation*}

\begin{cor} \label{cor:+1surgery}
Let $K$ be a knot as described above and $M$ be the $3$-manifold obtained by topological $(+1)$-surgery along $K$. Then the possible values of the $\de_3$-invariants of contact structures on $M$ obtained by integer contact surgery (corresponding to topological $(+1)$-surgery) along Legendrian realizations of $K$ are
	\begin{equation*}
	m(m-1),
	\end{equation*}
	where $m$ is an integer such that $m\geq-T$ and if $\operatorname{TB}\geq2$ then the $\de_3$-invariant can be in addition
	\begin{equation*}
	(T-m)^2+(T-m)-1,
	\end{equation*}
	where $m=1,2,\ldots,\operatorname{TB}-1$.
\end{cor}

\begin{proof}
	Let $L$ be a Legendrian realization of $K$ in $(S^3,\xist)$ with Thurston--Bennequin invariant $t=\tb(L)$. Since $L$ has the same classical invariants as a stabilization of $L_{max}$ or $-L_{max}$ the possible values for the rotation number $r=\rot(L)$ are
	\begin{equation}
	r=\pm\operatorname{ROT}+\operatorname{TB} -\,  t-2k,\label{r:eq1}
	\end{equation}
	for $k=0,1,\ldots,\operatorname{TB}-t$. From this we conclude the possible values for $t\pm r$ as:
	\begin{equation}
	t\pm r=\operatorname{TB}\pm\operatorname{ROT} -\,2k\label{t+r:eq2},
	\end{equation}
	for $k=0,1,\ldots,\operatorname{TB}-t$.
			
	We now want to use Corollary~\ref{cor:+1first} to compute the possible $\de_3$-values. Therefore, we need to distinguish two cases. We start with the case that $\operatorname{TB}\geq t\geq2$. Then we know by Corollary~\ref{cor:+1first} that 
	\begin{equation*}
	\de_3\big(L(1-t)\big)=\frac{{r'}^2-5}{4},
	\end{equation*}
	where $r'$ is the rotation number of the stabilized knot $L_{t-2}$. Since $\tb(L_{t-2})=2$ we know from Equation~(\ref{r:eq1}) that 
	\begin{equation*}
	r'=\pm(\operatorname{TB}+\operatorname{ROT}-\,2-2k),
	\end{equation*}
	for $k=0,1,\ldots,\operatorname{TB}-2$. Setting $m=k+1$ and plugging everything into the equation of the $\de_3$-invariant yields the second set of possible values.
	
	In the case $t\leq0$, we have from Corollary~\ref{cor:+1first} that
	\begin{equation*}
	\de_3\big(L(1-t)\big)=\frac{(t\pm r)^2-1}{4}.
	\end{equation*}
	Plugging in Equation~(\ref{t+r:eq2}) yields
	\begin{equation*}
	\de_3\big(L(1-t)\big)=\frac{(\operatorname{TB}\pm\operatorname{ROT})^2-1}{4}-k (\operatorname{TB}\pm\operatorname{ROT}) + k^2.
	\end{equation*}
	Next, we write $\operatorname{TB}\pm\operatorname{ROT}=2T_\pm+1$ and thus get
	\begin{equation*}
    \de_3\big(L(1-t)\big)=(T_\pm-k)^2+(T_\pm-k).
    \end{equation*}	
    By setting $m=k-T_\pm$ we get $m^2-m$ for $m=-T_\pm,-T_\pm+1, \ldots, -T_\pm + \operatorname{TB}-\,t$. But since $T=T_+\geq T_-$ and $-t$ can be arbitrarily large we see that $m\geq-T$.
\end{proof}

The analogous result for topological $(-1)$-surgery is as follows.

\begin{cor}\label{cor:-1surgery}
	Let $M$ be the $3$-manifold obtained by topological $(-1)$-surgery  along $K$. Then the possible values of the $\de_3$-invariants of contact structures on $M$ obtained by integer contact surgery (corresponding to topological $(-1)$-surgery) along Legendrian realizations of $K$ are
	\begin{equation*}
	m(3-m)-1,
	\end{equation*}
	where $m$ is an integer such that $m\geq-T$ and if $\operatorname{TB}\geq0$ then the $\de_3$-invariant can be in addition
	\begin{equation*}
	-(m-T)^2+(m-T),
	\end{equation*}
	where $m=0,1,\ldots,\operatorname{TB}$.
\end{cor}

\begin{proof}
	The proof works analogously as in the proof of Corollary~\ref{cor:+1surgery}. In the case of $\operatorname{TB}\geq t\geq0$ we know by Corollary~\ref{cor:-1first} that 
	\begin{equation*}
	\de_3\big(L(1-t)\big)=\frac{1-{r'}^2}{4},
	\end{equation*}
	where $r'$ is the rotation number of the stabilized knot $L_{t}$ with $\tb(L_{t})=0$. Together with Equation~(\ref{r:eq1}) we get the second set of values for the $\de_3$-invariants.
	
	The first set of values for $\de_3$, we get by combining Corollary~\ref{cor:-1first} and Equation~(\ref{t+r:eq2}) in the case $t\leq-2$.
\end{proof}

\subsection{\texorpdfstring{$\de_3$}{d3}-invariants of rational surgeries}\label{sec:rational}
In this section, we will discuss some special cases of rational contact surgeries, which we will need later in Section~\ref{sec:S3-rational} to compute rational contact surgery numbers of contact structures on $S^3$. In this subsection, $K$ will always denote a topological knot in $S^3$ such that the range of classical invariants of its Legendrian realizations is the same  as for the unknot, i.e.\ $\tb\leq-1$ and $\rot\in\{\tb+1,\tb+3,\tb+5,\ldots,1-\tb\}$. As pointed out to us by the referee examples of such knots are given by Lagrangian slice knots.

The goal is to analyze the $\de_3$-invariants of contact surgeries along Legendrian realizations of $K$ corresponding to a topological $(1/q)$-surgery. For that we consider some Legendrian realization $L$ of $K$ with Thurston--Bennequin invariant $t\leq-1$, rotation number $r$ and some integer $q\in\Z\setminus\{0\}$. The goal is to analyze 
\begin{equation*}
L\left(\frac{1}{q}-t\right)=L\left(\frac{1-qt}{q}\right),
\end{equation*}
where we assume $(t,q)\neq(-1,-1)$. (Otherwise, the contact surgery coefficient would vanish.) First, we need the analog of Lemma~\ref{lem:surgerylemma} in this setting. 

\begin{lem}\label{lem:rationalSurgeryLemma}
	If $t=-1$ and $q\leq-2$ we have
	\begin{equation}
	L\left(\frac{1-qt}{q}\right)=L\left(\frac{1+q}{q}\right)=L\left(\frac{1}{2}\right){\def\svgwidth{1,6ex}\,\,
\begingroup%
  \makeatletter%
  \providecommand\color[2][]{%
    \errmessage{(Inkscape) Color is used for the text in Inkscape, but the package 'color.sty' is not loaded}%
    \renewcommand\color[2][]{}%
  }%
  \providecommand\transparent[1]{%
    \errmessage{(Inkscape) Transparency is used (non-zero) for the text in Inkscape, but the package 'transparent.sty' is not loaded}%
    \renewcommand\transparent[1]{}%
  }%
  \providecommand\rotatebox[2]{#2}%
  \ifx\svgwidth\undefined%
    \setlength{\unitlength}{49.54732089bp}%
    \ifx\svgscale\undefined%
      \relax%
    \else%
      \setlength{\unitlength}{\unitlength * \real{\svgscale}}%
    \fi%
  \else%
    \setlength{\unitlength}{\svgwidth}%
  \fi%
  \global\let\svgwidth\undefined%
  \global\let\svgscale\undefined%
  \makeatother%
  \begin{picture}(1,1.00059637)%
    \put(0,0){\includegraphics[width=\unitlength,page=1]{PushOff.pdf}}%
  \end{picture}%
\endgroup%
\,\,} L_1\left(-\frac{1}{-q-2}\right).\label{eq:rational1}
	\end{equation}
	If $t\leq-2$ and $q\leq-1$ we have
	\begin{equation}
	L\left(\frac{1-qt}{q}\right)=L\left(+1\right){\def\svgwidth{1,6ex}\,\,
\begingroup%
  \makeatletter%
  \providecommand\color[2][]{%
    \errmessage{(Inkscape) Color is used for the text in Inkscape, but the package 'color.sty' is not loaded}%
    \renewcommand\color[2][]{}%
  }%
  \providecommand\transparent[1]{%
    \errmessage{(Inkscape) Transparency is used (non-zero) for the text in Inkscape, but the package 'transparent.sty' is not loaded}%
    \renewcommand\transparent[1]{}%
  }%
  \providecommand\rotatebox[2]{#2}%
  \ifx\svgwidth\undefined%
    \setlength{\unitlength}{49.54732089bp}%
    \ifx\svgscale\undefined%
      \relax%
    \else%
      \setlength{\unitlength}{\unitlength * \real{\svgscale}}%
    \fi%
  \else%
    \setlength{\unitlength}{\svgwidth}%
  \fi%
  \global\let\svgwidth\undefined%
  \global\let\svgscale\undefined%
  \makeatother%
  \begin{picture}(1,1.00059637)%
    \put(0,0){\includegraphics[width=\unitlength,page=1]{PushOff.pdf}}%
  \end{picture}%
\endgroup%
\,\,} L_1\left(-\frac{1}{-t-2}\right){\def\svgwidth{1,6ex}\,\,
\begingroup%
  \makeatletter%
  \providecommand\color[2][]{%
    \errmessage{(Inkscape) Color is used for the text in Inkscape, but the package 'color.sty' is not loaded}%
    \renewcommand\color[2][]{}%
  }%
  \providecommand\transparent[1]{%
    \errmessage{(Inkscape) Transparency is used (non-zero) for the text in Inkscape, but the package 'transparent.sty' is not loaded}%
    \renewcommand\transparent[1]{}%
  }%
  \providecommand\rotatebox[2]{#2}%
  \ifx\svgwidth\undefined%
    \setlength{\unitlength}{49.54732089bp}%
    \ifx\svgscale\undefined%
      \relax%
    \else%
      \setlength{\unitlength}{\unitlength * \real{\svgscale}}%
    \fi%
  \else%
    \setlength{\unitlength}{\svgwidth}%
  \fi%
  \global\let\svgwidth\undefined%
  \global\let\svgscale\undefined%
  \makeatother%
  \begin{picture}(1,1.00059637)%
    \put(0,0){\includegraphics[width=\unitlength,page=1]{PushOff.pdf}}%
  \end{picture}%
\endgroup%
\,\,} L_{1,1}\left(-\frac{1}{-q-1}\right).\label{eq:rational2}
	\end{equation}
	If $t\leq-1$ and $q\geq1$ we have
	\begin{equation}
	L\left(\frac{1-qt}{q}\right)=L\left(+1\right){\def\svgwidth{1,6ex}\,\,
\begingroup%
  \makeatletter%
  \providecommand\color[2][]{%
    \errmessage{(Inkscape) Color is used for the text in Inkscape, but the package 'color.sty' is not loaded}%
    \renewcommand\color[2][]{}%
  }%
  \providecommand\transparent[1]{%
    \errmessage{(Inkscape) Transparency is used (non-zero) for the text in Inkscape, but the package 'transparent.sty' is not loaded}%
    \renewcommand\transparent[1]{}%
  }%
  \providecommand\rotatebox[2]{#2}%
  \ifx\svgwidth\undefined%
    \setlength{\unitlength}{49.54732089bp}%
    \ifx\svgscale\undefined%
      \relax%
    \else%
      \setlength{\unitlength}{\unitlength * \real{\svgscale}}%
    \fi%
  \else%
    \setlength{\unitlength}{\svgwidth}%
  \fi%
  \global\let\svgwidth\undefined%
  \global\let\svgscale\undefined%
  \makeatother%
  \begin{picture}(1,1.00059637)%
    \put(0,0){\includegraphics[width=\unitlength,page=1]{PushOff.pdf}}%
  \end{picture}%
\endgroup%
\,\,} L_1\left(-\frac{1}{-t-1}\right){\def\svgwidth{1,6ex}\,\,
\begingroup%
  \makeatletter%
  \providecommand\color[2][]{%
    \errmessage{(Inkscape) Color is used for the text in Inkscape, but the package 'color.sty' is not loaded}%
    \renewcommand\color[2][]{}%
  }%
  \providecommand\transparent[1]{%
    \errmessage{(Inkscape) Transparency is used (non-zero) for the text in Inkscape, but the package 'transparent.sty' is not loaded}%
    \renewcommand\transparent[1]{}%
  }%
  \providecommand\rotatebox[2]{#2}%
  \ifx\svgwidth\undefined%
    \setlength{\unitlength}{49.54732089bp}%
    \ifx\svgscale\undefined%
      \relax%
    \else%
      \setlength{\unitlength}{\unitlength * \real{\svgscale}}%
    \fi%
  \else%
    \setlength{\unitlength}{\svgwidth}%
  \fi%
  \global\let\svgwidth\undefined%
  \global\let\svgscale\undefined%
  \makeatother%
  \begin{picture}(1,1.00059637)%
    \put(0,0){\includegraphics[width=\unitlength,page=1]{PushOff.pdf}}%
  \end{picture}%
\endgroup%
\,\,} L_{1,q-1}\left(-1\right).\label{eq:rational3}
	\end{equation}
\end{lem}

\begin{proof}[Proof of Lemma~\ref{lem:rationalSurgeryLemma}]
	We start with the case $t=-1$ and $q\leq-2$, in which we have
	\begin{equation*}
	1<\frac{q}{1+q}\leq 2.
	\end{equation*}
	From Lemma~\ref{lem:algo} we obtain
	\begin{equation*}
	L\left(\frac{1+q}{q}\right)=L\left(\frac{1}{2}\right){\def\svgwidth{1,6ex}\,\,
\begingroup%
  \makeatletter%
  \providecommand\color[2][]{%
    \errmessage{(Inkscape) Color is used for the text in Inkscape, but the package 'color.sty' is not loaded}%
    \renewcommand\color[2][]{}%
  }%
  \providecommand\transparent[1]{%
    \errmessage{(Inkscape) Transparency is used (non-zero) for the text in Inkscape, but the package 'transparent.sty' is not loaded}%
    \renewcommand\transparent[1]{}%
  }%
  \providecommand\rotatebox[2]{#2}%
  \ifx\svgwidth\undefined%
    \setlength{\unitlength}{49.54732089bp}%
    \ifx\svgscale\undefined%
      \relax%
    \else%
      \setlength{\unitlength}{\unitlength * \real{\svgscale}}%
    \fi%
  \else%
    \setlength{\unitlength}{\svgwidth}%
  \fi%
  \global\let\svgwidth\undefined%
  \global\let\svgscale\undefined%
  \makeatother%
  \begin{picture}(1,1.00059637)%
    \put(0,0){\includegraphics[width=\unitlength,page=1]{PushOff.pdf}}%
  \end{picture}%
\endgroup%
\,\,} L\left(-\frac{1+q}{2+q}\right)=L\left(\frac{1}{2}\right){\def\svgwidth{1,6ex}\,\,
\begingroup%
  \makeatletter%
  \providecommand\color[2][]{%
    \errmessage{(Inkscape) Color is used for the text in Inkscape, but the package 'color.sty' is not loaded}%
    \renewcommand\color[2][]{}%
  }%
  \providecommand\transparent[1]{%
    \errmessage{(Inkscape) Transparency is used (non-zero) for the text in Inkscape, but the package 'transparent.sty' is not loaded}%
    \renewcommand\transparent[1]{}%
  }%
  \providecommand\rotatebox[2]{#2}%
  \ifx\svgwidth\undefined%
    \setlength{\unitlength}{49.54732089bp}%
    \ifx\svgscale\undefined%
      \relax%
    \else%
      \setlength{\unitlength}{\unitlength * \real{\svgscale}}%
    \fi%
  \else%
    \setlength{\unitlength}{\svgwidth}%
  \fi%
  \global\let\svgwidth\undefined%
  \global\let\svgscale\undefined%
  \makeatother%
  \begin{picture}(1,1.00059637)%
    \put(0,0){\includegraphics[width=\unitlength,page=1]{PushOff.pdf}}%
  \end{picture}%
\endgroup%
\,\,} L_1\left(-\frac{1}{-q-2}\right),
	\end{equation*}
	where the last equality is exactly Lemma~\ref{lem:surgerylemma}(4) (for $n=-(1+q)$).
	
	For the case $t\leq-2$ and $q\leq-1$ we observe
	\begin{equation*}
	0<\frac{q}{1-qt}\leq 1
	\end{equation*}
	and thus get from Lemma~\ref{lem:algo}
	\begin{equation*}
	L\left(\frac{1-qt}{q}\right)=L\left(+1\right){\def\svgwidth{1,6ex}\,\,
\begingroup%
  \makeatletter%
  \providecommand\color[2][]{%
    \errmessage{(Inkscape) Color is used for the text in Inkscape, but the package 'color.sty' is not loaded}%
    \renewcommand\color[2][]{}%
  }%
  \providecommand\transparent[1]{%
    \errmessage{(Inkscape) Transparency is used (non-zero) for the text in Inkscape, but the package 'transparent.sty' is not loaded}%
    \renewcommand\transparent[1]{}%
  }%
  \providecommand\rotatebox[2]{#2}%
  \ifx\svgwidth\undefined%
    \setlength{\unitlength}{49.54732089bp}%
    \ifx\svgscale\undefined%
      \relax%
    \else%
      \setlength{\unitlength}{\unitlength * \real{\svgscale}}%
    \fi%
  \else%
    \setlength{\unitlength}{\svgwidth}%
  \fi%
  \global\let\svgwidth\undefined%
  \global\let\svgscale\undefined%
  \makeatother%
  \begin{picture}(1,1.00059637)%
    \put(0,0){\includegraphics[width=\unitlength,page=1]{PushOff.pdf}}%
  \end{picture}%
\endgroup%
\,\,} L\left(-\frac{qt-1}{q(t+1)-1}\right).
	\end{equation*}
 Now an easy induction proves the continued fraction expansion of the last surgery coefficient to be 
	\begin{align*}
	-\frac{qt-1}{q(t+1)-1}=\big[\underbrace{-2,\ldots, -2}_{-t-2},-3,\underbrace{-2,\ldots,-2}_{-q-2}\big]
	\end{align*}
	and from Lemma~\ref{lem:algo} and Lemma~\ref{lem:replacemenet} the claimed result follows.
	
	Finally, in the case $t\leq-1$ and $q\geq1$ we have
	\begin{equation*}
	0<\frac{q}{1-qt}<1
	\end{equation*}
	and thus the following continuous fraction expansion proves the lemma:
	\begin{equation*}
	-\frac{qt-1}{q(t+1)-1}=\big[\underbrace{-2,\ldots, -2}_{-t-1},-q-1\big].
	\end{equation*}
	
\end{proof}

Next, we compute the possible values of the $\de_3$-invariant corresponding to the three cases of Lemma~\ref{lem:rationalSurgeryLemma}.

\begin{lem}
	For $t=-1$ and $q\leq-2$ we have
	\begin{equation*}
	\de_3\left(L\left(\frac{1+q}{q}\right)\right)=1.
	\end{equation*}
\end{lem}

\begin{proof}
	From Lemma~\ref{lem:rationalSurgeryLemma}(\ref{eq:rational1}) we know that
	\begin{equation*}
	L\left(\frac{1+q}{q}\right)=L\left(\frac{1}{2}\right){\def\svgwidth{1,6ex}\,\,
\begingroup%
  \makeatletter%
  \providecommand\color[2][]{%
    \errmessage{(Inkscape) Color is used for the text in Inkscape, but the package 'color.sty' is not loaded}%
    \renewcommand\color[2][]{}%
  }%
  \providecommand\transparent[1]{%
    \errmessage{(Inkscape) Transparency is used (non-zero) for the text in Inkscape, but the package 'transparent.sty' is not loaded}%
    \renewcommand\transparent[1]{}%
  }%
  \providecommand\rotatebox[2]{#2}%
  \ifx\svgwidth\undefined%
    \setlength{\unitlength}{49.54732089bp}%
    \ifx\svgscale\undefined%
      \relax%
    \else%
      \setlength{\unitlength}{\unitlength * \real{\svgscale}}%
    \fi%
  \else%
    \setlength{\unitlength}{\svgwidth}%
  \fi%
  \global\let\svgwidth\undefined%
  \global\let\svgscale\undefined%
  \makeatother%
  \begin{picture}(1,1.00059637)%
    \put(0,0){\includegraphics[width=\unitlength,page=1]{PushOff.pdf}}%
  \end{picture}%
\endgroup%
\,\,} L_1\left(-\frac{1}{-q-2}\right)
	\end{equation*}
	with a generalized linking matrix
	\begin{equation*}
	Q=\begin{pmatrix}
	-1&q+2\\
	-2&2q+3
	\end{pmatrix}
	\end{equation*}
	whose signature is $\sigma(Q)=-2$. Solving
	\begin{equation*}
	Q \begin{pmatrix}
	b_1\\
	b_2
	\end{pmatrix}=\begin{pmatrix}
	0\\
	\pm1
	\end{pmatrix}
	\end{equation*}
	yields $b_1=\mp(q+2)$ and $b_2=\mp1$ and plugging everything into the formula from Lemma~\ref{lem:d3} gives the claimed value.
\end{proof}

\begin{lem}\label{lem:tleq2}
	For $t\leq-2$ and $q\leq-1$ we have
	\begin{align*}
	\de_3\left(L\left(\frac{1-qt}{q}\right)\right)=& q\frac{(t\pm r)^2+3}{4}+q(t\pm r)+1 \,\text{ or }\\
	\de_3\left(L\left(\frac{1-qt}{q}\right)\right)=& q\frac{(t\pm r)^2-1}{4}-(t\pm r)
	\end{align*}
	depending on the choices of stabilizations in the expression of the surgery in terms of reciprocal integer contact surgeries from Lemma~\ref{lem:rationalSurgeryLemma}.
\end{lem}

\begin{proof}[Proof of Lemma~\ref{lem:tleq2}]
	From Lemma~\ref{lem:rationalSurgeryLemma}(\ref{eq:rational2}) we know that
	\begin{equation*}
	L\left(\frac{1-qt}{q}\right)=L\left(+1\right){\def\svgwidth{1,6ex}\,\,
\begingroup%
  \makeatletter%
  \providecommand\color[2][]{%
    \errmessage{(Inkscape) Color is used for the text in Inkscape, but the package 'color.sty' is not loaded}%
    \renewcommand\color[2][]{}%
  }%
  \providecommand\transparent[1]{%
    \errmessage{(Inkscape) Transparency is used (non-zero) for the text in Inkscape, but the package 'transparent.sty' is not loaded}%
    \renewcommand\transparent[1]{}%
  }%
  \providecommand\rotatebox[2]{#2}%
  \ifx\svgwidth\undefined%
    \setlength{\unitlength}{49.54732089bp}%
    \ifx\svgscale\undefined%
      \relax%
    \else%
      \setlength{\unitlength}{\unitlength * \real{\svgscale}}%
    \fi%
  \else%
    \setlength{\unitlength}{\svgwidth}%
  \fi%
  \global\let\svgwidth\undefined%
  \global\let\svgscale\undefined%
  \makeatother%
  \begin{picture}(1,1.00059637)%
    \put(0,0){\includegraphics[width=\unitlength,page=1]{PushOff.pdf}}%
  \end{picture}%
\endgroup%
\,\,} L_1\left(-\frac{1}{-t-2}\right){\def\svgwidth{1,6ex}\,\,
\begingroup%
  \makeatletter%
  \providecommand\color[2][]{%
    \errmessage{(Inkscape) Color is used for the text in Inkscape, but the package 'color.sty' is not loaded}%
    \renewcommand\color[2][]{}%
  }%
  \providecommand\transparent[1]{%
    \errmessage{(Inkscape) Transparency is used (non-zero) for the text in Inkscape, but the package 'transparent.sty' is not loaded}%
    \renewcommand\transparent[1]{}%
  }%
  \providecommand\rotatebox[2]{#2}%
  \ifx\svgwidth\undefined%
    \setlength{\unitlength}{49.54732089bp}%
    \ifx\svgscale\undefined%
      \relax%
    \else%
      \setlength{\unitlength}{\unitlength * \real{\svgscale}}%
    \fi%
  \else%
    \setlength{\unitlength}{\svgwidth}%
  \fi%
  \global\let\svgwidth\undefined%
  \global\let\svgscale\undefined%
  \makeatother%
  \begin{picture}(1,1.00059637)%
    \put(0,0){\includegraphics[width=\unitlength,page=1]{PushOff.pdf}}%
  \end{picture}%
\endgroup%
\,\,} L_{1,1}\left(-\frac{1}{-q-1}\right)
	\end{equation*}
	with a generalized linking matrix
	\begin{equation*}
	Q=\begin{pmatrix}
	1+t&-t^2-2t&-qt-t\\
	t&1-t^2-t&-qt-t+q+1\\
	t&-t^2-t+2&1-qt-t+2q
	\end{pmatrix}
	\end{equation*}
	whose signature is $\sigma(Q)=-3$. Solving
	\begin{equation*}
	Q \mathbf b=\mathbf r,
	\end{equation*}
with $\mathbf r$ the vector of rotation numbers of the surgery curves, 
	yields 
	\begin{align*}
	b_1&=r_1+tr_3-tr_1+tqr_3-2qtr_2+t^2qr_1-t^2qr_2   \\
	b_2&=r_3-r_2+qr_3-2qr_2+tqr_1-tqr_2   \\
	b_3&=2r_2-r_3+tr_2-tr_1.   
	\end{align*}
	Now we know that $(r_1,r_2,r_3)=(r,r\pm1,r)$ or $(r_1,r_2,r_3)=(r,r\pm1,r\pm2)$. By plugging in $(r_1,r_2,r_3)=(r,r\pm1,r)$ into the formula from Lemma~\ref{lem:d3} we get the first claimed equation for the $\de_3$-invariant and by plugging in $(r_1,r_2,r_3)=(r,r\pm1,r\pm2)$ we get the second equation.
\end{proof}

\begin{lem}\label{lem:tleq1}
	For $t\leq-1$ and $q\geq1$ we have
	\begin{equation*}
	\de_3\left(L\left(\frac{1-qt}{q}\right)\right)=q\frac{(t\pm r)^2-1}{4}+(t\pm r +1)z,
	\end{equation*}
	where $z=0,1,2,\ldots, q-1$, is determined by $z=(q\pm r\mp r')/2$ with $r'$ denoting the rotation number of $L_{1,q-1}$ from Lemma~\ref{lem:rationalSurgeryLemma}.
\end{lem}

\begin{rem}
	We remark that for $q=\pm1$ we recover the formulas from Corollaries~\ref{cor:+1first} and~\ref{cor:-1first}.
\end{rem}

\begin{proof} [Proof of Lemma~\ref{lem:tleq1}]
	From Lemma~\ref{lem:rationalSurgeryLemma}(\ref{eq:rational3}) we know that
	\begin{equation*}
	L\left(\frac{1-qt}{q}\right)=L\left(+1\right){\def\svgwidth{1,6ex}\,\,
\begingroup%
  \makeatletter%
  \providecommand\color[2][]{%
    \errmessage{(Inkscape) Color is used for the text in Inkscape, but the package 'color.sty' is not loaded}%
    \renewcommand\color[2][]{}%
  }%
  \providecommand\transparent[1]{%
    \errmessage{(Inkscape) Transparency is used (non-zero) for the text in Inkscape, but the package 'transparent.sty' is not loaded}%
    \renewcommand\transparent[1]{}%
  }%
  \providecommand\rotatebox[2]{#2}%
  \ifx\svgwidth\undefined%
    \setlength{\unitlength}{49.54732089bp}%
    \ifx\svgscale\undefined%
      \relax%
    \else%
      \setlength{\unitlength}{\unitlength * \real{\svgscale}}%
    \fi%
  \else%
    \setlength{\unitlength}{\svgwidth}%
  \fi%
  \global\let\svgwidth\undefined%
  \global\let\svgscale\undefined%
  \makeatother%
  \begin{picture}(1,1.00059637)%
    \put(0,0){\includegraphics[width=\unitlength,page=1]{PushOff.pdf}}%
  \end{picture}%
\endgroup%
\,\,} L_1\left(-\frac{1}{-t-1}\right){\def\svgwidth{1,6ex}\,\,
\begingroup%
  \makeatletter%
  \providecommand\color[2][]{%
    \errmessage{(Inkscape) Color is used for the text in Inkscape, but the package 'color.sty' is not loaded}%
    \renewcommand\color[2][]{}%
  }%
  \providecommand\transparent[1]{%
    \errmessage{(Inkscape) Transparency is used (non-zero) for the text in Inkscape, but the package 'transparent.sty' is not loaded}%
    \renewcommand\transparent[1]{}%
  }%
  \providecommand\rotatebox[2]{#2}%
  \ifx\svgwidth\undefined%
    \setlength{\unitlength}{49.54732089bp}%
    \ifx\svgscale\undefined%
      \relax%
    \else%
      \setlength{\unitlength}{\unitlength * \real{\svgscale}}%
    \fi%
  \else%
    \setlength{\unitlength}{\svgwidth}%
  \fi%
  \global\let\svgwidth\undefined%
  \global\let\svgscale\undefined%
  \makeatother%
  \begin{picture}(1,1.00059637)%
    \put(0,0){\includegraphics[width=\unitlength,page=1]{PushOff.pdf}}%
  \end{picture}%
\endgroup%
\,\,} L_{1,q-1}\left(-1\right)
	\end{equation*}
	with a generalized linking matrix
	\begin{equation*}
	Q=\begin{pmatrix}
	1+t&-t^2-t&t\\
	t&-t^2&t-1\\
	t&1-t^2&t-q-1
	\end{pmatrix}
	\end{equation*}
	whose signature is $\sigma(Q)=-1$. Solving
	\begin{equation*}
	Q \mathbf b=\mathbf r,
	\end{equation*}
with $\mathbf r$ the vector of rotation numbers of the surgery curves, 
	yields 
	\begin{align*}
	b_1&=r_1+tr_3-tr_1-qtr_2+qt^2r_1-qt^2r_2\\
	b_2&=r_3-r_2-qr_2+qtr_1-qtr_2\\
	b_3&=-r_2+tr_1-tr_2.   
	\end{align*}
	Now we know that $r_1=r$, $r_2=r\pm1$ and $r_3=r'$ is the rotation number of $L_{1,q-1}$. By plugging everything into the formula from Lemma~\ref{lem:d3} we get the claimed equation for the $\de_3$-invariant.
\end{proof}

\subsection{The Euler class of integer surgeries}\label{sec:EulerZ}
If $M$ has nontrivial homology we get another invariant of tangential $2$-plane fields, namely its underlying $spin^c$ structure. The set of $spin^c$ structures on $M$ is in one-to-one correspondence with the elements of $H^2(M)=H_1(M)$ (although this correspondence is not natural). If there is no $2$-torsion in the first homology, two $spin^c$ structures are equal if and only if their Euler classes agree. (Recall, that the Euler class of an oriented $2$-plane field is always an even class.) Moreover, two tangential $2$-plane fields with the same underlying $spin^c$ structure only differ by a connected sum with some, in general non-unique, $(S^3,\xi_n)$ for appropriate $n\in\Z$. For more details, we refer to~\cite{DiGeSt04, Go98}.

Our first result is a formula for the Euler class for a single positive integer contact surgery. Results for surgeries along links with arbitrary integers can easily be derived from this.

\begin{prop}\label{prop:Euler}
	Let $L$ be a Legendrian knot in $(S^3,\xist)$ and $n\in \N$, with Thurston--Bennequin invariant $t$ and rotation number $r$. Then we can compute the Euler class of $L(n)=(M,\xi)$ as
	\begin{equation*}
	\PD\big(\e(\xi)\big)=\big(\pm(n-1)-r\big)\mu\in H_1(M)=\big\langle \mu \vert -(t+n)\mu\big\rangle\cong\Z_{t+n},
	\end{equation*}
	where $\mu$ denotes the meridian of $L$ and the first row of signs correspond to a positive stabilization in $L(n)$ and the second to a negative stabilization.
\end{prop}

\begin{proof}
	From Lemma~\ref{lem:surgerylemma} we can write 
	\begin{equation*}
	(M,\xi)=L(n)=L(1){\def\svgwidth{1,6ex}\,\,
\begingroup%
  \makeatletter%
  \providecommand\color[2][]{%
    \errmessage{(Inkscape) Color is used for the text in Inkscape, but the package 'color.sty' is not loaded}%
    \renewcommand\color[2][]{}%
  }%
  \providecommand\transparent[1]{%
    \errmessage{(Inkscape) Transparency is used (non-zero) for the text in Inkscape, but the package 'transparent.sty' is not loaded}%
    \renewcommand\transparent[1]{}%
  }%
  \providecommand\rotatebox[2]{#2}%
  \ifx\svgwidth\undefined%
    \setlength{\unitlength}{49.54732089bp}%
    \ifx\svgscale\undefined%
      \relax%
    \else%
      \setlength{\unitlength}{\unitlength * \real{\svgscale}}%
    \fi%
  \else%
    \setlength{\unitlength}{\svgwidth}%
  \fi%
  \global\let\svgwidth\undefined%
  \global\let\svgscale\undefined%
  \makeatother%
  \begin{picture}(1,1.00059637)%
    \put(0,0){\includegraphics[width=\unitlength,page=1]{PushOff.pdf}}%
  \end{picture}%
\endgroup%
\,\,} L_1\left(-\frac{1}{n-1}\right)
	\end{equation*}
	and we get the generalized linking matrix as
	\begin{equation*}
	Q=\begin{pmatrix}
	1+t&t(n-1)\\
	t&tn-t-n\\
	\end{pmatrix}.
	\end{equation*}
	The first homology group of $M$ is generated by the meridians of the surgery curves and the relations are given by $Q$. 
	From Lemma~\ref{lem:d3} we obtain the claimed description of the Euler class.
\end{proof}

As a corollary, we see how the Euler class changes under stabilizing the knot $L$ and adjusting the contact surgery coefficient appropriately so that the underlying smooth manifold is still the same.

\begin{cor}\label{cor:Euler}
	In the same notation as in Proposition~\ref{prop:Euler}, we consider the contact surgery diagram given by
	\begin{equation*}
	L_s(n+s)=L_s(+1){\def\svgwidth{1,6ex}\,\,
\begingroup%
  \makeatletter%
  \providecommand\color[2][]{%
    \errmessage{(Inkscape) Color is used for the text in Inkscape, but the package 'color.sty' is not loaded}%
    \renewcommand\color[2][]{}%
  }%
  \providecommand\transparent[1]{%
    \errmessage{(Inkscape) Transparency is used (non-zero) for the text in Inkscape, but the package 'transparent.sty' is not loaded}%
    \renewcommand\transparent[1]{}%
  }%
  \providecommand\rotatebox[2]{#2}%
  \ifx\svgwidth\undefined%
    \setlength{\unitlength}{49.54732089bp}%
    \ifx\svgscale\undefined%
      \relax%
    \else%
      \setlength{\unitlength}{\unitlength * \real{\svgscale}}%
    \fi%
  \else%
    \setlength{\unitlength}{\svgwidth}%
  \fi%
  \global\let\svgwidth\undefined%
  \global\let\svgscale\undefined%
  \makeatother%
  \begin{picture}(1,1.00059637)%
    \put(0,0){\includegraphics[width=\unitlength,page=1]{PushOff.pdf}}%
  \end{picture}%
\endgroup%
\,\,} L_{s,1}\left(-\frac{1}{n+s-1}\right),
	\end{equation*}
	where the first $s$ stabilizations are all performed with the same sign, but the last stabilization of $L_{s,1}$ has the opposite sign. Then the above surgery diagram presents again a contact structure $\xi'$ on $M$, whose Euler classes are related by
	\begin{equation*}
	\PD(\e(\xi'))=\PD(\e(\xi))\pm2s\mu\in H_1(M)=\langle \mu \vert -(t+n)\mu\rangle\cong\Z_{t+n},
	\end{equation*}
where $\mu$ denotes the meridian of $L_s$. 
	It follows that we can obtain contact structures on $M$ via integer contact surgery along a stabilization of $L$ realizing any possible Euler class of $2$-plane fields on $M$.
\end{cor}

If we have more than one surgery curve, we perform the same construction along every surgery curve and obtain integer contact surgery diagrams of contact structures of any possible Euler class along a stabilization of that Legendrian link.

\subsection{Gompf's $\Gamma$-invariant}

In the case that a $3$-manifold has no $2$-torsion in its first homology, we know that the Euler class of a contact structure $\xi$ on $M$ uniquely determines the underlying $spin^c$ structure. However, in the presence of $2$-torsion this is not true anymore. On the other hand, Gompf defined a refined invariant, the \textbf{$\Gamma$-invariant} of a contact structure, which resolves that ambiguity. In some sense, $\Gamma$ can be thought of as a half Euler class of the contact structure. It is defined as follows. Let $(M,\xi)$ be a contact manifold and $V$ be a vector field in $\xi$ such that the zero set of $V$ is given by $2\gamma$ for a $1$-cycle $\gamma$. We represent $\gamma$ by a link $L$ in $M$. On $M\setminus L$ we see that $V$ is a non-vanishing section of $\xi$ and thus $V$ together with a vector field in $\xi$ orthogonal to $V$ and a vector field orthogonal to $\xi$ defines a trivialization of $M\setminus L$. Since $V$ vanishes with even multiplicity on $L$, the $spin$ structure extends uniquely to a $spin$ structure $\mathfrak s$ on all of $M$.
We now define $\Gamma(\xi,\mathfrak s)$ to be
\begin{equation*}
\Gamma(\xi,\mathfrak s)=[\gamma]\in H_1(M).
\end{equation*} 

In~\cite{Go98} it is shown that $\Gamma(\xi,\mathfrak s)$ is well-defined and only depends on $\xi$ and $\mathfrak s$. Furthermore, it behaves naturally under connected sums and coverings, classifies tangential $2$-plane fields on $M\setminus \{pt\}$ and it satisfies $2\Gamma(\xi,\mathfrak s)=\operatorname{PD}(\e(\xi))$ for any $\mathfrak s$.

Here, we want to generalize Gompf's formula for computing $\Gamma(\xi,\mathfrak s)$ from contact $(-1)$-surgery diagrams to general contact $(\pm1)$-surgery diagrams. For that, we first recall how to present $spin$ structures in surgery diagrams~\cite{GoSt99}. Let $L=L_1\sqcup\ldots\sqcup L_k$ be a topological integer surgery diagram of a $3$-manifold $M$ (in particular, the framings of $L_i$ are measured with respect to the Seifert framing). We write $\lk(L_i,L_i)$ for the framing of $L_i$. A sublink $(L_j)_{j\in J}$ for some subset $J\subset \{1,2,\ldots,k\}$ is called \textbf{characteristic sublink} if for any component $L_i$ of $L$ we have
\begin{equation*}
\lk(L_i,L_i)\equiv \sum_{j\in J} \lk(L_i,L_j) \,\, \pmod2.
\end{equation*} 
The set of characteristic sublinks of $L$ is in bijection with the set of $spin$ structures of $M$ which is in (non-natural) bijection to $H_1(M;\Z_2)$. Thus we can describe a given $spin$ structure of $M$ via a characteristic sublink of $L$. 

\begin{lem}\label{lem:gamma}
	Let $L = L_1 \sqcup \ldots \sqcup L_k$ be an oriented Legendrian link in
	$(S^3, \xist)$ and let $(M, \xi)$ be the contact manifold obtained
	from $S^3$ by contact $(\pm {1})$-surgeries along $L$. 
	We write $r_i$ for the rotation number of $L_i$, we denote the topological surgery coefficient of $L_i$ by $\lk(L_i,L_i)$, and write $Q$ for the linking matrix. Moreover, we describe a $spin$ structure $\mathfrak s$ of $M$ via a characteristic sublink $(L_j)_{j\in J}$ of $L$. Then
	\begin{equation*}
	\Gamma(\xi,\mathfrak s)=\frac{1}{2}\left(  \sum_{i=1}^k r_i\mu_i +\sum_{j\in J} (Q\boldsymbol{\mu})_j  \right),
	\end{equation*}
	where $(Q\boldsymbol{\mu})_j $ denotes the $j$-th entry of $Q\cdot(\mu_1,\ldots,\mu_k)^t$.
\end{lem}

Since $Q\boldsymbol{\mu} =0$ in $H_1(M)$ we see directly that $2\Gamma(\xi,\mathfrak s)$ is Poincar\'e dual to the Euler class.

\begin{proof}
	If all contact surgery coefficients of $L$ are $(-1)$ Theorem~4.12 in~\cite{Go98} implies that
	\begin{equation*}
\Gamma(\xi,\mathfrak s)=\sum_{i=1}^k \frac{1}{2}\left(  r_i+\sum_{j\in J}  \lk(L_i,L_j) \right)\mu_i,
\end{equation*}	
	but since $$Q\boldsymbol{\mu}=\left(\sum_{i=1}^k \lk(L_i,L_j)\mu_i\right)_{1\leq j\leq k}$$ this directly implies the formula for contact $(-1)$-surgeries as stated in the lemma. 
	
	Now we denote by $(M_-,\xi_-)$ the contact manifold obtained by performing contact $(-1)$-surgery on all components of $L$ and we write $\mathfrak{s}_-$ for the spin structure on $M_-$ given by the characteristic sublink $L'=(L_j)_{j\in J}$. We write $(M_+,\xi_+)$ for the contact manifold obtained by performing contact surgery along $L$ where some of the surgery coefficients are changed to $(+1)$. Since the parity of the underlying topological surgery coefficients does not change under replacing a $(-1)$ coefficient with a $(+1)$ coefficient it follows that the characteristic sublinks of the two surgery diagrams are the same. In particular, $L'$ is again a characteristic sublink and induces a spin structure $\mathfrak{s}_+$ on $M_+$.
	
	 The spin structure $\mathfrak s_-$ can be represented by a vector field $V_-$ in $\xi_-$ that is vanishing with multiplicity two along a link $K_-$ in $M_-$ given by the surgery dual of $L'$. We may also think of $\mathfrak s_-$ as the characteristic sublink $L'$ of $L$. 
	 
	 The correspondence is given as follows. We denote by $X$ the $4$-manifold with $\partial X=M_-$ given by attaching $4$-dimensional $2$-handles to $D^4$ along the framed surgery link $L$. We denote by $F'$ the closed surface in $X$ given by taking a Seifert surface of the characteristic sublink $L'$ and capping it off in $X$ with the core disks of the handles attached along $L'$. Then the trivialization $\partial_x$ of $\xist$ induces a reference spin structure $\mathfrak{s}_0$ on $X\setminus F'$. (Here we use the notation from the proof of Theorem~4.12 in~\cite{Go98}. So we are thinking of $S^3$ as $\R^3$ with a point at infinity, $\xist$ is the standard contact structure on $\R^3$, and $\partial_x$ is the coordinate vector field in the $x$-direction.) Now we consider $X\cup (M_-\times [0,1])$ and put the spin structure $\mathfrak{s}_0$ on $X\setminus F'$ and the spin structure $\mathfrak{s}_-$ on $M_-\times \{1\}$. In the proof of Theorem~4.12 in~\cite{Go98}, Gompf constructs an integer obstruction to extending this spin structure over all of $X\cup (M_-\times [0,1])$ that is a $2$-chain $z_1$ given as a linear combination of the $2$-handles in $X$ that is a cycle with $\Z_2$-coefficients. Moreover, it is shown that $z_1$ induces a $\Z_2$-homology class in $M_-$ that is Poincar\'e dual to the difference class $\Delta(\mathfrak{s}_-,\mathfrak{s}_0)$. From that it is deduced in~\cite{Go98} that $z_1$ corresponds after inclusion to the homology class $[F']$ in $H_2(X;\Z_2)$ induced by the characteristic sublink $L'$. Since $z_1$ is a cycle with $\Z_2$-coefficients, it follows that $\partial z_1=2\gamma$ for some $1$-cycle $\gamma$ that then is homologous to the surgery dual $K_-$ of the characteristic sublink $L'$. Then we can define a vector field $v_-$ in $\xi_-$ that is vanishing on $K_-$ of order two by
	 \begin{equation*}
	 	\operatorname{PD}\big(\Delta(V_-,\partial_x)\big)=[z_1]\in H_2(M_-,K_-;\Z).
	 \end{equation*}
	 In conclusion, we can choose a vector field $V_{st}$ in $\xist$ that vanishes with multiplicity $2$ on the characteristic sublink, such that on the complement of $L'$, the vector fields $V_-$ and $V_{st}$ agree (and are non-vanishing) and the link $K_-$ is given by the surgery duals of the characteristic sublink $L'$. The homology class of $K_-$ in $H_1(M_-)$ represents $\Gamma(M_-,\xi_-,\mathfrak{s}_-)$.
	 
	 Similarly, $V_{st}$ induces a vector field $V_+$ in $\xi_+$ on $M_+$ that vanishes with multiplicity two in the surgery duals $K_+$ of the characteristic sublink $L'$. The homology class of $K_+$ in $H_1(M_+)$ represents $\Gamma(M_+,\xi_+,\mathfrak{s}_+)$.
	
	Now we represent $K_\pm$ in the surgery diagrams as a Legendrian link in the exterior $(S^3\setminus\mathring{\nu L},\xist)$. By writing down a concrete model of the gluing map under contact $(\pm1)$-surgery we see that the surgery dual knot is isotopic to the contact longitude for both contact surgery coefficients. It follows that $K_\pm$ are both isotopic to the Legendrian push-offs of the characteristic sublink $L'$. Here the Legendrian push-offs and the contact longitudes inherit an orientation from the orientation of $L$. Then we can choose the orientations on the surgery dual curves such that the above isotopies respect the orientations. In particular the knots $K_+$ and $K_-$ represent the same homology class in $H_1(S^3\setminus\mathring{\nu L})$. Thus the formula for computing the $\Gamma$-invariant does not change by changing some of the surgery coefficients from $(-1)$ to $(+1)$. (But of course the linking matrix $Q$ describing the first homology of the surgered manifold changes.)
\end{proof}

As done in~\cite{DuKe16} for the Euler class one can similarly deduce from this lemma a general formula for $\Gamma$ of contact $(1/n)$-surgeries. But since we will not need this here, we describe instead an example how to compute $\Gamma$ directly for a single contact $(1/n)$-surgery along a single Legendrian knot $K$. 

\begin{ex}
	Let $K$ be a Legendrian knot in $(S^3,\xist)$ with Thurston--Bennequin invariant $t$ and rotation number $r$. Let $n$ be a positive integer. Then\begin{equation*}
	K(\pm1/n)=\underbrace{K(\pm1){\def\svgwidth{1,6ex}\,\,
\begingroup%
  \makeatletter%
  \providecommand\color[2][]{%
    \errmessage{(Inkscape) Color is used for the text in Inkscape, but the package 'color.sty' is not loaded}%
    \renewcommand\color[2][]{}%
  }%
  \providecommand\transparent[1]{%
    \errmessage{(Inkscape) Transparency is used (non-zero) for the text in Inkscape, but the package 'transparent.sty' is not loaded}%
    \renewcommand\transparent[1]{}%
  }%
  \providecommand\rotatebox[2]{#2}%
  \ifx\svgwidth\undefined%
    \setlength{\unitlength}{49.54732089bp}%
    \ifx\svgscale\undefined%
      \relax%
    \else%
      \setlength{\unitlength}{\unitlength * \real{\svgscale}}%
    \fi%
  \else%
    \setlength{\unitlength}{\svgwidth}%
  \fi%
  \global\let\svgwidth\undefined%
  \global\let\svgscale\undefined%
  \makeatother%
  \begin{picture}(1,1.00059637)%
    \put(0,0){\includegraphics[width=\unitlength,page=1]{PushOff.pdf}}%
  \end{picture}%
\endgroup%
\,\,}\ldots{\def\svgwidth{1,6ex}\,\,
\begingroup%
  \makeatletter%
  \providecommand\color[2][]{%
    \errmessage{(Inkscape) Color is used for the text in Inkscape, but the package 'color.sty' is not loaded}%
    \renewcommand\color[2][]{}%
  }%
  \providecommand\transparent[1]{%
    \errmessage{(Inkscape) Transparency is used (non-zero) for the text in Inkscape, but the package 'transparent.sty' is not loaded}%
    \renewcommand\transparent[1]{}%
  }%
  \providecommand\rotatebox[2]{#2}%
  \ifx\svgwidth\undefined%
    \setlength{\unitlength}{49.54732089bp}%
    \ifx\svgscale\undefined%
      \relax%
    \else%
      \setlength{\unitlength}{\unitlength * \real{\svgscale}}%
    \fi%
  \else%
    \setlength{\unitlength}{\svgwidth}%
  \fi%
  \global\let\svgwidth\undefined%
  \global\let\svgscale\undefined%
  \makeatother%
  \begin{picture}(1,1.00059637)%
    \put(0,0){\includegraphics[width=\unitlength,page=1]{PushOff.pdf}}%
  \end{picture}%
\endgroup%
\,\,} K(\pm1)}_{n}
	\end{equation*}
	and the first homology is generated by $\mu:=\mu_1=\cdots=\mu_n$ with the relation $(nt\pm1)\mu=0$. Since the $spin$ structures are in bijection with the first homology group with $\Z_2$-coefficients, we see that there is a unique $spin$ structure if $nt\pm1$ is odd and that there exists exactly two spin structures if $nt\pm1$ is even. We observe that in both cases the empty sublink is characteristic defining a $spin$ structure $\mathfrak s_0$. In the case that $nt\pm1$ is even the whole surgery link is also characteristic, defining a $spin$ structure $\mathfrak s_1$. Then Lemma~\ref{lem:gamma} readily implies that 
	\begin{equation*}
	\Gamma\big(K(\pm1/n),\mathfrak s_0\big)=\frac{nr}{2}\mu
	\end{equation*}
	and if $nt\pm1$ is even we have in addition
	\begin{equation*}
	\Gamma\big(K(\pm1/n),\mathfrak s_1\big)=n\frac{r+nt\pm1}{2}\mu.
	\end{equation*}
\end{ex}

Next, we study the case of general positive integer contact surgeries. The following two results generalize the results from Proposition~\ref{prop:Euler} and Corollary~\ref{cor:Euler} for the Euler class to Gompf's $\Gamma$-invariant. Again, similar results hold for positive integer surgeries along links with more components.

\begin{prop}\label{prop:Gammaint}
	Let $L$ be a Legendrian knot in $(S^3,\xist)$ and $n$ a positive integer. Then $L(n)$ carries a contact structure $\xi$ given by the contact $(\pm1)$-surgery description
	\begin{equation}
	L(+1){\def\svgwidth{1,6ex}\,\,
\begingroup%
  \makeatletter%
  \providecommand\color[2][]{%
    \errmessage{(Inkscape) Color is used for the text in Inkscape, but the package 'color.sty' is not loaded}%
    \renewcommand\color[2][]{}%
  }%
  \providecommand\transparent[1]{%
    \errmessage{(Inkscape) Transparency is used (non-zero) for the text in Inkscape, but the package 'transparent.sty' is not loaded}%
    \renewcommand\transparent[1]{}%
  }%
  \providecommand\rotatebox[2]{#2}%
  \ifx\svgwidth\undefined%
    \setlength{\unitlength}{49.54732089bp}%
    \ifx\svgscale\undefined%
      \relax%
    \else%
      \setlength{\unitlength}{\unitlength * \real{\svgscale}}%
    \fi%
  \else%
    \setlength{\unitlength}{\svgwidth}%
  \fi%
  \global\let\svgwidth\undefined%
  \global\let\svgscale\undefined%
  \makeatother%
  \begin{picture}(1,1.00059637)%
    \put(0,0){\includegraphics[width=\unitlength,page=1]{PushOff.pdf}}%
  \end{picture}%
\endgroup%
\,\,} \underbrace{L_1(-1){\def\svgwidth{1,6ex}\,\,
\begingroup%
  \makeatletter%
  \providecommand\color[2][]{%
    \errmessage{(Inkscape) Color is used for the text in Inkscape, but the package 'color.sty' is not loaded}%
    \renewcommand\color[2][]{}%
  }%
  \providecommand\transparent[1]{%
    \errmessage{(Inkscape) Transparency is used (non-zero) for the text in Inkscape, but the package 'transparent.sty' is not loaded}%
    \renewcommand\transparent[1]{}%
  }%
  \providecommand\rotatebox[2]{#2}%
  \ifx\svgwidth\undefined%
    \setlength{\unitlength}{49.54732089bp}%
    \ifx\svgscale\undefined%
      \relax%
    \else%
      \setlength{\unitlength}{\unitlength * \real{\svgscale}}%
    \fi%
  \else%
    \setlength{\unitlength}{\svgwidth}%
  \fi%
  \global\let\svgwidth\undefined%
  \global\let\svgscale\undefined%
  \makeatother%
  \begin{picture}(1,1.00059637)%
    \put(0,0){\includegraphics[width=\unitlength,page=1]{PushOff.pdf}}%
  \end{picture}%
\endgroup%
\,\,}\cdots{\def\svgwidth{1,6ex}\,\,
\begingroup%
  \makeatletter%
  \providecommand\color[2][]{%
    \errmessage{(Inkscape) Color is used for the text in Inkscape, but the package 'color.sty' is not loaded}%
    \renewcommand\color[2][]{}%
  }%
  \providecommand\transparent[1]{%
    \errmessage{(Inkscape) Transparency is used (non-zero) for the text in Inkscape, but the package 'transparent.sty' is not loaded}%
    \renewcommand\transparent[1]{}%
  }%
  \providecommand\rotatebox[2]{#2}%
  \ifx\svgwidth\undefined%
    \setlength{\unitlength}{49.54732089bp}%
    \ifx\svgscale\undefined%
      \relax%
    \else%
      \setlength{\unitlength}{\unitlength * \real{\svgscale}}%
    \fi%
  \else%
    \setlength{\unitlength}{\svgwidth}%
  \fi%
  \global\let\svgwidth\undefined%
  \global\let\svgscale\undefined%
  \makeatother%
  \begin{picture}(1,1.00059637)%
    \put(0,0){\includegraphics[width=\unitlength,page=1]{PushOff.pdf}}%
  \end{picture}%
\endgroup%
\,\,} L_1(-1)}_{n-1}.\label{eq:description}
	\end{equation}
(Recall, that $L_1$ denotes a push-off of $L$ that is stabilized once, with unspecified but fixed choices of stabilization.)
	The first copy of $L$ in the above surgery description is a characteristic sublink and thus defines a $spin$ structure $\mathfrak s$. Then
	\begin{equation*}
	\Gamma\big(L(n),\xi,\mathfrak s\big)=\left(-\frac{t+r\pm1}{2}+\frac{\pm1-1}{2}n\right)\mu.
	\end{equation*}
\end{prop}

\begin{proof}
 From the surgery description~(\ref{eq:description}) we compute the linking matrix to be
\begin{align*}
Q:=\begin{pmatrix}
t+1&t &t&\cdots&t\\
t & t-2&t-1&\cdots&t-1\\
t & t-1&t-2& &t-1\\
\vdots&\vdots&&\ddots\\
t&t-1&\cdots&t-1& t-2
\end{pmatrix}.
\end{align*}	
This gives us the presentation of the first homology of $L(n)$ as 
\begin{equation*}
H_1\big(L(n)\big)=\langle\mu\vert-(n+t)\mu=0\rangle \cong\Z_{n+t},
\end{equation*}
where $\mu:=\mu_2=\cdots=\mu_n$ and $\mu_1=-n\mu$. Moreover, it is easy to compute that $L$, the first surgery curve in~(\ref{eq:description}), is a characteristic sublink of the surgery description~(\ref{eq:description}). (In case that $t+n$ is even there is another $spin$ structure that corresponds to the whole surgery link. However, we will not need that $spin$ structure and therefore do not consider it here.) Then a straightforward computation with Lemma~\ref{lem:gamma} yields the claimed value for the $\Gamma$-invariant.
\end{proof}

Next, we want to see how the $\Gamma$-invariant changes under stabilization and under appropriately changing the surgery coefficients. 

\begin{prop}\label{prop:halfEulerchange}
	In the same notation as in Proposition~\ref{prop:Gammaint}, we consider the contact surgery diagram given by
	\begin{equation*}
	L_s(n+s)=	L_s(+1){\def\svgwidth{1,6ex}\,\,
\begingroup%
  \makeatletter%
  \providecommand\color[2][]{%
    \errmessage{(Inkscape) Color is used for the text in Inkscape, but the package 'color.sty' is not loaded}%
    \renewcommand\color[2][]{}%
  }%
  \providecommand\transparent[1]{%
    \errmessage{(Inkscape) Transparency is used (non-zero) for the text in Inkscape, but the package 'transparent.sty' is not loaded}%
    \renewcommand\transparent[1]{}%
  }%
  \providecommand\rotatebox[2]{#2}%
  \ifx\svgwidth\undefined%
    \setlength{\unitlength}{49.54732089bp}%
    \ifx\svgscale\undefined%
      \relax%
    \else%
      \setlength{\unitlength}{\unitlength * \real{\svgscale}}%
    \fi%
  \else%
    \setlength{\unitlength}{\svgwidth}%
  \fi%
  \global\let\svgwidth\undefined%
  \global\let\svgscale\undefined%
  \makeatother%
  \begin{picture}(1,1.00059637)%
    \put(0,0){\includegraphics[width=\unitlength,page=1]{PushOff.pdf}}%
  \end{picture}%
\endgroup%
\,\,} \underbrace{L_{s,1}(-1){\def\svgwidth{1,6ex}\,\,
\begingroup%
  \makeatletter%
  \providecommand\color[2][]{%
    \errmessage{(Inkscape) Color is used for the text in Inkscape, but the package 'color.sty' is not loaded}%
    \renewcommand\color[2][]{}%
  }%
  \providecommand\transparent[1]{%
    \errmessage{(Inkscape) Transparency is used (non-zero) for the text in Inkscape, but the package 'transparent.sty' is not loaded}%
    \renewcommand\transparent[1]{}%
  }%
  \providecommand\rotatebox[2]{#2}%
  \ifx\svgwidth\undefined%
    \setlength{\unitlength}{49.54732089bp}%
    \ifx\svgscale\undefined%
      \relax%
    \else%
      \setlength{\unitlength}{\unitlength * \real{\svgscale}}%
    \fi%
  \else%
    \setlength{\unitlength}{\svgwidth}%
  \fi%
  \global\let\svgwidth\undefined%
  \global\let\svgscale\undefined%
  \makeatother%
  \begin{picture}(1,1.00059637)%
    \put(0,0){\includegraphics[width=\unitlength,page=1]{PushOff.pdf}}%
  \end{picture}%
\endgroup%
\,\,}\cdots{\def\svgwidth{1,6ex}\,\,
\begingroup%
  \makeatletter%
  \providecommand\color[2][]{%
    \errmessage{(Inkscape) Color is used for the text in Inkscape, but the package 'color.sty' is not loaded}%
    \renewcommand\color[2][]{}%
  }%
  \providecommand\transparent[1]{%
    \errmessage{(Inkscape) Transparency is used (non-zero) for the text in Inkscape, but the package 'transparent.sty' is not loaded}%
    \renewcommand\transparent[1]{}%
  }%
  \providecommand\rotatebox[2]{#2}%
  \ifx\svgwidth\undefined%
    \setlength{\unitlength}{49.54732089bp}%
    \ifx\svgscale\undefined%
      \relax%
    \else%
      \setlength{\unitlength}{\unitlength * \real{\svgscale}}%
    \fi%
  \else%
    \setlength{\unitlength}{\svgwidth}%
  \fi%
  \global\let\svgwidth\undefined%
  \global\let\svgscale\undefined%
  \makeatother%
  \begin{picture}(1,1.00059637)%
    \put(0,0){\includegraphics[width=\unitlength,page=1]{PushOff.pdf}}%
  \end{picture}%
\endgroup%
\,\,} L_{s,1}(-1)}_{n+s-1},
	\end{equation*}
	where the first $s$ stabilizations are all performed with the same sign, but the last stabilization of $L_{s,1}$ has the opposite sign. Then the above surgery diagram presents again a contact structure $\xi'_s$ on $M=L(n)=L_s(n+s)$, whose  $\Gamma$-invariants are related by 
	\begin{equation*}
	\Gamma\big(L_s(n+s),\xi'_s,\mathfrak s\big)-\Gamma\big(L(n),\xi,\mathfrak s\big)= \pm s\mu \in H_1(M)=\langle \mu \vert -(t+n)\mu\rangle\cong\Z_{t+n}.
	\end{equation*}
\end{prop}

	It follows that for any given first homology class $c$ of $M$ we can find a contact structure on $M$ obtained via integer contact surgery along a stabilization of $L$ whose $\Gamma$-invariant equals $c$.

\begin{proof}
	First, we need to show that the corresponding characteristic sublinks and thus the $spin$ structures are mapped to each other under the Kirby moves relating the underlying integer surgery descriptions. For that, we need to understand how to keep track of $spin$ structures (or equivalently) characteristic sublinks through Kirby moves~\cite{GoSt99}. We think of a characteristic sublink as colored in a different color than the other link components. If we slide $L_i$ over $L_j$ then $L_j$ changes the color if and only if $L_i$ is in the characteristic sublink. After blow ups the characteristic sublink changes as follows. Let $L=L_1\cup\cdots\cup L_n$ be a surgery link and let the characteristic sublink be indexed by $J\subset\{1,\ldots,n\}$. If we blow up the surgery diagram by adding a $(\pm1)$-framed unknot $K$ to the surgery link $L$ and changing its surgery coefficient appropriately, then $K$ gets added to the characteristic sublink if and only if $\sum_{j\in J} \lk(K,L_j)\equiv 0 \,\,\pmod2$.
	\begin{figure}[htbp]{\small
		\begin{overpic}
			{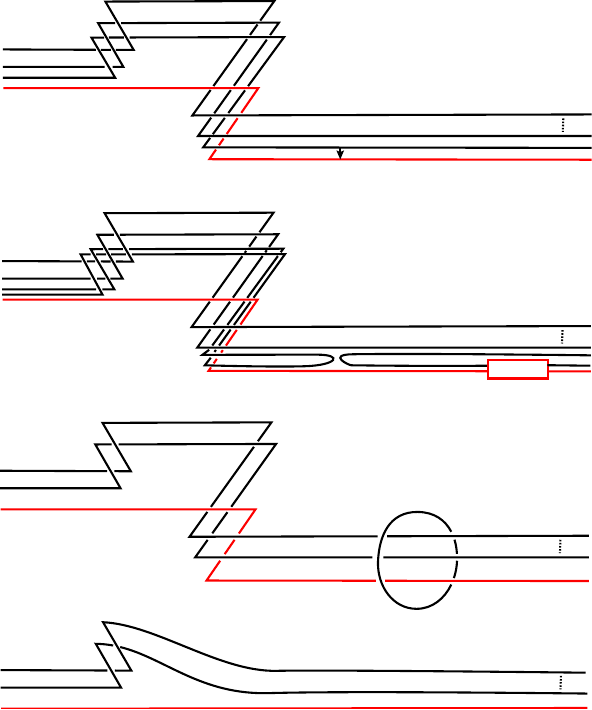}
			\put(270, 254){\color{red} $(+1)$}
			\put(287, 269){$(-1)$}
			\put(287, 288){$(-1)$}
			\put(287, 173){$(-1)$}
			\put(287, 184){$(-1)$}
			\put(287, 163){ $+1$}
			\put(270, 150){\color{red} $(+1)$}
			\put(240, 161.5){\color{red} $t+1$}
			\put(287, 72){$(-1)$}
			\put(287, 80){$(-1)$}
			\put(270, 52){\color{red} $(+1)$}
			\put(210, 40){$+1$}
			\put(287, 18){$(-1)$}
			\put(287, 8){$(-1)$}
			\put(287, -2){\color{red} $(+1)$}
			\put(-18, 310){$(i)$}
			\put(-18, 208){$(ii)$}
			\put(-18, 108){$(iii)$}
			\put(-18, 8){$(iv)$}
	\end{overpic}}
	\caption{Figures $(i)$ shows the characteristic sublinks of $L_1(n+1)$ and Figure~$(iv)$ the characteristic sublink of $L(n)$ in red. To see that they induce the same $spin$ structure on $M$, we first perform a smooth handle slide as indicated with the black arrow in Figure~$(i)$ to get Figure~$(ii)$. The new black link belongs not to the characteristic sublink. Since it is not a Legendrian link anymore we measure its framing with respect to the Seifert framing (as indicated by omitting the parentheses). To get Figure~$(iii)$ we apply an isotopy. Blowing down the $(-1)$-framed unknot yields Figure~$(iv)$. The case for general $s$ follows by induction.}
	\label{fig:characteristic}
\end{figure}
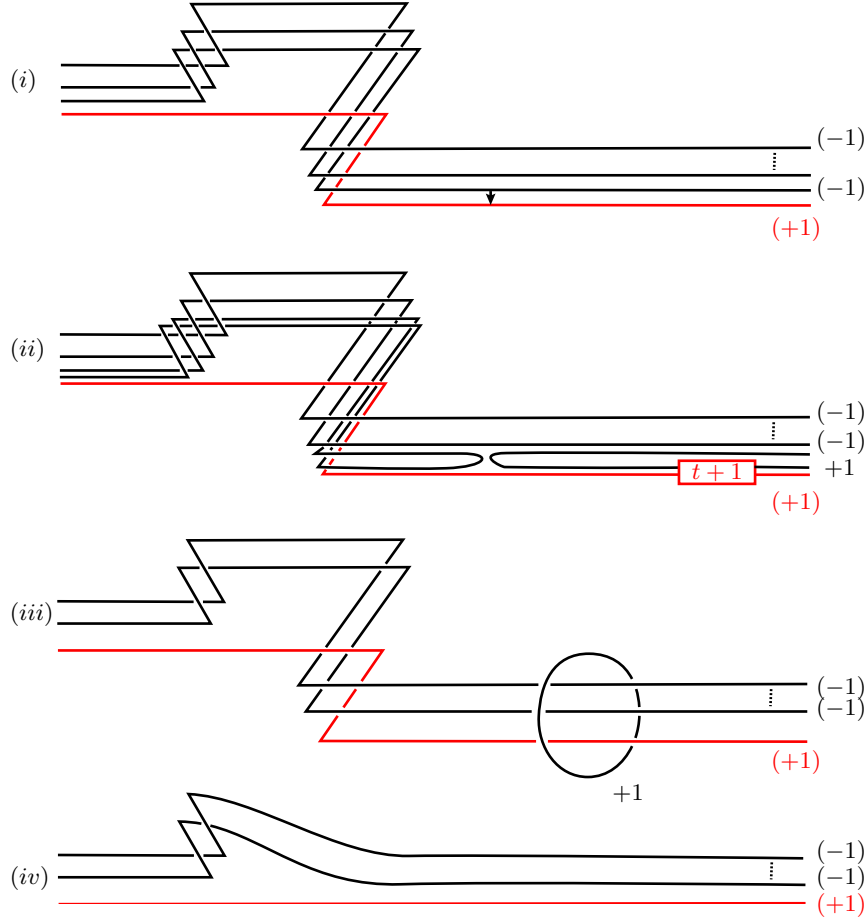	
	
	Figure~\ref{fig:characteristic} shows a sequence of Kirby moves relating the two surgery descriptions while keeping track of the characteristic sublink defining $\mathfrak s$. It follows that we can apply directly Proposition~\ref{prop:Gammaint} and thus get the claimed change of $\Gamma$-invariant.
\end{proof}

Next, we consider rational surgeries. For more general rational contact surgery coefficients $r\in\Q$ there is no general formula since the results will heavily depend on the continued fraction expansion of $r$. However, we can still determine how the $\Gamma$-invariant changes under stabilization by generalizing Proposition~\ref{prop:halfEulerchange} and in particular we will see that we get any possible  $\Gamma$-invariant by stabilizing the surgery link.

\begin{thm}\label{thm:rationalHalfEuler}
	Let $L$ be a Legendrian knot in $(S^3,\xist)$ and we write $(M,\xi)$ for a contact manifold from $L(r)$, $r>1$. For any integer $s\geq1$, there exist contact manifolds $(M,\xi')$ in $L_s(r+s)$  and a $spin$ structure $\mathfrak s$ on $M$ such that		
	\begin{equation*}
	\Gamma\big(\xi',\mathfrak s\big)-\Gamma\big(\xi,\mathfrak s\big)=\pm  s\mu \in H_1(M).
	\end{equation*}
\end{thm}

	In particular, for any given first homology class $c$ of $M$ there is a contact structure on $M$ obtained via rational contact surgery along a stabilization of $L$ whose $\Gamma$-invariant equals $c$. As a direct corollary, we also obtain contact structures on $M$ via contact surgery along a stabilization of $L$ realizing any given even homology class as Euler class.

\begin{proof}
	Let $r=p/q$, $p>q>0$, and let
	\begin{equation}
	\label{eqn:cfe1}
	-\frac{p}{p-q}=[r_1,\ldots,r_n]
	\end{equation}
	be the negative continued fraction expansion. Then the Transformation Lemma~\ref{lem:algo} implies that
	\begin{equation}
	\label{eqn:surgdesc1}
	L(r)=L(1){\def\svgwidth{1,6ex}\,\,
\begingroup%
  \makeatletter%
  \providecommand\color[2][]{%
    \errmessage{(Inkscape) Color is used for the text in Inkscape, but the package 'color.sty' is not loaded}%
    \renewcommand\color[2][]{}%
  }%
  \providecommand\transparent[1]{%
    \errmessage{(Inkscape) Transparency is used (non-zero) for the text in Inkscape, but the package 'transparent.sty' is not loaded}%
    \renewcommand\transparent[1]{}%
  }%
  \providecommand\rotatebox[2]{#2}%
  \ifx\svgwidth\undefined%
    \setlength{\unitlength}{49.54732089bp}%
    \ifx\svgscale\undefined%
      \relax%
    \else%
      \setlength{\unitlength}{\unitlength * \real{\svgscale}}%
    \fi%
  \else%
    \setlength{\unitlength}{\svgwidth}%
  \fi%
  \global\let\svgwidth\undefined%
  \global\let\svgscale\undefined%
  \makeatother%
  \begin{picture}(1,1.00059637)%
    \put(0,0){\includegraphics[width=\unitlength,page=1]{PushOff.pdf}}%
  \end{picture}%
\endgroup%
\,\,} L_{s_1}(-1){\def\svgwidth{1,6ex}\,\,
\begingroup%
  \makeatletter%
  \providecommand\color[2][]{%
    \errmessage{(Inkscape) Color is used for the text in Inkscape, but the package 'color.sty' is not loaded}%
    \renewcommand\color[2][]{}%
  }%
  \providecommand\transparent[1]{%
    \errmessage{(Inkscape) Transparency is used (non-zero) for the text in Inkscape, but the package 'transparent.sty' is not loaded}%
    \renewcommand\transparent[1]{}%
  }%
  \providecommand\rotatebox[2]{#2}%
  \ifx\svgwidth\undefined%
    \setlength{\unitlength}{49.54732089bp}%
    \ifx\svgscale\undefined%
      \relax%
    \else%
      \setlength{\unitlength}{\unitlength * \real{\svgscale}}%
    \fi%
  \else%
    \setlength{\unitlength}{\svgwidth}%
  \fi%
  \global\let\svgwidth\undefined%
  \global\let\svgscale\undefined%
  \makeatother%
  \begin{picture}(1,1.00059637)%
    \put(0,0){\includegraphics[width=\unitlength,page=1]{PushOff.pdf}}%
  \end{picture}%
\endgroup%
\,\,}\ldots{\def\svgwidth{1,6ex}\,\,
\begingroup%
  \makeatletter%
  \providecommand\color[2][]{%
    \errmessage{(Inkscape) Color is used for the text in Inkscape, but the package 'color.sty' is not loaded}%
    \renewcommand\color[2][]{}%
  }%
  \providecommand\transparent[1]{%
    \errmessage{(Inkscape) Transparency is used (non-zero) for the text in Inkscape, but the package 'transparent.sty' is not loaded}%
    \renewcommand\transparent[1]{}%
  }%
  \providecommand\rotatebox[2]{#2}%
  \ifx\svgwidth\undefined%
    \setlength{\unitlength}{49.54732089bp}%
    \ifx\svgscale\undefined%
      \relax%
    \else%
      \setlength{\unitlength}{\unitlength * \real{\svgscale}}%
    \fi%
  \else%
    \setlength{\unitlength}{\svgwidth}%
  \fi%
  \global\let\svgwidth\undefined%
  \global\let\svgscale\undefined%
  \makeatother%
  \begin{picture}(1,1.00059637)%
    \put(0,0){\includegraphics[width=\unitlength,page=1]{PushOff.pdf}}%
  \end{picture}%
\endgroup%
\,\,} L_{s_1, \ldots, s_n}(-1).
	\end{equation}

	We choose an orientation on the surgery link from Equation~(\ref{eqn:surgdesc1}) that is inherited from the orientation of $L$. 
	
	By induction it is enough to prove the theorem for $s=1$ and thus we consider the contact manifold $L_1(r+1)=(M,\xi')$. Then we get
	\begin{equation*}
	-\frac{p+q}{(p+q)-q}=[-2,r_1,r_2,\ldots,r_n],
	\end{equation*}
	where $[r_1,\ldots,r_n]$ is the continued fraction expansion from Equation~(\ref{eqn:cfe1}). Thus we obtain from the Transformation Lemma~\ref{lem:algo} a contact $(\pm1)$-surgery description of $L_1(r+1)$ as
	\begin{equation}
	L_1(+1){\def\svgwidth{1,6ex}\,\,
\begingroup%
  \makeatletter%
  \providecommand\color[2][]{%
    \errmessage{(Inkscape) Color is used for the text in Inkscape, but the package 'color.sty' is not loaded}%
    \renewcommand\color[2][]{}%
  }%
  \providecommand\transparent[1]{%
    \errmessage{(Inkscape) Transparency is used (non-zero) for the text in Inkscape, but the package 'transparent.sty' is not loaded}%
    \renewcommand\transparent[1]{}%
  }%
  \providecommand\rotatebox[2]{#2}%
  \ifx\svgwidth\undefined%
    \setlength{\unitlength}{49.54732089bp}%
    \ifx\svgscale\undefined%
      \relax%
    \else%
      \setlength{\unitlength}{\unitlength * \real{\svgscale}}%
    \fi%
  \else%
    \setlength{\unitlength}{\svgwidth}%
  \fi%
  \global\let\svgwidth\undefined%
  \global\let\svgscale\undefined%
  \makeatother%
  \begin{picture}(1,1.00059637)%
    \put(0,0){\includegraphics[width=\unitlength,page=1]{PushOff.pdf}}%
  \end{picture}%
\endgroup%
\,\,} L_{1,1}(-1){\def\svgwidth{1,6ex}\,\,
\begingroup%
  \makeatletter%
  \providecommand\color[2][]{%
    \errmessage{(Inkscape) Color is used for the text in Inkscape, but the package 'color.sty' is not loaded}%
    \renewcommand\color[2][]{}%
  }%
  \providecommand\transparent[1]{%
    \errmessage{(Inkscape) Transparency is used (non-zero) for the text in Inkscape, but the package 'transparent.sty' is not loaded}%
    \renewcommand\transparent[1]{}%
  }%
  \providecommand\rotatebox[2]{#2}%
  \ifx\svgwidth\undefined%
    \setlength{\unitlength}{49.54732089bp}%
    \ifx\svgscale\undefined%
      \relax%
    \else%
      \setlength{\unitlength}{\unitlength * \real{\svgscale}}%
    \fi%
  \else%
    \setlength{\unitlength}{\svgwidth}%
  \fi%
  \global\let\svgwidth\undefined%
  \global\let\svgscale\undefined%
  \makeatother%
  \begin{picture}(1,1.00059637)%
    \put(0,0){\includegraphics[width=\unitlength,page=1]{PushOff.pdf}}%
  \end{picture}%
\endgroup%
\,\,} L_{1,1,s_1-1}(-1){\def\svgwidth{1,6ex}\,\,
\begingroup%
  \makeatletter%
  \providecommand\color[2][]{%
    \errmessage{(Inkscape) Color is used for the text in Inkscape, but the package 'color.sty' is not loaded}%
    \renewcommand\color[2][]{}%
  }%
  \providecommand\transparent[1]{%
    \errmessage{(Inkscape) Transparency is used (non-zero) for the text in Inkscape, but the package 'transparent.sty' is not loaded}%
    \renewcommand\transparent[1]{}%
  }%
  \providecommand\rotatebox[2]{#2}%
  \ifx\svgwidth\undefined%
    \setlength{\unitlength}{49.54732089bp}%
    \ifx\svgscale\undefined%
      \relax%
    \else%
      \setlength{\unitlength}{\unitlength * \real{\svgscale}}%
    \fi%
  \else%
    \setlength{\unitlength}{\svgwidth}%
  \fi%
  \global\let\svgwidth\undefined%
  \global\let\svgscale\undefined%
  \makeatother%
  \begin{picture}(1,1.00059637)%
    \put(0,0){\includegraphics[width=\unitlength,page=1]{PushOff.pdf}}%
  \end{picture}%
\endgroup%
\,\,} \ldots {\def\svgwidth{1,6ex}\,\,
\begingroup%
  \makeatletter%
  \providecommand\color[2][]{%
    \errmessage{(Inkscape) Color is used for the text in Inkscape, but the package 'color.sty' is not loaded}%
    \renewcommand\color[2][]{}%
  }%
  \providecommand\transparent[1]{%
    \errmessage{(Inkscape) Transparency is used (non-zero) for the text in Inkscape, but the package 'transparent.sty' is not loaded}%
    \renewcommand\transparent[1]{}%
  }%
  \providecommand\rotatebox[2]{#2}%
  \ifx\svgwidth\undefined%
    \setlength{\unitlength}{49.54732089bp}%
    \ifx\svgscale\undefined%
      \relax%
    \else%
      \setlength{\unitlength}{\unitlength * \real{\svgscale}}%
    \fi%
  \else%
    \setlength{\unitlength}{\svgwidth}%
  \fi%
  \global\let\svgwidth\undefined%
  \global\let\svgscale\undefined%
  \makeatother%
  \begin{picture}(1,1.00059637)%
    \put(0,0){\includegraphics[width=\unitlength,page=1]{PushOff.pdf}}%
  \end{picture}%
\endgroup%
\,\,} L_{1,1, s_1-1,s_2\ldots s_n}(-1),	\label{eqn:surgdesc2}
	\end{equation}
	where we consider the case that the two stabilizations of $L$ yielding $L_{1,1}$ have different signs. All the other stabilizations of $L_{1,1}$ yielding the $L_{1,1, s_1-1,s_2,\ldots,s_i}$ have the same sign as the original stabilizations of $L_{s_1,\ldots,s_i}$ in Equation~(\ref{eqn:surgdesc1}). 
	
	We write 
	\begin{align*}
		\rot_1&:=\rot(L),\\
		\rot_{i+1}&:=\rot(L_{s_1, \ldots, s_i}), \,\text{ for }\,i=1,\ldots,n,
	\end{align*}
for the rotation numbers of the Legendrian knots from~(\ref{eqn:surgdesc1}) and $\mu_1,\ldots,\mu_{n+1}$ for their meridians. Similarly, we write 
	\begin{align*}
\rot'_0&:=\rot(L_{1,1})=\rot_1,\\
\rot'_1&:=\rot(L_1)=\rot_1\pm\,1,\\
\rot'_{i+1}&:=\rot(L_{1,1,s_1-1,s_2,\ldots,s_i})=\rot_{i+1}\pm\,1, \,\text{  for }\,i=1,\ldots,n, 
\end{align*}
for the rotation numbers of the surgery description~(\ref{eqn:surgdesc2}) and $\mu'_0,\ldots,\mu'_{n+1}$ for their meridians. We also can express the linking matrix $Q'$ of~(\ref{eqn:surgdesc2})  in terms of the linking matrix $Q$ of~(\ref{eqn:surgdesc1}) as follows
\begin{align*}
Q'=\begin{pmatrix}
t-2 &\begin{matrix} t-1&\cdots&t-1 \end{matrix} \\
\begin{matrix} t-1\\ \vdots\\t-1& \end{matrix} & Q
\end{pmatrix}- \begin{pmatrix}
1&\ldots& 1\\
1&\ldots& 1\\
\vdots& \ddots&\vdots\\
1&\ldots& 1
\end{pmatrix}.
\end{align*}	
	
As in Figure~\ref{fig:characteristic} we can transform the surgery description~(\ref{eqn:surgdesc2}) into the surgery description~(\ref{eqn:surgdesc1}) by performing a single handle slide and then blowing down. Thus we observe that if $J\subset\{1,\ldots,n+1\}$ represents a characteristic sublink of~(\ref{eqn:surgdesc1}) then it also represents a characteristic sublink of~(\ref{eqn:surgdesc2}) and so we choose a characteristic sublink given by $J$ for both surgery descriptions that represents a $spin$ structure $\mathfrak s$ on $M$.

Finally, we relate the homology classes of $\mu_1, \dots, \mu_n$ and $\mu'_1, \dots, \mu'_n$, by following the smooth handle slides as in Figure~\ref{fig:characteristic} relating the two surgery descriptions. Using Lemma~\ref{lem:gamma} we observe that the difference $\Gamma(\xi',\mathfrak s)-\Gamma(\xi,\mathfrak s)$ is given by $\pm \mu$, where $\mu$ is the meridian of $L$ generating $H_1(M)$ and the sign is given by the sign of the stabilizations in~(\ref{eqn:surgdesc2}).
\end{proof}


\section{Inequalities between contact surgery numbers}\label{sec:ineq}

In this section, we start analysing contact surgery numbers. We first discuss general inequalities between various contact surgery numbers.

Directly from the definitions, we get many inequalities between the different versions of contact surgery numbers, for example we have
\begin{equation*}
\cs\leq\cs_{\Z}, \text{ and } \cs_{1/\Z}\leq\cs_{\pm1}\leq\cs_{U, \pm1}.
\end{equation*}

The main result of this section is the following non-trivial inequalities.

\begin{thm} \label{thm:continequality}
	Let $(M,\xi)$ be a contact $3$-manifold. Then the following inequalities hold true
	\begin{align*}
	\cs_{U}(M,\xi)&\leq3\cs(M,\xi),\\
	\cs_{U, \Z}(M,\xi)&\leq3\cs_{\Z}(M,\xi),\\
	\cs_{U, 1/\Z}(M,\xi)&\leq3\cs_{1/\Z}(M,\xi),\\
	\cs_{U, \pm1}(M,\xi)&\leq3\cs_{\pm1}(M,\xi).
	\end{align*}
\end{thm}

Before we discuss the proof we review the following analogous result by Guo and Yu~\cite{GuYu10} in the topological category.

\begin{thm} [Guo--Yu~\cite{GuYu10}]\label{thm:inequality}
	Let $M$ be a $3$-manifold. Then the following inequalities hold true
	\begin{align*}
	\su_{\Z, U}(M)&\leq3\su_\Z(M)\\
	\su_{U}(M)&\leq3\su(M)
	\end{align*}
\end{thm}

The first inequality is stated in their paper~\cite{GuYu10}, and the second inequality follows by similar methods. 

For completeness (and since we want to use similar ideas later in the proof of Theorem~\ref{thm:continequality}) we shortly summarize a variation of their proof and also indicate why the second inequality holds.

\begin{proof}[Proof of Theorem~\ref{thm:inequality}]
	The main ingredient in the proof is the next lemma due to Guo and Yu~\cite{GuYu10}.
	
	\begin{lem}[Guo--Yu~\cite{GuYu10}]\label{lem:skein}
		Any oriented knot $K$ in $S^3$ admits a skein move transforming it into a two-component link $L$ consisting of two unknots.
	\end{lem}
	
			\begin{figure}[htbp]{\small
		\begin{overpic}
			{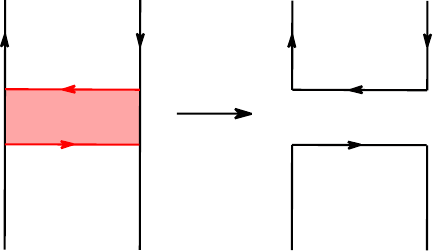}
	\end{overpic}}
		\caption{A skein move}
		\label{fig:skeinmove}
\end{figure}
	A \textbf{skein move} is the following operation on an oriented knot $K$. We take an oriented band that meets the knot $K$ in two disjoint arcs where the orientations of $K$ and the band disagree. Then we remove the arcs from the knot and connect the endpoints via the remaining boundary components of the band, see Figure~\ref{fig:skeinmove}. Notice that if one takes a thin annulus with one boundary component being $K$, and attaches the band to it, we have a surface $\Sigma$ and $\partial \Sigma$ is $K\cup J\cup U$ where $J$ and $U$ are the unknots associated to the skein move.
	
	Now let $K$ be a component in a minimal surgery description of $M$. By Lemma~\ref{lem:skein} we can find a diagram of the surgery link in which there is a band as in Figure~\ref{fig:skeinmove} such that performing a skein move on the band will transform $K$ into a link consisting of two unknots $J\cup U$. We illustrate this for a knot $K$ in Figure~\ref{fig:GuYu}(i).
	If the surgery coefficient of $K$ is not an integer we use the standard method to change the surgery diagram into a diagram with only integer coefficients. (We perform inverse slam dunks, as indicated in Figure~\ref{fig:GuYu}(ii), with surgery coefficients on the chain of unknots given as the entries in the continued fraction expression of the old surgery coefficient.)
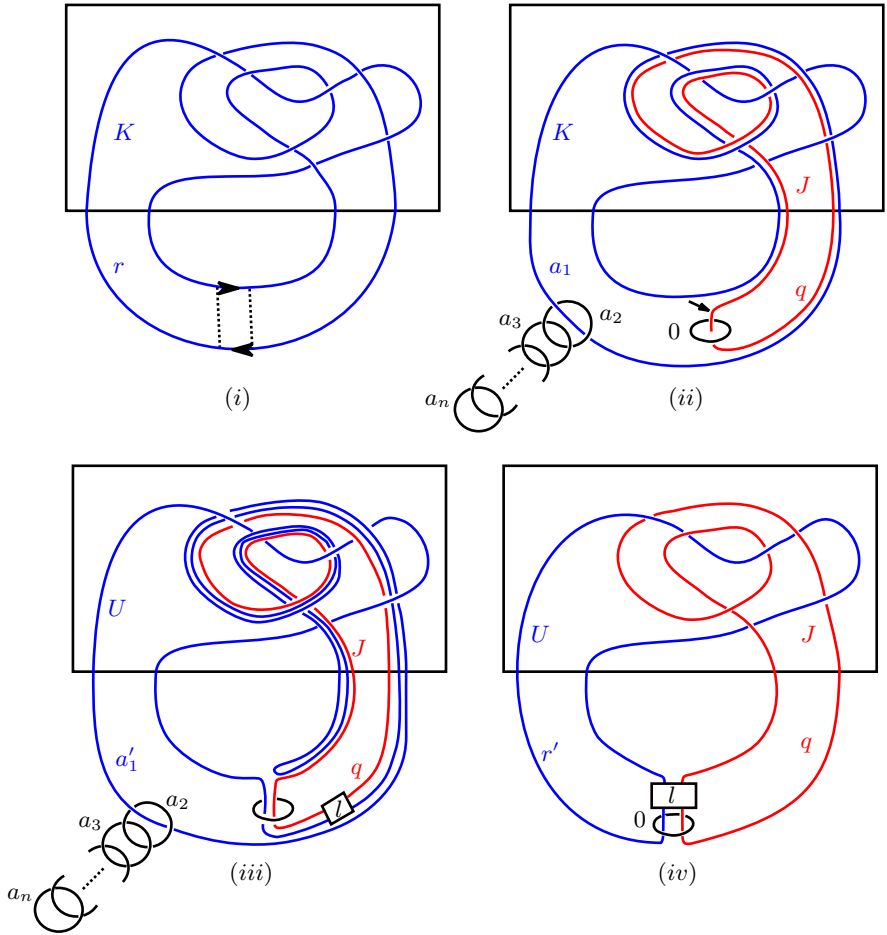
\begin{figure}[htbp]{\small
	\begin{overpic}
			{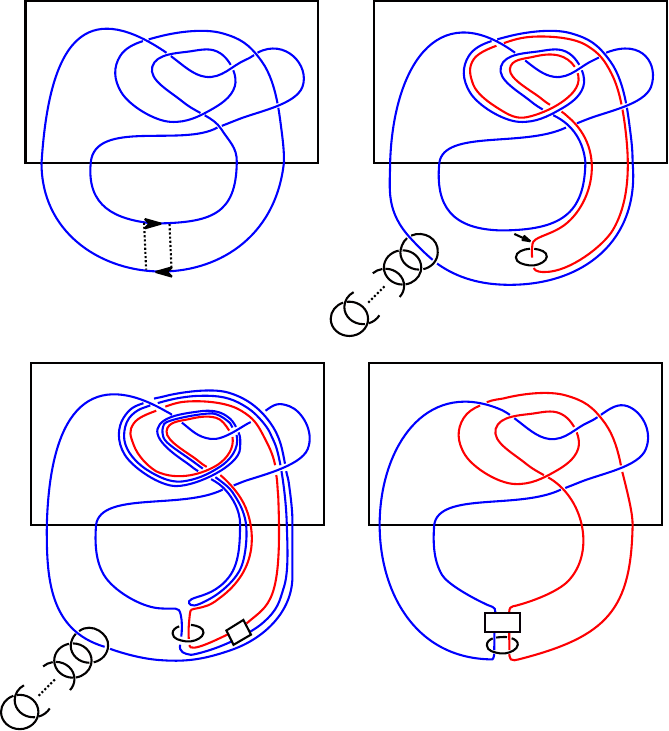}
			\put(30, 300){\color{blue} $K$}
			\put(30, 250){\color{blue} $r$}
			\put(72, 200){$(i)$}
			\put(196, 300){\color{blue} $K$}
			\put(195, 250){\color{blue} $a_1$}
			\put(288, 280){\color{red} $J$}
			\put(288, 241){\color{red} $q$}
			\put(240, 200){$(ii)$}
			\put(240, 225){$0$}
			\put(214, 232){$a_2$}
			\put(176, 230){$a_3$}
			\put(148, 200){$a_n$}
			\put(28, 120){\color{blue} $U$}
			\put(31, 63){\color{blue} $a_1'$}
			\put(120, 105){\color{red} $J$}
			\put(120, 60){\color{red} $q$}
			\put(290, 110){\color{red} $J$}
			\put(290, 70){\color{red} $q$}
			\put(114, 44){$l$}
			\put(50, 48){$a_2$}
			\put(17, 40){$a_3$}
			\put(-10, 13){$a_n$}
			\put(74, 20){$(iii)$}
			\put(236, 20){$(iv)$}
			\put(240, 49){$l$}
			\put(227, 40){$0$}
			\put(188, 110){\color{blue} $U$}
			\put(192, 65){\color{blue} $r'$}
	\end{overpic}}
		\caption{Kirby moves changing a knot into a simple $3$-component link}
		\label{fig:GuYu}
\end{figure}	

	Next, we introduce curves into the diagram $J$ (one of the components of the link formed by the skein move) with any framing $q$ and a zero framed meridian $\mu$, see Figure~\ref{fig:GuYu}(ii). A slam dunk move shows that this new surgery diagram is equivalent to the original one. 
	In Figure~\ref{fig:GuYu}(iii) we have performed a handle slide of $K$ over $J$. The number of twists $l$ and the new surgery coefficient $a_1'$ of $U$ will depend on $q$, $a_1$ and the linking number of $K$ and $J$. After an isotopy and possibly slam dunking the chain of unknots away we get a surgery presentation of the same $3$-manifold $M$ consisting of three unknots, see Figure~\ref{fig:GuYu}(iv). By doing the same construction for every component of the original surgery diagram the result follows. 
\end{proof}

\begin{rem}
	We do not necessarily need to deform the surgery coefficient of $K$ into an integer by inverse slam dunks. We can see the handle slide also as a purely $3$-dimensional operation: for that, we consider the knot $K$ as a knot in the manifold obtained from $S^3$ by surgery along $J$ and its meridian. If we move a small part of $K$ near $J$ and slide it over the newly glued-in solid torus the knot $K$ will deform to the knot $U$ exactly as in Figure~\ref{fig:GuYu}, cf.\ the proof of the contact handle slide in~\cite{CEK21}. 
\end{rem}

The above slight adaptation of the proof from~\cite{GuYu10} generalizes to the setting of contact manifolds by using the results from Section~\ref{sec:Kirby}.

\begin{proof}[Proof of Theorem~\ref{thm:continequality}] We consider the front projection of a Legendrian knot $K$ in a minimal contact surgery presentation of $(M,\xi)$. By Lemma~\ref{lem:skein} we know that there exists a band $B$ such that if we do a skein move along the band $B$ we get topologically a simple $2$-component link. By sliding the band along the Legendrian knot we can assume one endpoint of $B$ to lie at a cusp of $K$ and the other at a regular point of the front projection. If in the front projection the number of half-twists of that band is odd, we perform a Legendrian Reidemeister move I near the regular point of $K$ where the band meets $K$, i.e. we can arrange a situation as in Figure~\ref{fig:abstractPicture}(i). See Figure~\ref{fig:LegendrianInequality} for an example. 
	
	\begin{figure}[htbp]{\small
	\begin{overpic}
			{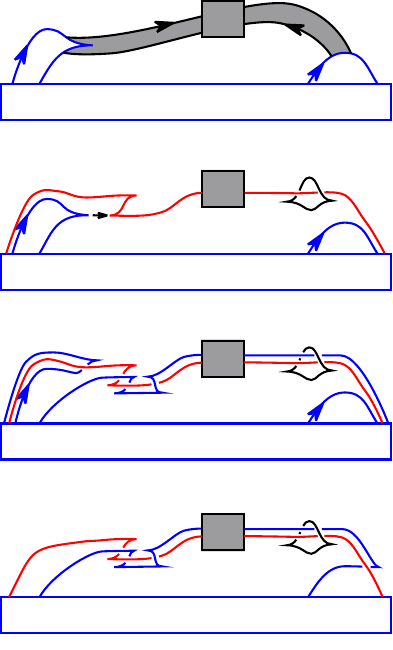}
			\put(95, 260){\color{blue} $K$}
			\put(95, 179){\color{blue} $K$}
			\put(95, 97){\color{blue} $U$}
			\put(95, 13){\color{blue} $U$}
			\put(30, 275){\color{blue} $(r)$}
			\put(30, 195){\color{blue} $(r)$}
			\put(30, 114){\color{blue} $(r)$}
			\put(30, 30){\color{blue} $(r)$}
			\put(104, 300){$B$}
			\put(104, 218){$B$}
			\put(104, 137){$B$}
			\put(104, 53){$B$}
			\put(45, 222){\color{red} $J$}
			\put(180, 210){\color{red} $(-1)$}
			\put(155, 228){$\mu$}
			\put(125, 206){$(+1)$}
			\put(50, 140){\color{red} $J$}
			\put(82, 123){\color{red} $(-1)$}
			\put(155, 147){$\mu$}
			\put(125, 125){$(+1)$}
			\put(50, 57){\color{red} $J$}
			\put(-6, 45){\color{red} $(-1)$}
			\put(155, 64){$\mu$}
			\put(125, 41){$(+1)$}
			\put(95, 242){$(i)$}
			\put(95, 162){$(ii)$}
			\put(95, 81){$(iii)$}
			\put(95, -2){$(iv)$}
	\end{overpic}}
		\caption{Contact Kirby moves changing a Legendrian knot into a simple $3$-component Legendrian link. Here the box $B$ denotes an unspecified band as in Lemma~\ref{lem:skein}. Note that in general the band $B$ will be linked with $K$ in a non-trivial way.}
		\label{fig:abstractPicture}
\end{figure}

	\begin{figure}[htbp]{\small
	\begin{overpic}
			{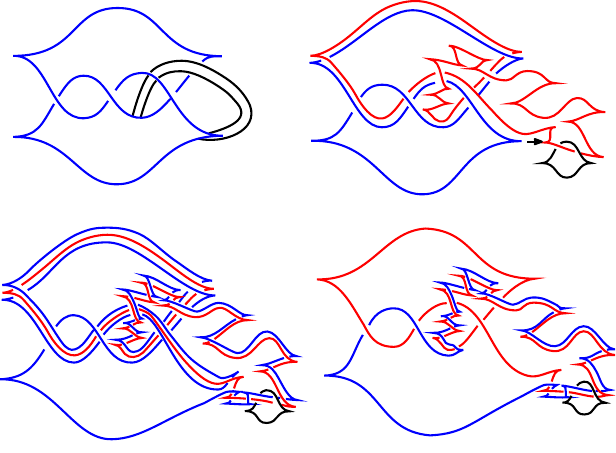}
			\put(40, 200){\color{blue} $K$}
			\put(40, 140){\color{blue} $(r)$}
			\put(170, 190){\color{blue} $K$}
			\put(170, 150){\color{blue} $(r)$}
			\put(245, 200){\color{red} $J$}
			\put(287, 149){\color{red} $(-1)$}
			\put(258, 126){$\mu$}
			\put(278, 126){$(+1)$}
			\put(30, 84){\color{blue} $U$}
			\put(40, 17){\color{blue} $(r)$}
			\put(123, 63){\color{red} $J$}
			\put(137, 34){\color{red} $(-1)$}
			\put(115, 6){$\mu$}
			\put(134, 6){$(+1)$}
			\put(183, 74){\color{blue} $U$}
			\put(183, 25){\color{blue} $(r)$}
			\put(257, 80){\color{red} $J$}
			\put(295, 40){\color{red} $(-1)$}
			\put(267, 10){$\mu$}
			\put(286, 10){$(+1)$}
			\put(50, 117){$(i)$}
			\put(198, 116){$(ii)$}
			\put(48, -2){$(iii)$}
			\put(200, -2){$(iv)$}
	\end{overpic}}
		\caption{Contact Kirby moves changing a contact surgery along a Legendrian trefoil into a contact surgery along a simple $3$-component Legendrian link.}
			\label{fig:LegendrianInequality}
\end{figure}
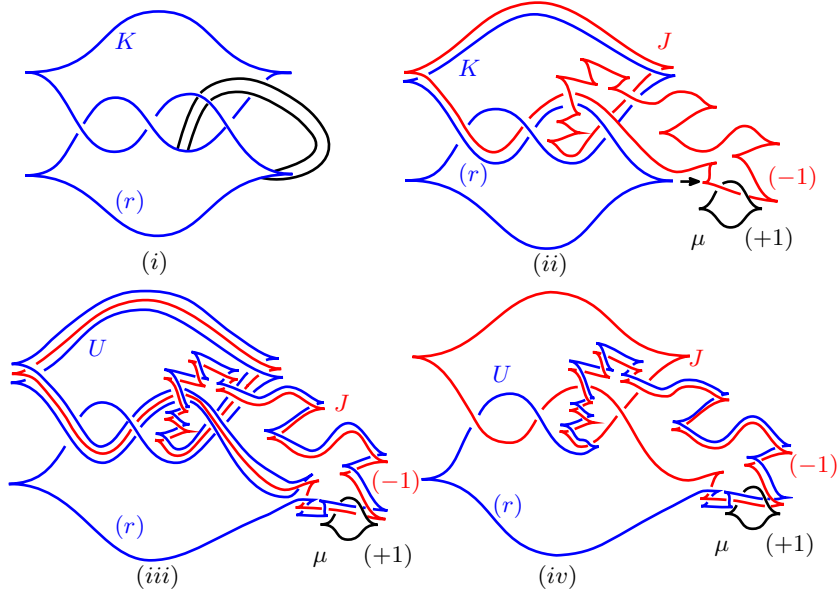	
		
		Next, we introduce two canceling surgeries along Legendrian knots $\mu\cup J$ by an inverse contact slam dunk (Lemma~\ref{lem:contactSlamDunk}) as in Figure~\ref{fig:abstractPicture}(ii), where $J$ is the once stabilized obvious Legendrian realization of one of the unknots produced by the Skein move with contact framing $(-1)$ and $\mu$ is a meridian of $J$ with $\tb=-1$ and contact framing $(+1)$, see Figure~\ref{fig:abstractPicture}(ii). 
		Then we perform a contact handle slide of $K$ over $J$ and get a surgery description of the same contact manifold consisting only of Legendrian unknots, shown in Figure~\ref{fig:abstractPicture}(iii) and~(iv). Again, a contact handle slide can be seen as an isotopy of $K$ in the contact manifold obtained by surgery along $J$, and thus the same proof works for other coefficients of $K$ as well. In this setting, the contact surgery coefficient stays the same, see for example~\cite{CEK21}. So we also obtain the other inequalities.
	\end{proof}

	
	\section{Upper bounds on contact surgery numbers}\label{upperboundsoncsn} We can often bound surgery numbers from above by providing explicit constructions. We will start with a discussion of contact structures on $S^3$ and then upgrade this by relating contact surgery numbers of overtwisted contact manifolds to the topological surgery numbers of the underlying topological manifold. We will start with an observation of how contact surgery numbers behave under contact connected sum.
	
	\begin{lem}\label{lem:connectedSum}
		Let $(M_1,\xi_1)$ and $(M_2,\xi_2)$ be contact manifolds and $(M_1\# M_2,\xi_1\#\xi_2)$ be their contact connected sum. Then their contact surgery numbers are related by
		\begin{equation*}
		\cs_*(M_1\# M_2,\xi_1\#\xi_2)\leq\cs_*(M_1,\xi_1)+\cs_*(M_2,\xi_2)
		\end{equation*}
		where $*$ is $\emptyset$, $\Z$, $1/\Z$, $\pm1$, or $U$.
	\end{lem} 
	
	\begin{proof}
		Let $L_i$, for $i=1,2$, be surgery diagrams of $(M_i,\xi_i)$ with property $*$ and with minimal number of components. A surgery diagram of $(M_1\# M_2,\xi_1\#\xi_2)$ is given by the disjoint union of $L_1$ and $L_2$ and thus the claim follows.
	\end{proof}
	
	 However, the above inequality is in general not an equality. The easiest such example comes from $\xi_0$, the overtwisted contact structure on $S^3$ with $\de_3$-invariant $0$. Since $\cs(S^3,\xi_0)>0$ and $(S^3,\xi_0)\#(S^3,\xi_0)=(S^3,\xi_0)$ it follows that
	 \begin{equation*}
	 	\cs\big(S^3,\xi_0\big)=\cs\big((S^3,\xi_0)\#(S^3,\xi_0)\big)<\cs\big(S^3,\xi_0\big)+\cs\big(S^3,\xi_0\big).
	 \end{equation*}
	 
	 In Section~\ref{sec:comp}, we will construct more examples of overtwisted contact manifolds $(M_1,\xi_1)$ and $(M_2,\xi_2)$ such that 
	 \begin{equation*}
	 \cs(M_1\#M_2,\xi_1\#\xi_2)<\cs(M_1,\xi_1)+\cs(M_2,\xi_2).
	 \end{equation*}
	 
	 The same phenomena can also happen for tight contact manifolds. ~\cite{Ya16,LS16}, construct examples of Legendrian knots in $(S^3,\xist)$ along which Legendrian surgery yields a reducible manifold. 
	 
	 In the rest of the article, we will also use repeatedly the following two elementary observations, cf.~\cite{Ke17}.
	 
	 \begin{lem}\label{lem:presentingKnots}
	 	Let $L=L_1\cup\cdots \cup L_n\subset(S^3,\xist)$ be a contact surgery diagram with $(1/{k_i})$-contact surgery coefficients, $i=1,\ldots, n$, of a contact manifold $(M,\xi)$. Then any Legendrian knot $K$ in $(M,\xi)$ can be represented by a Legendrian knot in the exterior of $L$.
	 \end{lem}
	 
	 \begin{proof}
	 	By  the classification~\cite{Gir00, Ho00} of tight contact structures on $S^1\times D^2$ all newly glued-in solid tori are standard neighborhoods of Legendrian knots $L_i$, $i=1,\ldots, n$, in $(M,\xi)$ and by Lemma~\ref{lem:cancelation} contact $(-1/{k_i})$-surgeries along the Legendrian knots $L_i$ in $(M,\xi)$ reproduces $(S^3,\xist)$. The standard neighborhoods of the Legendrian knots $L_i$, used to construct $(S^3,\xist)$, can be chosen arbitrarily small. Thus,  it is enough to show that an arbitrary Legendrian knot $K$ in an arbitrary contact manifold $(M,\xi)$ can be made disjoint from an arbitrary Legendrian link $L = L_1 \cup\cdots\cup L_n$ by a Legendrian isotopy. By Darboux's theorem, it is sufficient to show the same statement for Legendrian knot segments in $(\R^3,\xist)$.
	 	
	 	For this consider the front projection of the Legendrian knot segment of $K$ and the Legendrian link segments $L_i$. By the Transversality theorem, $K$ can be $\mathcal{C}^\infty$-close approximated relative to its boundary by a curve that is transverse to all $L_i$ and represents a Legendrian knot segment in $(\R^3,\xist)$, which is in $(\R^3,\xist)$ disjoint from the $L_i$.
	 \end{proof}
	 
	 \begin{rem}
	 	On the other hand, Lemma~\ref{lem:presentingKnots} does not hold for arbitrary surgery coefficients, for a concrete example see Example~4.7.2 in~\cite{Ke17}. In particular, it follows that, in general, a single contact $r$-surgery cannot be reversed by a single contact surgery.
	 \end{rem}
	 
	 We get the following application for contact surgery numbers.
	 
	 \begin{prop}\label{prop:addition}
	 	If $\cs_{1/\Z}(M,\xi)\leq k$ and if we can obtain another contact manifold $(N,\eta)$ by contact $(\pm1/n_i)$-surgeries along an $l$ component Legendrian link $L$ in $(M,\xi)$. Then 
	 	\begin{equation*}
	 	\cs_{1/\Z}(N,\eta)\leq k+l.
	 	\end{equation*}
	 \end{prop}
	 
	 \begin{proof}
	 	Let $J$ be a $k$-component Legendrian link in $(S^3,\xist)$ along which contact $(\pm1/n_i)$-surgery produces $(M,\xi)$. By Lemma~\ref{lem:presentingKnots} we can present $L$ in the exterior of $J$. Thus we have constructed a contact $(\pm1/n_i)$-surgery diagram for $(N,\eta)$ along a $(k+l)$-component link.
	 \end{proof}
	 			
	\subsection{The 3-sphere}\label{subsec:S3} We recall from the introduction, that on $S^3$ there is a unique tight (and in fact Stein fillable) contact structure $\xist$~\cite{El92}, which is the unique contact manifold with vanishing contact surgery number. The overtwisted contact structures on $S^3$ are classified by their $\de_3$-invariants, which take integer values~\cite{El89,Go98}. We denote the unique overtwisted contact structure on $S^3$ with $\de_3$-invariant equal to $n$ by $\xi_n$.
	
	We begin by showing that any contact structure on $S^3$ can be obtained from $(S^3,\xist)$ by at most two contact $(\pm1)$-surgeries. This improves a result from~\cite{DiGeSt04}, where they could obtain an upper bound of $3$. In Section~\ref{sec:S3Integer} we will compute all contact surgery numbers of all contact structures on $S^3$ and in particular, we will see that the inequality in Lemma~\ref{lem:connectedSum} is not always an equality.
	
	\begin{prop} \label{prop:S3upperbound}
		For every contact structure $\xi$ on $S^3$ we have $\cs_{\pm1}(S^3,\xi)\leq2$.
	\end{prop}
	
	\begin{proof}
		For the contact structures with odd $d_3$-invariant, we use a construction due to Ding, Geiges and Stipsicz~\cite{DiGeSt04}, which we briefly recall below.
		
		First we consider the contact surgery diagram~$(i)$ in Figure~\ref{fig:S310} which smoothly represents $S^3$ and via Lemma~\ref{lem:d3} we verify that $(i)$ yields $\xi_1$. Thus, we have shown that $\cs_{\pm1}(S^3,\xi_1)=1$. In Proposition~\ref{prop:plusminusS3} we will show that this is the only contact structure on $S^3$ with that property.
		
		Next, we take an arbitrary Legendrian knot $K$ with Thurston--Bennequin invariant $t$ and rotation number $r$ and we consider $K(+1){\def\svgwidth{1,6ex}\,\,
\begingroup%
  \makeatletter%
  \providecommand\color[2][]{%
    \errmessage{(Inkscape) Color is used for the text in Inkscape, but the package 'color.sty' is not loaded}%
    \renewcommand\color[2][]{}%
  }%
  \providecommand\transparent[1]{%
    \errmessage{(Inkscape) Transparency is used (non-zero) for the text in Inkscape, but the package 'transparent.sty' is not loaded}%
    \renewcommand\transparent[1]{}%
  }%
  \providecommand\rotatebox[2]{#2}%
  \ifx\svgwidth\undefined%
    \setlength{\unitlength}{49.54732089bp}%
    \ifx\svgscale\undefined%
      \relax%
    \else%
      \setlength{\unitlength}{\unitlength * \real{\svgscale}}%
    \fi%
  \else%
    \setlength{\unitlength}{\svgwidth}%
  \fi%
  \global\let\svgwidth\undefined%
  \global\let\svgscale\undefined%
  \makeatother%
  \begin{picture}(1,1.00059637)%
    \put(0,0){\includegraphics[width=\unitlength,page=1]{PushOff.pdf}}%
  \end{picture}%
\endgroup%
\,\,} K_2(+1)$, where $K_2$ is a $2$-fold stabilization of $K$ where both stabilizations are positive, see Figure~\ref{fig:S3abstract} (i). Since contact $(+1)$-surgery along $K_2$ is topologically the same as doing contact $(-1)$-surgery along $K$, by Lemma~\ref{lem:cancelation} this surgery yields a contact structure on $S^3$. Using Lemma~\ref{lem:d3} its $\de_3$-invariant is computed as
		\begin{equation*}
		\de_3=-(t+r).
		\end{equation*}
		It is well known that for a Legendrian knot $\tb+\rot$ is always an odd number~\cite{Ge08} and that any odd number can be realized like this and thus we have shown that any contact structure on $S^3$ with odd $d_3$-invariant has $\cs_{\pm1}\leq2$.
		
	\begin{figure}[htbp]{\small
	\begin{overpic}
			{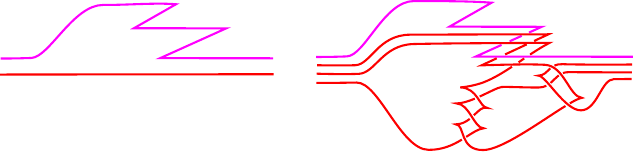}
			\put(30, 30){\color{red} $(+1)$}
			\put(190, 43){\color{red} $(-1)$}
			\put(10, 65){\color{pink} $(+1)$}
			\put(160, 65){\color{pink} $(+1)$}
			\put(48, -3){$(i)$}
			\put(215, -3){$(ii)$}
	\end{overpic}}
\caption{Contact $(\pm1)$-surgery diagrams along $2$-component links for all overtwisted contact structures on $S^3$. The construction in $(i)$ yields odd $\de_3$-invariants while $(ii)$ yields even $\de_3$-invariants. Figure $(ii)$ is obtained by first performing Legendrian Reidemeister $I$ moves to $\pm\Delta$ and $K$ and then taking the appropriate interior connected sum.}
			\label{fig:S3abstract}
\end{figure}
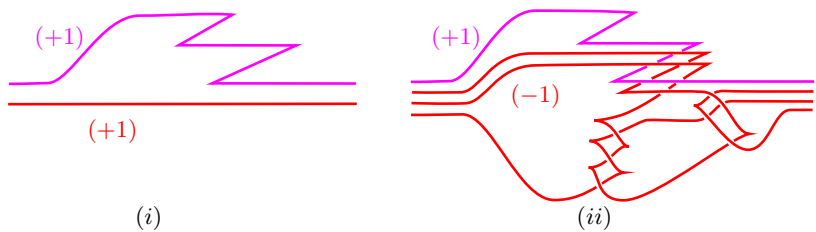			
		
		For example, the red knot in Figure~\ref{fig:S3examples}~$(i)$ has $\tb=2n-1$ and $\rot=0$, where $2n+1$ are the number of self-crossings. Thus, the surgery diagram $(i)$ yields $\xi_{1-2n}$. All positive odd $\de_3$-invariants can be realized by the same construction starting with a Legendrian unknot with $\tb=-2n-1$ and $\rot=0$ for $K$.

	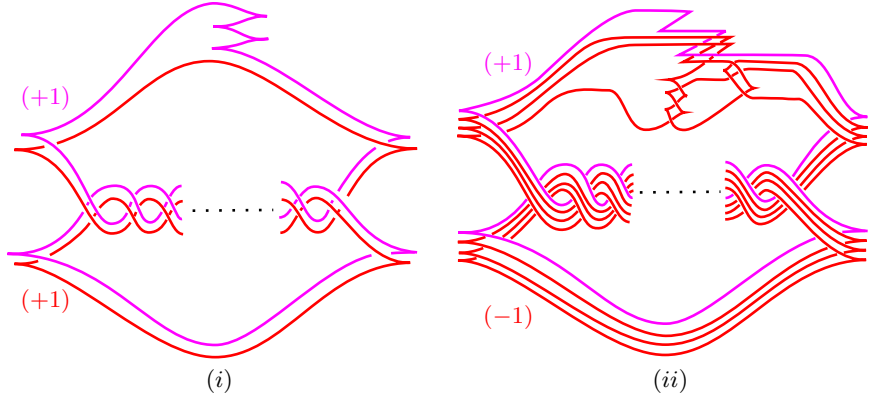
\begin{figure}[htbp]{\small
	\begin{overpic}
			{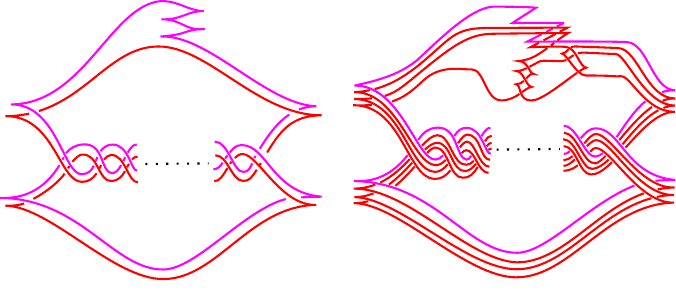}
			\put(180, 24){\color{red} $(-1)$}
			\put(5, 30){\color{red} $(+1)$}
			\put(5, 107){\color{pink} $(+1)$}
			\put(180, 120){\color{pink} $(+1)$}
			\put(75, 0){$(i)$}
			\put(244, 0){$(ii)$}
	\end{overpic}}
			\caption{Figure (i) represents $(S^3,\xi_{1-2n})$ and Figure (ii) yields $(S^3,\xi_{-2n})$, where the purple knot has $2n+1$ self-crossings. The same construction with the unknot yields positive $\de_3$-invariants.}
			\label{fig:S3examples}
\end{figure}	
			
		For the contact structures with even $d_3$-invariants, we proceed as follows. Let $K$ be an arbitrary Legendrian knot with Thurston--Bennequin invariant $t$ and rotation number $r$. Let $K_2$ be a twice positively stabilized push-off of $K$. We perform a contact $(+1)$-surgery along $K_2$ followed by a contact $(-1)$-surgery along the red Legendrian knot $L$ shown locally in Figure~\ref{fig:S3abstract} (ii). See Figure~\ref{fig:S3examples} (ii) for a global example. Note that the red surgery curve $L$ is obtained by performing an interior connected sum 
		$$L=K\#\Delta\#{-\Delta},$$ 
		where $\Delta$ denotes the boundary of an overtwisted disk in $K_2(+1)$ that is disjoint from $K$ in $K_2(+1)$. (Recall if one performs a contact $(+1)$-surgery on a stabilized knot then the original knot bounds an overtwisted disk.) 
		
		By construction, the surgery diagram yields a contact structure on $S^3$. Indeed, topologically $L$ is isotopic to $K$ in $K_2(+1)$, since $L$ is obtained from $K$ by connected summing two unknots. In the surgered manifold, $K_2(+1)$ the $(-1)$-framing of $L$ agrees with the $(+1)$-framing of $K$ and thus the Cancellation Lemma implies that the underlying smooth manifold is $S^3$. 
		
		Next, it is straightforward to compute in $(S^3,\xist)$ that $\tb(K_2)=t-2$, $\rot(K_2)=r+2$, $\tb(L)=t+2$, $\rot(L)=r$, and $\lk(K_2,L)=t$. Thus the linking matrix is
		\begin{equation*}
		Q=\begin{pmatrix}
		t-1&t\\
		t&t+1
		\end{pmatrix}.
		\end{equation*}
		And with Lemma~\ref{lem:d3} we compute the $d_3$-invariant to be
		\begin{equation*}
		\de_3=-(t+r+1).
		\end{equation*}	 
		Since $t+r$ can be any odd number, the claim follows. 
	\end{proof}
	
	Similarly, we get upper bounds on the $U$-versions of contact surgery numbers.
	
		\begin{prop} \label{prop:S3U}
			For the unknot contact surgery numbers of the overtwisted contact structures of $S^3$ the following upper bounds hold:
			\begin{align*}
\cs_{1/\Z, U}(S^3,\xi_n)&\leq\begin{cases}
2\,\text{ for }\,n \in 2\Z+1,\\
3\,\text{ for }\, n \in 2\Z ,
\end{cases}\\
\cs_{\pm1, U}(S^3,\xi_n)&\leq\begin{cases}
2\,\text{ for }\,\,n \in 2\N+1,\\
3\,\text{ for }\, \,n \in 2\N ,\\
1+\frac{|n-1|}{2}\,\text{ for }\, n\in -2\N+1,\\
2+\frac{|n|}{2}\,\text{ for }\, n\in -2\N.
\end{cases}
\end{align*}
	\end{prop}
	
\begin{proof}
	First, we consider the contact surgery diagram from Figure~\ref{fig:S3unknot}. 
	\begin{figure}[htbp]{\small
	\begin{overpic}
			{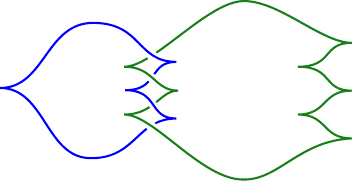}
			\put(7, 70){\color{blue} $(+1)$}
			\put(110, 70){\color{darkgreen} $(\pm \frac 1k)$}
	\end{overpic}}
	\caption{A contact surgery diagram of $(S^3,\xi_{1\pm2k})$.}
	\label{fig:S3unknot}
\end{figure}	
By computing its $\de_3$-invariant we see that it represents $(S^3,\xi_{1\pm2k})$ and thus we get the first inequality on $\cs_{1/\Z, U}(S^3,\xi_n)$. The other inequality is obtained via Lemma~\ref{lem:connectedSum} by taking connected sums of the above surgery description with $(S^3,\xi_1)$ (and observing via Lemma~\ref{lem:d3} that $\de_3$ behaves additively under contact connected sum).
	
	The proof of Proposition~\ref{prop:S3upperbound} gives the first inequality for $\cs_{\pm1, U}(S^3,\xi_n)$ and the second follows from connected sums with $\xi_1$ again. The last two inequalities follow from using the Replacement Lemma~\ref{lem:replacemenet} on the surgery diagrams used to obtain the  $\cs_{1/\Z, U}(S^3,\xi_n)$ bounds. 
\end{proof}
	
For the Legendrian surgery numbers (where only one contact $(+1)$-surgery is allowed) we have the following general upper bound.

\begin{prop}\label{prop:S3Lagrangian}
	For every contact structure $\xi$ on $S^3$ we have $\cs_{L,\pm1}(S^3,\xi)\leq3$. 	
\end{prop}

\begin{proof}
	If the $\de_3$-invariant of $\xi$ is even, the surgery diagram from Figure~\ref{fig:S3abstract} (ii) shows that $\cs_{L,\pm1}(S^3,\xi)\leq2$.
	For an overtwisted contact structure $\xi_{2n-1}$ with odd $\de_3$-invariant we can perform a single contact $(+1)$-surgery along a Legendrian unknot with $\tb=-2$ in a Darboux ball in $(S^3,\xi_{2n-1})$ to obtain $(S^3,\xi_{2n})$. By the Cancellation Lemma, we find a Legendrian knot $K$ in $(S^3,\xi_{2n})$ such that Legendrian surgery along $K$ produces $(S^3,\xi_{2n-1})$ and thus  $\cs_{L,\pm1}(S^3,\xi_{2n-1})\leq3$.
	\end{proof}

	
	\subsection{Overtwisted contact structures}\label{subsec:OTbounds}
	Next, we describe some general upper bounds on contact surgery numbers of overtwisted contact structures.
	
	\begin{prop} \label{prop.OTbound}
		Let $(M,\xi)$ be an overtwisted contact manifold. Then
		\begin{align*}
		\cs_{\pm1}(M,\xi)&\leq \su_\Z(M)+2,\\
		\cs_{L,\pm1}(M,\xi)&\leq \su_\Z(M)+3,\\
		\cs(M,\xi)&\leq \su(M)+2,\\
		\cs_{\Z,U}(M,\xi)&\leq \su_\Z(M)+2,\\
		\cs_{1/\Z,U}(M,\xi)&\leq \su_\Z(M)+2,\\
		\cs_{U}(M,\xi)&\leq\su_U(M)+2.
		\end{align*}
	\end{prop}
	
	\begin{proof}
		We start by proving the first inequality. Let $(M,\xi)$ be an overtwisted contact manifold and let $K$ be a smooth link of $\su_\Z(M)$-components in $M$ that admits an integral $S^3$-surgery. Since $(M,\xi)$ is overtwisted, we can change the contact framing of a loose Legendrian link in $(M,\xi)$ arbitrarily without changing the underlying smooth knot type by stabilizing and performing connected sums with overtwisted disks. Thus, we can choose a loose Legendrian representative $L$ of $K$ in $(M,\xi)$ such that contact $(\pm1)$-surgery yields a contact structure on $S^3$. 
		
		By Proposition~\ref{prop:S3upperbound} we can reach $(S^3,\xist)$ by performing at most $2$ more contact $(\pm1)$-surgeries. Since we can reverse a contact $(\pm1)$-surgery, we have constructed a Legendrian link in $(S^3,\xist)$ with at most $(\su_\Z(M)+2)$ components that admits a contact $(\pm1)$-surgery to $(M,\xi)$.
		
		The second inequality follows along the same lines. Here  we choose a loose Legendrian representative $L$ of $K$ in $(M,\xi)$ such that contact $(+1)$-surgery yields a contact structure on $S^3$. By the Cancellation Lemma, we can obtain $(M,\xi)$ by Legendrian surgery along a $s_\Z(M)$-component link in an overtwisted contact structure on $S^3$. Thus Proposition~\ref{prop:S3Lagrangian} implies the claimed bound.
		
		The third inequality is more complicated since the surgeries might not be integral and then we cannot reverse contact surgeries. Here the idea is to start with a topological surgery diagram of $M$ with the minimal number of components and deform it into contact surgery diagrams inducing any given $spin^c$ structure. By taking the connected sum with the contact structures of $S^3$ we get any overtwisted contact structure with that $spin^c$ structure.

		Let $K$ in $S^3$ be a link giving a topological surgery diagram of $M$ realizing $\su(M)$ and take a Legendrian realization $L$ in $(S^3,\xist)$ of $K$. If necessary we stabilize the components of $L$ such that all contact surgery coefficients yielding $M$ are larger than $1$. We have constructed a rational contact surgery diagram of some contact structure $\xi$ on $M$ along a Legendrian link with $\su(M)$ components.

		We fix a $spin$ structure $\mathfrak s$ on $M$. Now Theorem~\ref{thm:rationalHalfEuler} implies that for any given first homology class $c\in H_1(M)$ there exists a contact structure $\xi_c$ on $M$ which is obtained by rational surgery along a stabilization of $L$ such that $\Gamma(\xi_c,\mathfrak s)=c$ and in particular we get any possible $spin^c$ structure on $M$ by a rational surgery along $L$.
			By taking connected sums with $(S^3,\xi_n)$ we get surgery diagrams of all overtwisted contact structures on $M$ by adding at most $2$ more Legendrian knots to the surgery link.
		
		The remaining three inequalities follow along the same lines: we only need to observe that we can always assume the knot types of the surgery knots to be unknots if we started with a smooth surgery link $K$ that consisted only of unknots.
	\end{proof}

	
	\subsection{Tight contact structures}\label{subsec:overtwisted}
	The case of tight contact structures can now be reduced to the overtwisted contact structures. 
	
	\begin{thm}\label{thm.overtwisted}
		Let $(M,\xi)$ be a contact manifold. Then
		\begin{align*}
		\cs_{\pm1}(M,\xi)&\leq \su_\Z(M)+3,\\
		\cs_{L,\pm1}(M,\xi)&\leq \su_\Z(M)+4,\\
		\cs(M,\xi)&\leq \su(M)+3,\\
		\cs_{\Z,U}(M,\xi)&\leq \su_\Z(M)+3,\\
		\cs_{1/\Z,U}(M,\xi)&\leq \su_\Z(M)+3,\\
		\cs_{U}(M,\xi)&\leq\su_U(M)+3.
		\end{align*}
	\end{thm}

	\begin{proof}
		Let $(M,\xi)$ be a contact manifold. If $\xi$ is overtwisted Proposition~\ref{prop.OTbound} implies the result. If $\xi$ is tight we can perform a contact $(+1)$-surgery along a Legendrian unknot $U$ in $(M,\xi)$ with $\tb(U)=-2$ to get an overtwisted contact structure on $M$. Since we can reverse a contact $(+1)$-surgery (and the surgery dual knot of $U$ is again an unknot) the claimed inequalities follow from Proposition~\ref{prop.OTbound}. 
	\end{proof}
	
		\begin{rem}
	In all examples that we had considered the above proof could be improved. Indeed, in many cases (see for example the case of the Poincar\'e homology sphere or the $3$-torus in Sections~\ref{sec:P} and~\ref{sec:T3}) we could identify a Legendrian knot $K$ in $(M,\xi)$ such that contact $(\pm1)$-surgery along $K$ yields a contact manifold $K(+1)$ with $\cs_{\pm1}(K(+1))=\su(M)-1$ and thus Proposition~\ref{prop.OTbound} yields $$\cs_{\pm1}(M,\xi)\leq \su(M)+2.$$
	
	In fact, we do not know a single example of a contact manifold with $$\cs_{\pm1}(M,\xi)- \su(M)=3,$$ cf.~Question~\ref{ques:difference}. However, Proposition~\ref{prop.OTbound} implies that such a manifold has to be tight.
	\end{rem}


	\section{Computations of contact surgery numbers}\label{sec:comp}
	In this section, we explicitly compute contact surgery numbers for contact structures $\xi$ on some special manifolds. 
	
	\subsection{\texorpdfstring{$S^3$}{S3} -- integer surgeries}\label{sec:S3Integer}
	In this section, we compute integer contact surgery numbers of all contact structures on $S^3$. We start with contact $(\pm1)$-surgeries.
	
	\begin{prop}\label{prop:plusminusS3}
		The overtwisted contact structure $\xi_1$ is the unique contact structure on $S^3$ with $\cs_{\pm1}=1$. Any other overtwisted contact structure on $S^3$ has $\cs_{\pm1}=2$.  
	\end{prop}
	
	\begin{proof}
		Let $L(\pm1)$ be a contact surgery diagram of a contact structure on $S^3$ along a single Legendrian knot $L$. By a result of Gordon and Luecke~\cite{GoLu89} $L$ has to be a Legendrian unknot and the topological surgery coefficient has to be of the form $1/k$, for some integer $k\in\Z$. Moreover, Legendrian unknots are completely classified by Eliashberg--Fraser~\cite{ElFr09}: every Legendrian unknot is a stabilization of the unique Legendrian unknot $U$ with $\tb=-1$ and $\rot=0$. It follows that $L$ is a Legendrian unknot with Thurston--Bennequin invariant $t\leq-1$. Thus the topological surgery coefficient corresponding to the contact $(\pm1)$-surgery along $L$ is $t\pm 1$, which should be of the form $1/k$. The only solution for this equation is $t=-2$, $k=-1$ and the sign of $\pm1$ has to be $+1$. But there is a unique (unoriented) Legendrian knot with $\tb=-2$ and in Section~\ref{subsec:S3} we have seen that contact $(+1)$-surgery along it produces $\xi_1$. All the other overtwisted contact structures on $S^3$ have $\cs_{\pm1}\geq2$. The proposition now follows from the upper bounds in Proposition~\ref{prop:S3upperbound}.
	\end{proof}
	
	From the proof of Theorem ~\ref{prop:plusminusS3} we can directly conclude that $\xi_1$ has a unique contact $(\pm1)$-surgery diagram along a single Legendrian knot.
	
	\begin{cor}\label{cor:S3char}
		If $(S^3,\xi_1)$ is obtained by a single contact $(\pm1)$-surgery along a Legendrian knot $K$ in $(S^3,\xist)$, then $K$ has to be the (unoriented) Legendrian unknot with Thurston--Bennequin invariant $\tb=-2$ and rotation number $|\rot|=1$ and the contact surgery coefficient has to be $+1$.
	\end{cor}
	
	\begin{rem}
		The case of contact $(\pm1/n)$-surgery works the same and also yields the same result: $\xi_1$ is the unique contact structure on $S^3$ with $\cs_{1/\Z}=1$ and all other overtwisted contact structures on $S^3$ have $\cs_{1/\Z}=2$. However, in this case, $\xi_1$ does not have a unique contact $(1/n)$-surgery diagram along a single Legendrian knot anymore, it has exactly two diagrams. Indeed, contact $(1/2)$-surgery along the Legendrian unknot with $\tb=-1$ also produces $\xi_1$. 
	\end{rem}
	
	The cases of integer (and later in Section~\ref{sec:S3-rational} rational) contact surgeries are more involved. 
	
	\begin{thm}\label{thm:integerS3}
		An overtwisted contact structure on $S^3$ has $\cs_\Z=1$ if and only if its $\de_3$-invariant is of the form
		\begin{equation}
		l(1+l)\,\text{ or } \, m(1-m)+1\label{eq:surgery1}
		\end{equation}
		where $m$ and $l$ are arbitrary integers with $l\geq1$ and $m\geq0$. All other overtwisted contact structures on $S^3$ have $\cs_\Z=2$. 
	\end{thm}
	
	\begin{ex}
		As a concrete example we have $\cs_{\pm1}(\xi_0)=\cs_{\Z}(\xi_0)=2$. In other words, we cannot obtain $(S^3,\xi_0)$ by a single integer contact surgery along a Legendrian knot in $(S^3,\xist)$. In Section~\ref{sec:S3-rational} we will show that this is also not possible via a single rational contact surgery.
	\end{ex}
	
	\begin{proof}
		Let $L(k)$, for $k\in\Z\setminus\{0\}$, be an integer contact surgery diagram of a contact structure on $S^3$ along a single Legendrian knot $L$. As in the proof of Proposition~\ref{prop:plusminusS3} we conclude that $L$ has to be a Legendrian unknot with Thurston--Bennequin invariant $t\leq-1$ and $k=\pm1+t$. Since every Legendrian unknot is a stabilization of the unique Legendrian unknot with $\tb=-1$ and $\rot=0$, we can apply Corollaries~\ref{cor:+1surgery} and~\ref{cor:-1surgery} and get the claimed values for the $\de_3$-invariants. 
		
		Note, that Corollary~\ref{cor:+1surgery} indeed yields $\de_3$-invariants of the form $l(1+l)$ with $l\geq0$. However, by Lemma~\ref{thm:LiscaStipsicz} we see that the case $l=0$ always yields $\xist$ and never yields $\xi_0$. 
		
		 The second part of the theorem follows from Proposition~\ref{prop:S3upperbound}. 
	\end{proof}
	
	We get some partial results on the Legendrian contact surgery numbers of contact structures on $S^3$. 
	
	\begin{thm}\label{thm:LegendrianSurgeryNumber}
    A contact structure on $S^3$ has $\cs_L=1$ if and only if it is isotopic to $\xi_1$. Moreover, there exist infinite families of contact structures on $S^3$ with Legendrian contact surgery number equal to two. As concrete examples we have
   \begin{itemize}
	\item [(1)] $\cs_{L,\pm1}(S^3,\xi_{2k})=2$ for $k\in \Z$, and
	\item [(2)] $\cs_L(S^3,\xi_{1-2k})=2$ for $k\in \N$. 
\end{itemize}
	\end{thm}
	
	\begin{proof}
		By Proposition~\ref{prop:plusminusS3} we know that the only contact structure on $S^3$ that can be obtained by a single contact $(\pm1)$-surgery is $\xi_1$. A contact surgery diagram of $(S^3,\xi_1)$ is contact $(+1)$-surgery along the Legendrian unknot with $\tb=-2$ and thus the first claim follows. Item~(1) follows from the proof of Proposition~\ref{prop:S3upperbound}, while Item~(2) follows from the proof of Proposition~\ref{prop:S3U} and the Replacement Lemma~\ref{lem:replacemenet}. 
	\end{proof}
	
	
	\subsection{The Poincar\'e homology sphere}\label{sec:P}
	Next, we study the case of the Poincar\'e homology sphere $P$ which is the Brieskorn manifold $\Sigma(2,3,5)$. Since it is Seifert fibered with normalized Seifert invariants $(-2; 1/2, 2/3, 4/5)$ we get the surgery diagram of $P$ shown in the middle of Figure~\ref{fig:PoincareDescriptions}. By some elementary Kirby calculus, we get two simpler surgery diagrams, one along a $3$-chain link of unknots and one along the left-handed trefoil knot, shown in the left and right of Figure~\ref{fig:PoincareDescriptions}, respectively.

		\begin{figure}[htbp]{\small
	\begin{overpic}
			{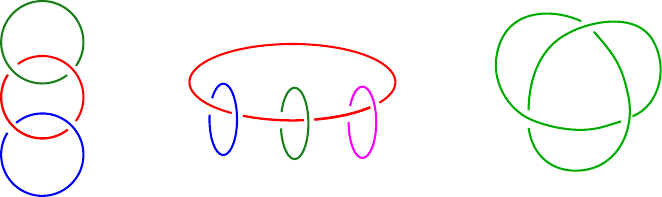}
			\put(-15, 15){\color{blue} $-\frac 32$}
			\put(-15, 45){\color{red} $-\frac 32$}
			\put(-15, 75){\color{darkgreen} $-\frac 54$}
			\put(135, 80){\color{red} $-2$}
			\put(100, 8){\color{blue} $-\frac 32$}
			\put(135, 8){\color{darkgreen} $-\frac 54$}
			\put(167, 8){\color{pink} $-2$}
			\put(300, 15){\color{darkgreen} $-1$}
			\put(60, 45){$\cong$}
			\put(210, 45){$\cong$}
	\end{overpic}}
		\caption{Three different surgery descriptions of the Poincar\'e homology sphere.}
		\label{fig:PoincareDescriptions}
\end{figure}
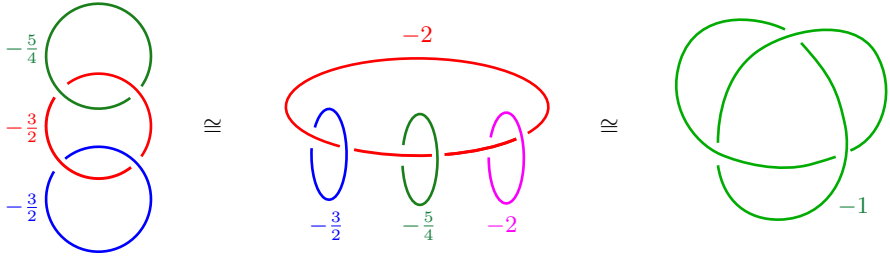
	
	Similar to the case of $S^3$ we give upper bounds on contact surgery numbers by writing down explicit diagrams. For lower bounds we will use that the Poincar\'e homology sphere has a unique surgery diagram along a single knot which is shown on the right of Figure~\ref{fig:PoincareDescriptions},  \cite{Gh08}. The Poincar\'e homology sphere $P$ has a unique tight (and in fact, Stein fillable) contact structure $\xist^P$~\cite{Sc01}. Figure~\ref{fig:Ptight} shows a Legendrian realization of the left surgery diagram from Figure~\ref{fig:PoincareDescriptions}. Since all surgery coefficients are negative the resulting contact structure is Stein fillable and thus represents $(P,\xist^P)$. (In this surgery diagram it is straightforward to compute $\de_3(\xist^P)=2$.) In particular, we get an explicit upper bound for the contact surgery number of $(P,\xist^P)$,
	\begin{equation*}
	\cs_{U,1/\Z}(P,\xist^P)\leq3.
	\end{equation*}
	
\begin{figure}[htbp]{\small
	\begin{overpic}
			{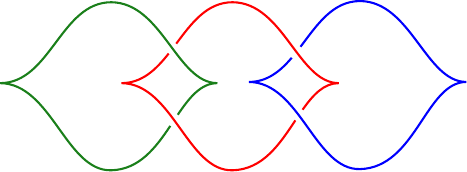}
			\put(134, 78){\color{blue} $(-\frac 12)$}
			\put(75, 78){\color{red} $(-\frac 12)$}
			\put(10, 78){\color{darkgreen} $(-\frac 14)$}
	\end{overpic}}
		\caption{A surgery diagram of $(P,\xist^P)$.}
		\label{fig:Ptight}
\end{figure}	
	
	From Figure~\ref{fig:Ptight} we also get a contact surgery presentation with only $(\pm1)$-contact surgery coefficients via Lemma~\ref{lem:replacemenet}, a diagram with $8$ components. 
	However, we can get an even better bound on $\cs_{\pm1}(P,\xist^P)$, although we do not have an explicit description of a surgery diagram realizing this bound. 
	
	\begin{prop}\label{prop:Ptightupperbound}
		The contact surgery number $\cs_{\pm1}$ of $(P,\xist^P)$ can be bounded from above as
		\begin{equation*}
		\cs_{\pm1}(P,\xist^P)\leq3.
		\end{equation*}
	\end{prop}
	
	\begin{proof}
		By~\cite[Theorem~1.5]{GO20} we know that there exists a Legendrian realization $K$ of the left-handed trefoil with $\tb=0$ in $(S^3,\xi_2)$, such that contact $(-1)$-surgery along $K$ produces a tight contact structure, see Figure~5 of~\cite{GO20} for an explicit diagram. Since $\tb(L)=0$ we know that contact $(-1)$-surgery along $K$ produces topologically $P$. Since there is a unique tight contact structure on $P$, it follows that contact $(-1)$-surgery along $K$ yields $(P,\xist^P)$. By Propositions~\ref{prop:addition} and~\ref{prop:S3upperbound} we have
		\begin{equation*}
		\cs_{\pm1}(P,\xist^P)\leq\cs_{\pm1}(S^3,\xi_2)+1=3.
		\end{equation*}
		Note, that although $\cs_\Z(S^3,\xi_2)=1$, this does not imply that $\cs_\Z(P,\xist^P)\leq2$.
	\end{proof}

From the understanding of the symplectic fillings of $(P,\xist)$ we also get the following. (Similar results can be obtained for other manifolds with a good understanding of their Stein fillings, like lens spaces or $T^3$.)

\begin{prop}
	If $L$ is a Legendrian link in $(S^3,\xist)$ such that Legendrian surgery along $L$ (i.e.\ contact $(-1)$-surgery along all components of $L$) yields $\#_n(P,\xist)$, then $L$ has exactly $8n$ components.
\end{prop}

\begin{proof}
	Theorem~2 in~\cite{OO99} says that any Stein filling of $(P,\xist)$ has intersection form of rank $8$. By Eliashberg's classification of Stein fillings on connected sums~\cite{El92} any Stein filling of $\#_n(P,\xist)$ has intersection form of rank $8n$. Since a Legendrian surgery along a Legendrian link $L$ in $(S^3,\xist)$ with $c$ components yields a Stein filling of the surgered contact manifold with intersection form of rank $c$ the claimed result follows. 
\end{proof}
	
	Next, we turn our attention to the overtwisted contact structures on the Poincar\'e homology sphere $P$. Since $P$ is a homology sphere the overtwisted contact structures are completely classified by their $\de_3$-invariants, which take values in $\Z$. Denote the unique overtwisted contact structure on $P$ with $\de_3=n$ by $\xi_n$. We can get $P$ by integer surgery along the left-handed trefoil and hence by Proposition~\ref{prop.OTbound} 
	\begin{align*}
	\cs_{\pm1}(P,\xi_n)\leq3.
	\end{align*}
	Now we will discuss how to obtain better lower bounds. We use these lower bounds to deduce that certain infinite families of contact structures on $P$ have contact surgery numbers equal to $2$. The strategy for getting lower bounds is similar to that we used in the case of $S^3$. Here we use that $P$ has a unique surgery diagram along a single knot, namely the left-handed trefoil with topological framing $-1$,~\cite{Gh08} and we use that all Legendrian realizations of left-handed trefoils in $(S^3,\xist)$ are stabilizations of the unique representative with $\tb=-6$ and $|\rot|=1$,~\cite{EtHo01}.
	
	\begin{thm}\label{thm:P}
		A contact structure on $P$ has $\cs=1$ if and only if it has $\cs_\Z=1$ if and only if its $\de_3$-invariant is of the form
		\begin{equation}
		m(3-m)-1 \label{eq:P}
		\end{equation}
		where $m$ is an arbitrary integer with $m\geq3$. 
		
		Moreover, a contact structure $\xi_n$ on $P$ has $\cs_\Z=2$ if $n$ cannot be written as in Equation~(\ref{eq:P}) but if it can be written as the sum of a number from Equation~(\ref{eq:P}) and a number from Equation~(\ref{eq:surgery1}). 
	\end{thm}
	
	Before giving the proof of Theorem~\ref{thm:P}, we formulate some examples and corollaries.
	
	\begin{cor}\label{cor:P}
		There exist infinitely many contact structures on $P$ with $\cs=2$.
	\end{cor}
	
	\begin{ex}
		As concrete examples we have $\cs(\xi_{-1})=1$ and $\cs_\Z(\xi_0)=\cs(\xi_0)=2$. Since the $\de_3$-invariants in Equation~(\ref{eq:P}) are all negative, we also have $\cs(\xi_n)\geq2$ for all $n\geq0$.
	\end{ex}
	
	\begin{cor}\label{cor:Pstd}
		The unique tight contact structure $\xist^P$ on $P$ cannot be obtained by a rational contact surgery along a single Legendrian knot from $(S^3,\xist)$. In particular, we know that 
		\begin{equation*}
		2\leq\cs(P,\xist^P)\leq\cs_{1/\Z}(P,\xist^P)\leq\cs_{\pm1}(P,\xist^P)\leq3.
		\end{equation*}
		However, $(P,\xist^P)$ can be obtained by a single Legendrian surgery (i.e. contact $(-1)$-surgery)  along a Legendrian knot in an overtwisted contact structure.
	\end{cor}
	
	\begin{proof}
		We computed $\de_3(\xist^P)=2$. Thus by Theorem~\ref{thm:P}, $(P,\xist^P)$ cannot be obtained by a single contact surgery from a Legendrian knot in $(S^3,\xist)$. The upper bound is obtained in Proposition~\ref{prop:Ptightupperbound} and the second part of the corollary follows directly from the proof of Proposition~\ref{prop:Ptightupperbound}.
	\end{proof}
	
	\begin{proof}[Proof of Theorem~\ref{thm:P}]
		Let $L(r)$, for $r\in\Q\setminus\{0\}$, be a contact surgery diagram of a contact structure on $P$ along a single Legendrian knot $L$. From the result of Ghiggini~\cite{Gh08} we conclude that $L$ is a Legendrian left-handed trefoil and the topological surgery coefficient of $L$ has to be $-1$. Now every Legendrian realization of a left-handed trefoil in $(S^3,\xist)$ is a stabilization of a unique representative with $\tb=-6$ and $|\rot|=1$~\cite{EtHo01}. (Observe, that there exists two different \textit{oriented} Legendrian realizations of the left-handed trefoil with maximal $\tb$. But here we consider unoriented knots since the contactomorphism type of the surgered contact manifold will not depend on the orientation of the knot.) It follows that we can apply Corollary~\ref{cor:-1surgery} and get the claimed values for the $\de_3$-invariants.
		
		The second part of the theorem follows by taking connected sums of the contact structures from the first part and overtwisted contact structures on $S^3$.
	\end{proof}
	
	\begin{rem}
		Here is another construction of an infinite family of contact structures on $P$ with $\cs\leq2$. According to~\cite{CN10} there are two maximal Legendrian Whitehead links, with Thurston--Bennequin invariants of the components equal to $(-3,-2)$ respectively $(-4,-1)$. (Note that the Whitehead link admits an isotopy interchanging the two components and therefore we do not need to distinguish the two components.) Since the two components of the Whitehead link are algebraically unlinked the computation of the $\de_3$ invariants of a contact structure obtained by contact surgery along a Legendrian Whitehead link is the same as the $\de_3$-invariant of the contact structure obtained by contact surgery along a Legendrian two-component unlink with the same classical invariants and contact surgery coefficients. Then we notice that we can obtain $P$ by topological $(-1)$-surgery along both components of the Whitehead link. Thus there is an infinite family of contact surgeries on Legendrian realizations of the Whitehead link giving contact structures on $P$ with $\cs=2$ that have the same $\de_3$-invariants as the contact structures on $S^3$ obtained by the corresponding contact surgeries along the Legendrian two-component unlink. By Theorem~\ref{thm:integerS3} the resulting values for the $\de_3$-invariants are
		\begin{equation*}
		m(1-m)+n(1-n)+2,
		\end{equation*}
		for $m\geq0$ and $n\geq3$ or $m\geq1$ and $n\geq2$.
	\end{rem}

\subsection{The Brieskorn sphere \texorpdfstring{$\Sigma(2,3,7)$}{E(2,3,7)}}\label{sec:E}
According to~\cite{To20} the Brieskorn homology sphere $\Sigma(2,3,7)$ has a unique tight (and in fact Stein fillable) contact structure which we denote by $\xist^\Sigma$. Moreover, denote the unique overtwisted contact structure on $\Sigma(2,3,7)$ with $\de_3$-invariant equal to $n$ by $\xi_n$. Then we have the following result about contact surgery numbers of contact structures on $\Sigma(2,3,7)$.

\begin{thm}\label{thm:E}
The unique tight contact structure $\xist^\Sigma$ on $\Sigma(2,3,7)$ can be obtained by a single Legendrian surgery along a right-handed Legendrian trefoil and thus $\cs_{\pm1}(\Sigma(2,3,7),\xist)=1$.

An overtwisted contact structure on $\Sigma(2,3,7)$ has $\cs=1$ if and only if it has $\cs_\Z=1$ if and only if its $\de_3$-invariant is of the form
\begin{equation}
l(3-l)-1\,\text{ or }\, m(m-1)\label{eq:E}
\end{equation}
where $m,l$ are arbitrary integers with $l\geq0$ and $m\geq2$. 

Moreover, a contact structure $\xi_n$ on $\Sigma(2,3,7)$ has $\cs_\Z=2$ if $n$ cannot be written as in Equation~(\ref{eq:E}) but if it can be written as the sum of a number in Equation~(\ref{eq:E}) and a number from Equation~(\ref{eq:surgery1}). 

Finally, we have $\cs_\Z(\Sigma(2,3,7),\xi)\leq3$ for any contact structure on $\Sigma(2,3,7)$.
\end{thm}

\begin{proof}
By the work of Ozsv\'ath and Szab\'o~\cite{OS19} we know that there are exactly two ways to get $\Sigma(2,3,7)$ by surgery along a single knot: $(+1)$-surgery along the figure eight knot and $(-1)$-surgery along the right-handed trefoil knot (both slopes measured with respect to the Seifert framing).

Now let $L(r)$, for $r\in\Q\setminus\{0\}$, be a contact surgery description of a contact structure on $\Sigma(2,3,7)$ along a single Legendrian knot $L$. Then we know that $L$ has to be a Legendrian realization of the figure eight knot or of the right-handed trefoil with the above topological surgery coefficients. Again by the work of Etnyre and Honda~\cite{EtHo01} we know that every Legendrian realization of the figure eight knot is a stabilization of the unique representative with $\tb=-3$ and $\rot=0$. From Corollary~\ref{cor:+1surgery} we get the second term in Equation~(\ref{eq:E}).

Similarly, we know from~\cite{EtHo01} that every Legendrian realization of a right-handed trefoil is a stabilization of the unique representative with $\tb=+1$ and $\rot=0$ and thus we get from  Corollary~\ref{cor:-1surgery} the first term in Equation~(\ref{eq:E}).

However, in this case that the maximal Thurston--Bennequin invariant is equal to $1$, additionally we have negative surgeries that we need to consider: the contact $(-2)$-surgery along the Legendrian right-handed trefoil with $\tb=1$ and $\rot=0$ (which is equivalent to contact $(-1)$-surgery along the Legendrian right-handed trefoil with $\tb=0$ and $|\rot|=1$). Note that negative contact surgeries preserve tightness and symplectic fillability, and thus this surgery yields the unique tight contact structure $\xist^\Sigma$ on $\Sigma(2,3,7)$. This proves the statements about the contact surgery number of $\xist^\Sigma$.

The last step in order to classify the contact structures with contact surgery number equal to $1$, we need to argue that all of the other contact surgeries yield an overtwisted contact structure. This can be seen by computing the $\de_3$-invariant of $\xist^\Sigma$ to be $0$ (either with Corollary~\ref{cor:-1surgery} or directly with Lemma~\ref{lem:d3}). Note that all the $\de_3$-invariants in Equation~(\ref{eq:E}) are odd or positive and thus each corresponds to overtwisted contact structures.

The statement about the contact structures with contact surgery numbers equal to two follows again by the connected sum as in the proof of Theorem~\ref{thm:P}. The general upper bounds are Proposition~\ref{prop.OTbound}.
\end{proof}

The proof of Theorem~\ref{thm:E} implies the following results.

\begin{cor}\label{cor:Est}
$(\Sigma(2,3,7), \xist^\Sigma)$ cannot be obtained by rational contact surgery along a Legendrian realization of the figure eight knot. 
\end{cor}

\begin{cor}\label{cor:Echar}
$( \Sigma(2,3,7),\xist^\Sigma)$ has a unique Legendrian surgery description along a single Legendrian knot, the contact $(-1)$-surgery along the unique (unoriented) Legendrian right-handed trefoil knot with $\tb=0$ and $|\rot|=1$. 
\end{cor}

\subsection{\texorpdfstring{$S^3$}{S3} -- rational surgeries}\label{sec:S3-rational}
With the same strategy as in Section~\ref{sec:S3Integer} we obtain results for rational contact surgery numbers of contact structures on $S^3$.

\begin{thm}\label{thm:S3rational}
An overtwisted contact structure on $S^3$ has $\cs=1$ if and only if its $\de_3$-invariant is of the form
\begin{equation}
k(q+qk-2z)\label{eq:rsurgery1}
\end{equation}
for $q\geq1$, $k\geq1$ and $z=0,1,\ldots q-1$, or
\begin{equation}
qk(k+1)+2k+1\label{eq:rsurgery2}
\end{equation}
for $q\leq-1$, $k\geq0$, or
\begin{equation}
qk(k-1)+1\label{eq:rsurgery3}
\end{equation}
for $q\leq-1$, $k\geq0$.

All other overtwisted contact structures on $S^3$ have $\cs_\Z=2$. 
\end{thm}

\begin{proof}
Let $L(r)$, for $r\in\Q\setminus\{0\}$, be a rational contact surgery diagram of a contact structure on $S^3$ along a single Legendrian knot $L$. From~\cite{GoLu89} we conclude that $L$ has to be a Legendrian unknot with Thurston--Bennequin invariant $t\leq-1$ and contact surgery coefficient $r=1/q-t$. Since every Legendrian unknot is a stabilization of the unique Legendrian unknot with $\tb=-1$ and $\rot=0$, we can get every negative odd integer as $\tb\pm\rot$ of a Legendrian unknot. We write $\tb+\rot=-1-2k$, for $k\geq0$ and we can directly apply the results from Section~\ref{sec:rational} and get the claimed values for the $\de_3$-invariants. 

We check directly that the only possibility to obtain a contact structure with $\de_3=0$ is by setting $k=0$ in Equation~\eqref{eq:rsurgery1}, but in this case it follows from Lemma~\ref{thm:LiscaStipsicz} that the resulting contact structure is always $\xist$. 

The second part of the theorem follows from Proposition~\ref{prop:S3upperbound}.
\end{proof}

\begin{cor}\label{cor:S3}
There exist infinitely many non-isotopic contact structures on $S^3$ which cannot be obtained by a single rational contact surgery from $(S^3,\xist)$. As a concrete example we see that $\cs(S^3,\xi_0)=2$.
\end{cor}

As a direct corollary of the proof, we recover a result from~\cite{Ke17,Ke18}, which implies that Legendrian knots in $(S^3,\xist)$ are determined by their exteriors.

\begin{cor} 
The only Legendrian knots along which we can obtain $(S^3,\xist)$ are Legendrian unknots with extremal rotation number, i.e.\ $|\rot|=-\tb-1$.
\end{cor}


\subsection{Contact structures on \texorpdfstring{$S^1\times S^2$}{S1 x S2}}\label{sec:S1xS2}
We know that $S^1\times S^2$ has a unique tight contact structure $\xist$. All remaining contact structures on $S^1\times S^2$ are overtwisted and by Eliashberg's classification of overtwisted contact structures they only depend on the algebraic topology of the underlying $2$-plane field. Since $H_1(S^1\times S^2)\cong\Z$ does not contain $2$-torsion, two plane fields correspond to the same $spin^c$ structure if and only if they have the same Euler class (which can be any even element in $H_1(S^1\times S^2)\cong\Z$). We can get all $2$-plane fields in a fixed $spin^c$ structure by connected summing the overtwisted contact structures on $S^3$. The first observation is that $\xist$ is the unique contact structure on $S^1\times S^2$ with $\cs_{\pm1}=1$.

\begin{prop}\label{prop:xist}
A contact structure $\xi$ on $S^1\times S^2$ has $\cs_{\pm1}(S^1\times S^2,\xi)=1$ if and only if $(S^1\times S^2,\xi)$ is contactomorphic to $(S^1\times S^2,\xist)$. Moreover, the contact $(+1)$-surgery along the Legendrian unknot with $\tb=-1$ and $\rot=0$ is the unique contact $(\pm1)$-surgery diagram of $(S^1\times S^2,\xist)$ along a single Legendrian knot in $(S^3,\xist)$.
\end{prop}

\begin{proof}
It is well-known that $(S^1\times S^2,\xist)$ can be obtained by a single contact $(+1)$-surgery along a Legendrian unknot with $\tb=-1$ and thus $\cs_{\pm1}(S^1\times S^2,\xist)=1$.

Conversely, let $\xi$ be a contact structure on $S^1\times S^2$ which can be obtained by a single contact $\pm1$ surgery along a Legendrian knot $K$ in $(S^3,\xist)$. By the work of Gabai~\cite{Ga87} $S^1\times S^2$ has a unique surgery diagram along a single knot, namely the unknot with topological surgery slope $0$. It follows that $K$ is a Legendrian unknot with contact framing $-\tb(K)$. From the classification of Legendrian unknots, we conclude that $K$ has to be the Legendrian unknot with $\tb(K)=-1$ and the contact framing is $+1$ and thus $\xi=\xist$.
\end{proof}

Next, we consider an overtwisted contact structure $\xi$ on $S^1\times S^2$. From Proposition~\ref{prop.OTbound} we get general upper bounds for overtwisted contact structures:
$$\cs_{\pm1}(S^1\times S^2,\xi)\leq3$$
And similarly we can deduce general upper bounds of the $U$-versions of contact surgery numbers.

Since surgery along the unknot with topological surgery coefficient $0$ is the unique surgery diagram of $S^1\times S^2$ along a single knot we can again classify all contact structures on $S^1\times S^2$ with $\cs=1$.

\begin{thm}\label{thm:S1S2lowerBound}
There exists exactly one contact structure in every $spin^c$ structure of $S^1\times S^2$ which can be obtained by a contact surgery along a single Legendrian knot in $(S^3,\xist)$.

In particular, no overtwisted contact structure with trivial Euler class can be obtained by surgery along a single Legendrian knot in $(S^3,\xist)$.
\end{thm}

\begin{proof} 
We will perform the same strategy as in the preceding subsections for $S^3$ and the other homology spheres. However, the difficulty here is, that $S^1\times S^2$ is not a homology sphere and the $\de_3$-invariant is not enough to understand the algebraic topology of $2$-plane fields on $S^1\times S^2$. In fact, the $\de_3$-invariant is only a well-defined rational number when the Euler class of the contact structure is a torsion element. 

To classify all contact structures on $S^1\times S^2$ with $\cs=1$ we will proceed as follows. By the above mentioned result of Gabai we know that the $\cs=1$ contact structures on $S^1\times S^2$ are exactly those which can be obtained by contact $(-t)$-surgery along a Legendrian unknot $U$ with Thurston--Bennequin invariant $t$ and rotation number $r$. Since $H^2(S^1\times S^2)\cong\Z$ has no $2$-torsion we know that a $spin^c$ structure on $S^1\times S^2$ is uniquely determined by its Euler class. We choose the explicit identification of $H^2(S^1\times S^2)$ with $\Z$ by sending the Poincar\'e dual of the meridian $\mu_U$ of $U$ to $1\in\Z$.

Then we can describe via Proposition~\ref{prop:Euler} the Euler classes of the contact $(-t)$-surgeries along $U$ as
\begin{equation*}
\e\big(U(-t)\big)=\e\left(U(+1){\def\svgwidth{1,6ex}\,\,
\begingroup%
  \makeatletter%
  \providecommand\color[2][]{%
    \errmessage{(Inkscape) Color is used for the text in Inkscape, but the package 'color.sty' is not loaded}%
    \renewcommand\color[2][]{}%
  }%
  \providecommand\transparent[1]{%
    \errmessage{(Inkscape) Transparency is used (non-zero) for the text in Inkscape, but the package 'transparent.sty' is not loaded}%
    \renewcommand\transparent[1]{}%
  }%
  \providecommand\rotatebox[2]{#2}%
  \ifx\svgwidth\undefined%
    \setlength{\unitlength}{49.54732089bp}%
    \ifx\svgscale\undefined%
      \relax%
    \else%
      \setlength{\unitlength}{\unitlength * \real{\svgscale}}%
    \fi%
  \else%
    \setlength{\unitlength}{\svgwidth}%
  \fi%
  \global\let\svgwidth\undefined%
  \global\let\svgscale\undefined%
  \makeatother%
  \begin{picture}(1,1.00059637)%
    \put(0,0){\includegraphics[width=\unitlength,page=1]{PushOff.pdf}}%
  \end{picture}%
\endgroup%
\,\,} U_1\left(-\frac{1}{-t-1}\right)\right)=t\pm r +1.
\end{equation*}

With Lemma~\ref{lem:LanternDestabilization}, we directly see that all contact structures on $S^1\times S^2$ with the same Euler class arising in this manner are contactomorphic. See Figure~\ref{fig:S1S2} for a depiction of the results. Note that the Euler classes of the contact structures on the left of Figure~\ref{fig:S1S2} differ from those of the right only by a sign. However, the induced contact structures are contactomorphic as can be seen by reversing the orientation of all surgery curves in one diagram (which does not change the contactomorphism type of the surgered contact manifold). On the other hand, since in all cases we have a surgery along a single unknot, the contact structures on the left are not isotopic to the contact structures on the right.

Finally, we observe that contact $(+1)$-surgery along the Legendrian unknot with $\tb=-1$ yields the unique tight contact structure on $S^1\times S^2$, and thus we cannot obtain an overtwisted contact structure on $S^1\times S^2$ with trivial Euler class by a contact surgery along a single Legendrian knot from $(S^3,\xist)$.
\end{proof}

\begin{figure}[htbp]{\small
	\begin{overpic}
			{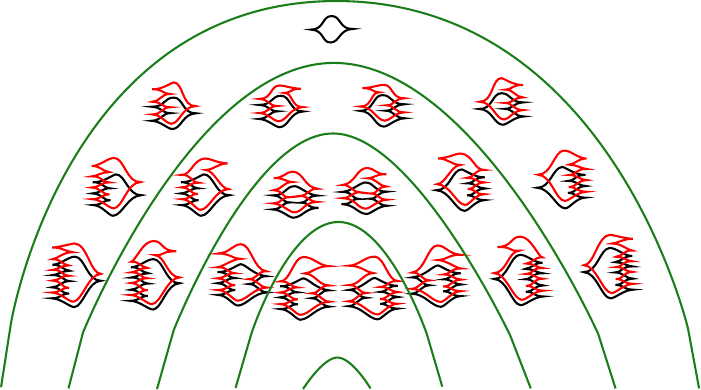}
			\put(10, 20){\color{darkgreen} $e=0$}
			\put(48, 20){\color{darkgreen} $e=\pm 2$}
			\put(88, 20 ){\color{darkgreen} $e=\pm 4$}
			\put(150, 20){\color{darkgreen} $e=\pm 6$}
			\put(214, 20){\color{darkgreen} $e=\pm 4$}
			\put(257, 20){\color{darkgreen} $e=\pm 2$}
			\put(300, 20){\color{darkgreen} $e=0$}
	\end{overpic}}
\caption{All contact structures on $S^1\times S^2$ with $\cs=1$ arise by contact $(+1)$ surgery along the black Legendrian unknots followed by a contact $(-\frac{1}{-t-1})$-contact surgery along the red knot, where all knots are oriented in the same direction. All contact structures in the regions bounded by the green arcs are contactomorphic with the indicated Euler classes.}
\label{fig:S1S2}
\end{figure}
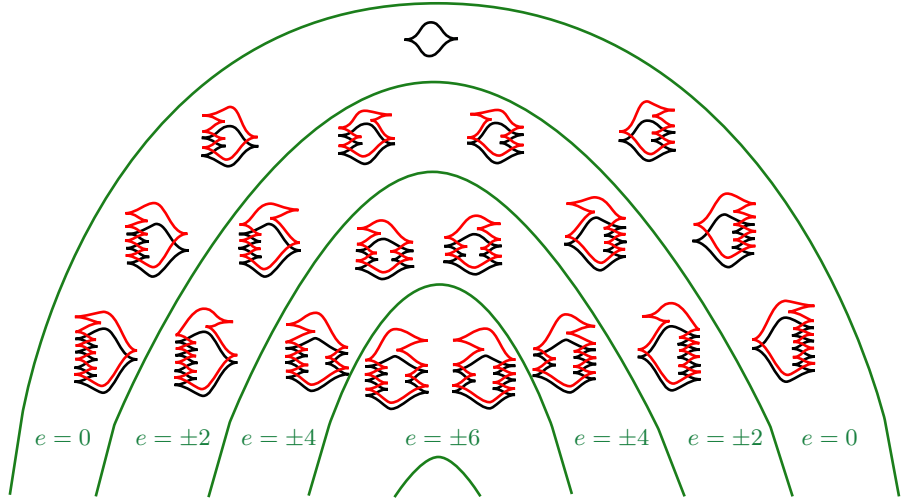	

As a corollary, we get an infinite family of overtwisted contact structures on $S^1\times S^2$ with contact surgery number two.

\begin{cor}\label{cor:S1xS2}
Let $\xi$ be an overtwisted contact structure on $S^1\times S^2$ with trivial Euler class. Then $\xi$ has $\cs=2$ if its $\de_3$-invariant is of the form
\begin{equation*}
k(q+qk+2z) -\frac 12
\end{equation*}
for $q\geq1$, $k\geq1$ and $z=0,1,\ldots q-1$, or
\begin{equation*}
qk(k+1)+2k+\frac 12
\end{equation*}
for $q\leq-1$, $k\geq0$, or
\begin{equation*}
qk(k-1)+\frac 12
\end{equation*}
for $q\leq-1$, $k\geq0$.
\end{cor}
\begin{rem}
Recall, as discussed at the end of the introduction, that our normalization of the $d_3$-invariant has $d_3(S^1\times S^2,\xist)=\frac 12$. 
\end{rem}
\begin{proof} 
We first note that the standard contact structure $\xist$ on $S^1\times S^2$ has $\de_3$-invariant equal to $1/2$, since it can be described by contact $(+1)$-surgery along a knot with vanishing rotation number. By Theorem~\ref{thm:S1S2lowerBound} we know that any overtwisted contact structure with trivial Euler class has $\cs\geq2$. 

By taking a connected sum of the contact $(+1)$-surgery along the Legendrian unknot with $\tb=-1$ and the contact surgery diagrams of the contact structures on $S^3$ with $\cs=1$ from Theorem~\ref{thm:S3rational} the claim follows.
\end{proof}

\begin{rem}
To describe the contact structures with $\cs=1$ with non-trivial Euler class in terms of their algebraic topology we can choose trivializations of the exterior of $U$ and the newly glued-in solid torus that fit together to a global trivialization of $S^1\times S^2$. (This is possible since the topological surgery coefficient is even, cf.~\cite{DGGK18}.) With respect to that fixed choice of trivialization, the Hopf invariant is a complete invariant of tangential $2$-plane fields on $S^1\times S^2$ with the same Euler classes. Then we can explicitly compute the Hopf invariants of all contact structures obtained by contact $(-t)$-surgery along all realizations of Legendrian unknots to describe their algebraic topology.
		
Alternatively, one could work directly with the finer invariants of tangential $2$-plane fields developed by Gompf~\cite{Go98} (which do not depend on a chosen trivialization of $S^1\times S^2$). 

On the other hand, there are several other ways to produce infinite families of contact structures on $S^1\times S^2$ that can be obtained by two contact surgeries. We can create such examples by performing a contact realization of a topological $0$-surgery on a Legendrian unknot in an overtwisted contact structure on $S^3$ with contact surgery number one. 
\end{rem}


\subsection{\texorpdfstring{The $3$-torus}{The 3-torus}}\label{sec:T3}
In this subsection, we briefly study the $3$-torus $T^3=\R^3/(2\pi\Z^3)$. The tight contact structures on $T^3$ are classified by Kanda~\cite{Ka97} in terms of their \textbf{Giroux torsion} $n\in\N$: any positive tight contact structure on $T^3$ is contactomorphic to exactly one of 
\begin{equation*}
\xi_n:=\ker \cos(n\theta)\,dx-\sin(n\theta)\,dy
\end{equation*}
for $n\in\N$ and Eliashberg showed that only $\xi_1$ is Stein fillable~\cite{El96}. Moreover, it is easy to see that all $\xi_n$ are homotopic as tangential $2$-plane fields. 

Since $H_1(T^3)=\Z^3$, we know that $\su(T^3)\geq3$ and since $0$-surgeries on the Borromean rings produce $T^3$ we have $\su(T^3)=3$. Thus Proposition~\ref{prop.OTbound} implies that any overtwisted contact structure on $T^3$ can be obtained by at most $5$ contact $(\pm1)$-surgeries from $(S^3,\xist)$. A Kirby diagram of the Stein filling of $(T^3,\xi_1)$ is shown in Figure~\ref{fig:3torus} on the left~\cite{Go98}. By replacing the $1$-handles with $(+1)$-framed Legendrian unknots (see Theorem~4 of~\cite{DiGe09}) we get a contact $(\pm1)$-surgery diagram of $(T^3,\xi_1)$ shown in Figure~\ref{fig:3torus} on the right along a Legendrian realization of the Borromean rings and thus $$\cs_{\pm1,U}(T^3,\xi_1)=3.$$

\begin{figure}[htbp]{\small
	\begin{overpic}
			{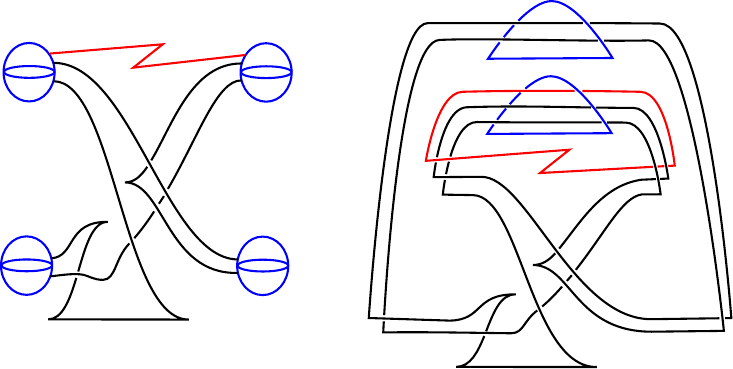}
			\put(65, 135){\color{red} $K_1$}
			\put(260, 85){\color{red} $K_1$}
			\put(288, 104 ){\color{blue} $(+1)$}
			\put(288, 140){\color{blue} $(+1)$}
			\put(93, 90){$(-1)$}
			\put(290, 48){$(-1)$}
			\put(155, 90){$\cong$}
	\end{overpic}}
	\caption{Left: A Kirby diagram of the Stein filing of $(T^3,\xi_1)$ together with the Legendrian knot $K_1$. Right: A contact $(\pm1)$-surgery diagram of $(T^3,\xi_1)$.}
	\label{fig:3torus}
\end{figure}
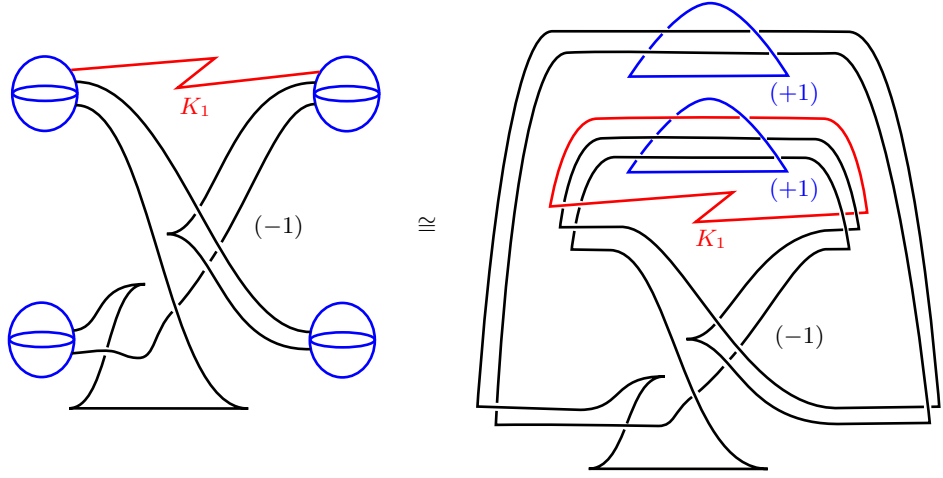	

For the other tight contact structures, we can improve the bound from Theorem~\ref{thm.overtwisted} as follows.

\begin{thm}\label{thm:T3}
	For any tight contact structure $\xi$ on $T^3$ we have
	\begin{align*}
	\cs_{\pm1}(T^3,\xi)&\leq 4.
	\end{align*}
\end{thm}

The main ingredient in the above proof is the following lemma.

\begin{lem}\label{lem:T3surgery}
	There is a fixed overtwisted contact structure $\xi$ on $\#_2 S^1\times S^2$, such that any tight contact $3$-torus $(T^3,\xi_n)$ admits a contact $(+1)$-surgery to $(\#_2 S^1\times S^2, \xi)$; moreover, $\xi$ has vanishing Euler class and $\de_3$-invariant $1$.
\end{lem}

\begin{proof}
	We consider the Legendrian knot $K$ given by $$[0,2\pi]\ni s\longmapsto (\theta(s),x(s),y(s))=(0,s,0)\in (T^3,\xi_n).$$ 
	First, we observe that $K$ lies on the $T^2_\theta$-fiber of height $\theta=0$ and thus a topological $0$-surgery measured with respect to the framing coming from the $T^2$-fiber yields $\#_2 S^1\times S^2$. 
	
	Since the contact framing of $K$ agrees with its $T^2_0$-framing it follows that contact $(+1)$-surgery along the once stabilized knot $K_1$ yields a contact structure $\eta_n$ on $\#_2 S^1\times S^2$. By Theorem~\ref{thm:ozbagci} $\eta_n$ is overtwisted.
	
	Since $\#_2 S^1\times S^2$ has no $2$-torsion its $spin^c$ structure is determined by its Euler class and if the Euler class is zero then its $\de_3$-invariant will determine the isotopy class of $\eta_n$. 
	
	We start with the case $n=1$. The Legendrian knot $K_1$ in $(T^3,\xi_1)$ is shown in the contact surgery diagram of $(T^3,\xi_1)$ in Figure~\ref{fig:3torus}. From that description, we use Lemma~\ref{lem:d3} to compute the homotopical invariants of $\eta_1$ to be
	\begin{align*}
	\e(\#_2S^1\times S^2)&=0,\\
	\de_3(\#_2S^1\times S^2)&=1.
	\end{align*}
	Finally, we can obtain $\xi_n$ from $\xi_1$ by cutting $T^3$ along the $T^2_\pi$-fiber of height $\theta=\pi$ and introduce Giroux torsion. This does not change the homotopy type of the underlying tangential $2$-plane field and thus the Euler class and the $\de_3$-invariant of $\xi_n$ are the same as for $\xi_1$. Since the contact $(+1)$-surgery along $K_1$ happens in a region where the underlying tangential $2$-plane field was not changed, the Euler class of $\eta_n$ and the $\de_3$-invariant (seen as a Hopf invariant with the appropriate trivializations of $T^3$ and $\#_2 S^1\times S^2$) are the same as for $\eta_1$.	
\end{proof}

\begin{proof} [Proof of Theorem~\ref{thm:T3}]
By Lemma~\ref{lem:T3surgery} there exists a contact $(+1)$-surgery from $(T^3,\xi_n)$ to the overtwisted contact structure $\xi$ on $\#_2 S^1\times S^2$ with vanishing Euler class and $\de_3$-invariant one. This contact structure $\xi$ can be obtained by $3$ contact $(\pm1)$-surgeries from $(S^3,\xist)$, since we easily compute that
\begin{equation*}
	(\#_2 S^1\times S^2,\xi)=(S^1\times S^2,\xist)\#(S^1\times S^2,\xist)\#(S^3,\xi_1).
\end{equation*}
\end{proof}

\begin{rem} 
The surgery dual knots $L_n$ in $(\#_2S^1\times S^2,\xi)$ from Lemma~\ref{lem:T3surgery} are all smoothly equivalent with the same classical invariants. 

Moreover, they are all exceptional, since their contact $(-1)$-surgeries yield tight manifolds. Since these tight manifolds are all different the $L_n$ are pairwise non-equivalent. We note that $L_1$ is strongly exceptional, but the other $L_n$ are not strongly exceptional since by construction we explicitly add Giroux torsion to their complements.\footnote{A Legendrian knot is called strongly exceptional if it is exceptional and contains no Giroux torsion in its complement~\cite{SV09}.}  

However, we do not have explicit contact surgery diagrams of the tight contact structures on $T^3$ or the non-loose Legendrian knots $L_n$ in $(\#_2S^1\times S^2,\xi)$. But due to the work of Van Horn-Morris~\cite{VHM07} we know compatible open book decompositions. From that open book decompositions, we can in principle construct contact surgery diagrams which will, however, have in general more than five surgery curves.
\end{rem}

\subsection{Lens spaces}\label{sec:lens}

The lens space $L(p,q)$ is defined to be the result of $(-p/q)$-surgery along the unknot. From Moser's classification of surgeries along torus knots~\cite{Mo71} it also follows that one can get some lens spaces by rational surgeries along torus knots. 

The cyclic surgery theorem~\cite{CGLS87} implies that if $K$ is not a torus knot, then at most two surgeries on $K$, which must be successive integers, can yield a lens space. A construction due to Berge gives infinite families of knots admitting surgeries yielding a lens space~\cite{Be90}. It is an open conjecture that any surgery to a lens space is of this form, cf.~\cite{Gr13}.

In some situations more is known. For example it follows from Rasmussen's work~\cite{Ra07} that the only integral surgery that produces the lens space $L(4m + 3, 4)$ is surgery along the negative torus knot $T_{(2, -(2m+ 1))}$ with coefficient $-(4m+ 3)$.

The same strategy used above will also work to study contact surgery numbers of contact structures on lens spaces. (The classification of tight contact structures on lens spaces is known~\cite{Gir00, Ho00} and the classification of Legendrian torus knots is obtained in~\cite{EtHo01}.) The only difference is that lens spaces are not homology spheres and we also need to consider the Euler classes (or in the case of $2$-torsion also the $spin^c$-structures) of the underlying $2$-plane field. But the computations will not be much more difficult.

Instead of following this route, we determine the upper and lower bounds of contact surgery numbers of the tight contact structures on the lens spaces using what is known about their symplectic fillings. From the classification of tight contact structures on lens spaces, we deduce directly that any tight contact structure on a lens space can be obtained by a single rational contact surgery along a Legendrian unknot and a Legendrian surgery along an $l$-component Legendrian link, where $l$ is the length of the negative continued fraction expansion of $-p/q$. From this, one easily deduces the following. 

\begin{cor} [Giroux~\cite{Gir00} and Honda~\cite{Ho00}]
	Let $\xi$ be any tight contact structure on the lens space $L(p,q)$, for $p>q\geq1$ and we denote by $\operatorname{length}(-p/q)$ the length of the negative continued fraction expansion of $-p/q$. Then we know
	\begin{align*}
	\cs\big(L(p,q),\xi\big)&=1,\\
	\cs_{\pm1}\big(L(p,q),\xi\big)&\leq \operatorname{length}(-p/q),\\
	\cs_{\pm1}\big(L(p,1),\xi\big)&=1.
	\end{align*}
\end{cor}

Using our knowledge about the Stein fillings of lens spaces we can also determine lower bounds. In~\cite{CY20, ER20} the classification of symplectic fillings of lens space was completed and in particular Theorem~1.9 of~\cite{ER20} says the following. 
\begin{prop}
A contact structure on the lens space $L(nm + 1, m^2)$ can be obtained from Legendrian surgery on a Legendrian realization of the $(n, -m)$–torus knot with Thurston-Bennequin invariant $-nm$; here $n$ and $m$ are relatively prime positive integers. In addition, a contact structure on the lens space $L(3n^2 + 3n + 1, 3n + 1)$ can be obtained from Legendrian surgery on a Legendrian realization of a Berge knot with Thurston-Bennequin invariant $-3n^2 - 3n$ (see~\cite[Figure~3]{ER20}).
\end{prop} 
Thus we know there are tight contact structures on these lens spaces with $cs_{\pm 1}=1$. It was conjectured in~\cite{ER20} that these (together with the tight contact structures on $L(p,1)$) are the only tight contact structures on lens spaces that can be obtained from $(S^3,\xist)$ via a single Legendrian surgery. 

Recall that if $-p/q$ has continued fraction expansion $[a_1,\ldots, a_n]$ then $L(p,q)$ is obtained from surgery on a chain of $n$ unknots with these surgery coefficients. The classification of tight contact structures implies that Legendrian surgery on all Legendrian realizations of this chain with $i^{th}$ component having Thurston--Bennequin invariant $a_i+1$ will produce all possible contact structures on $L(p,q)$. Given $\mathcal{C}$ such a Legendrian realization, let $\mathcal{C}_d$ be the components of the chain that have been stabilized both positively and negatively. Let $D$ be the cardinality of $\mathcal{C}_d$. Notice that $\mathcal{C}-\mathcal{C}_d$ consists of sub-chains of $\mathcal{C}$ and each component in the sub-chain is stabilized with only one sign. Let $I$ be the number of inconsistent sub-chains in $\mathcal{C}-\mathcal{C}_d$. An inconsistent sub-chain is one whose first element is stabilized one way, the last element is stabilized the other way, and the elements in between are not stabilized at all. We now recall Theorem~1.5 from~\cite{ER20}. 
\begin{thm}\label{eulerbound}
With the notation above, if $X$ is a Stein filling of the contact structure on $L(p,q)$ determined by $\mathcal{C}$ then its Euler characteristic satisfies
\[
\chi(X)\geq D + \lceil I/2 \rceil +1.
\]
\end{thm}
We now have the following lower bound for some contact structures on lens spaces.
\begin{thm}\label{lowerblens}
Let $\xi$ be a contact structure on $L(p,q)$ defined by a chain of Legendrian unknots $\mathcal{C}$ as above. With the notation above, if $D + \lceil I/2 \rceil$ is larger than $1$, then $cs_{\pm1}\big(L(p,q),\xi\big)\geq2.$
\end{thm}
For the proof, we need simple observation.
\begin{lem}\label{lem:lensplusone}
No tight contact structure on a lens space can be obtained from $(S^3,\xist)$ by contact $(+1)$-surgery on a Legendrian knot.
\end{lem}
\begin{proof}
Suppose that $(L(p,q),\xi)$ can be so obtained and let $L'$ be the dual of the surgery knot. We then know that $(S^3,\xist)$ can be obtained from $(L(p,q),\xi)$ by Legendrian surgery on $L'$. Now $(L(p,q),\xi)$ has a simply connected Stein filling $X$ and one can attach a $2$-handle to $L'$ and extend the Stein structure to get a Stein filling $X'$ of $(S^3,\xist)$. But this Stein filling will have non-trivial second homology but by work of Gromov and McDuff~\cite{Gromov85, McDuff90} we know that the $4$-ball is the unique Stein filling of the sphere. 
\end{proof}

\begin{rem}
	The referee has informed us about the following alternative argument for Lemma~\ref{lem:lensplusone} using Heegaard Floer homology. Suppose that a contact structure $\xi$ (overtwisted or tight) on a lens space $L(p,q)$ can be obtained by contact $(\pm1)$-surgery along a Legendrian knot $L$. Then we claim that $L$ is a Legendrian realization of a negative L-space knot. We assume on the contrary that $L$ is a positive L-space knot. (Since $L$ admits a lens space surgery it has to be a positive or negative L-space knot.) Then it follows from Greene~\cite{Gr15} and the Bennequin inequality that the topological surgery coefficient $n_{\textrm{top}}$ corresponding to the contact $(\pm1)$-surgery fulfills
	\begin{equation*}
		n_{\textrm{top}}>2g(L)>\tb(L).
	\end{equation*}
	But this implies the contradiction $\pm1=n_{\textrm{top}}-\tb(L)>1$. 
	
	Lemma~\ref{lem:lensplusone} then follows since the contact invariant of a contact $(+1)$-surgery along a negative L-space knot vanishes~\cite{GO15,MT18} and thus $\xi$ has to be overtwisted since the contact invariant of a tight contact structure on a lens space is non-vanishing.  
\end{rem}

\begin{proof}[Proof of Theorem~\ref{lowerblens}]
By the previous lemma, we see that no tight contact structure on a lens space can be obtained from $(S^3,\xist)$ by contact $(+1)$-surgery on a Legendrian knot. If a contact structure can be obtained from $(S^3,\xist)$ by contact $(-1)$-surgery on a Legendrian knot, then it has a Stein filling with Euler characteristic $2$. But given the hypothesis of our theorem, Theorem~\ref{eulerbound} says this is not possible. 
\end{proof}

\section{Stein cobordisms between contact structures on the \texorpdfstring{$3$-sphere}{3-sphere}}\label{sec:cobordism}
In this section, we will classify all Stein cobordisms with the second Betti number $b_2=1$ and no $1$-handles between contact structures on the $3$-sphere. For that, we first recall the classification of Legendrian unknots in overtwisted contact structures up to coarse equivalence. If $U$ is a non-loose Legendrian unknot in an overtwisted contact structure $\xi$ on $S^3$, then $\xi$ is isotopic to $\xi_1$ and $U$ is determined by its classical invariants $(\tb,\rot)$ which take all values in $\{(n,\pm(n-1)) : n\in\Z_{>0}\}$~\cite{ElFr09}. Surgery diagrams of all non-loose unknots are given in Figure~3 of~\cite{GO15}.

On the other hand, it was shown in~\cite{Et13} that the loose Legendrian unknots in any overtwisted contact structure on $S^3$ are classified by their classical invariants $(\tb,\rot)$ which take all values such that $\tb+\rot$ is odd. In the following, we present their surgery diagrams.

\begin{lem}\label{lem:loose}
	Figure~\ref{fig:loose} (ii) shows surgery diagrams of all loose Legendrian unknots in $(S^3,\xi_1)$. A surgery diagram of any loose Legendrian unknot in $(S^3,\xi_m)$ is given by taking the connected sum of a surgery diagram of $(S^3,\xi_{m-1})$ with the diagram of Figure~\ref{fig:loose} (ii). 
\end{lem}

\begin{figure}[htbp]{\small
	\begin{overpic}
			{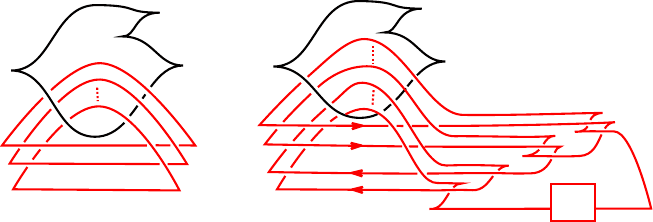}
			\put(167, 58){\color{red} $n_-$}
			\put(167, 78){\color{red} $n_+$}
			\put(274, 7){\color{red} $s$}
			\put(78, 88){$(+1)$}
			\put(208, 88){$(+1)$}
	\end{overpic}}
	\caption{Figure (i): An unlink of overtwisted disks in $(S^3,\xi_1)$. Figure (ii): All loose unknots in $(S^3,\xi_1)$. The box stands for an $s$-fold stabilization. For $s=0$ the unknot in $(S^3,\xi_1)$ has classical invariants $\tb=n_++n_- -1$ and $\rot=n_+-n_-$.}
	\label{fig:loose}
\end{figure}
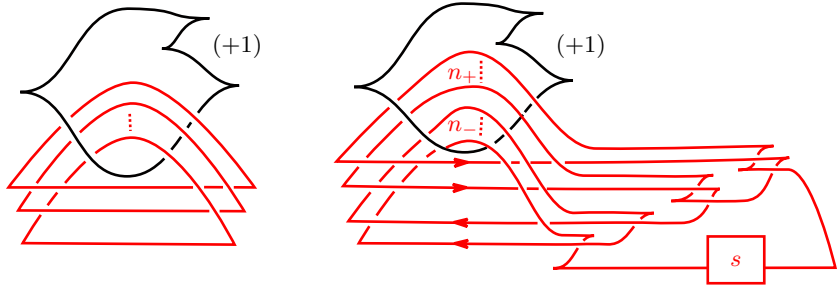	

\begin{proof}
	Figure~\ref{fig:loose} (i) shows a Legendrian unlink of boundaries $\Delta$ of overtwisted disks in $(S^3,\xi_1)$ as can be seen by performing a single Rolfsen twist undoing the surgery. The Legendrian link $U$ shown in Figure~\ref{fig:loose} (ii) is obtained from a standard Legendrian unknot with $\tb=-1$ and $\rot=0$ by performing $n_+$ connected sums with $\Delta$ and $n_-$ connected sums with $-\Delta$ followed by some number of stabilizations. Thus $U$ represents a Legendrian unknot in $(S^3,\xi_1)$, which is loose since we can find another overtwisted disk in its complement. 
	
	Now connected summing with $\pm\Delta$ changes the classical invariants by $$\big(\tb(L\#\pm\Delta),\rot(L\#\pm\Delta)\big)=\big(\tb(L)+1,\rot(L)\pm1\big)$$ and stabilization changes the classical invariants by $$\big(\tb(L_{\pm1}),\rot(L_{\pm1})\big)=\big(\tb(L)-1,\rot(L)\pm1\big)$$ and thus we see that all classical invariants (with $\tb + \rot$ odd) can be achieved.
\end{proof}

\begin{prop}\label{prop:cobordism1}
	Let $W$ be a Stein cobordism with $b_2(W)=1$ and no $1$-handles from $(S^3,\xi_1)$ to another contact $3$-sphere $(S^3,\xi_+)$. Then $W$ is obtained by attaching a Weinstein $2$-handle along a Legendrian unknot $U$ in $(S^3,\xi_1)$. Moreover, we have that
	\begin{itemize}
		\item $U$ is the non-loose unknot with $(\tb(U),\rot(U))=(2,\pm1)$ and $\xi_+$ is isotopic to $\xist$, or
		\item $U$ is a loose unknot with $(\tb(U),\rot(U))=(0,2n-1)$, for $n\in\Z$ and $\xi_+$ is isotopic to $\xi_{n(1-n)+1}$, or
		\item $U$ is a loose unknot with $(\tb(U),\rot(U))=(2,2n-1)$, for $n\in\Z$ and $\xi_+$ is isotopic to $\xi_{n(1-n)}$.
	\end{itemize}
In particular, there are infinitely many contact structures on $S^3$ that cannot be obtained from $(S^3,\xi_1)$ by a single Legendrian surgery.
\end{prop}

\begin{proof}
	Since a Stein cobordism, has a handle decomposition without $3$-handles, $W$ consists of a single Weinstein $2$-handle attached along a Legendrian knot $U$ in $(S^3,\xi_1)$. And thus $(S^3,\xi_+)$ is obtained from $(S^3,\xi_1)$ by Legendrian surgery along $U$.  By~\cite{GoLu89} $U$ has to be a Legendrian unknot with $\tb(U)=0$ or $\tb(U)=2$.
	
	First, we consider the case that $U$ is a non-loose unknot. The classification of the non-loose unknots shows that $(\tb(U),\rot(U))=(2,\pm1)$ and that Legendrian surgery along $U$ yields $(S^3,\xist)$, cf.~\cite{GO15}.
	
	If $U$ is a loose unknot with $(\tb(U),\rot(U))=(0,2n-1)$ (respectively $(2,2n-1)$) we use the surgery diagram from Lemma~\ref{lem:loose} to compute the $\de_3$-invariant of the Legendrian surgery along $U$ to be $n(1-n)+1$ (respectively $n(1-n)$). 
\end{proof}

Now, we are ready to prove the general statement.

\begin{thm}\label{thm:Stein}
	Let $W$ be a Stein cobordism with $b_2(W)=1$ and no $1$-handles from $(S^3,\xi_-)$ to another contact $3$-sphere $(S^3,\xi_+)$. Then $W$ is obtained by attaching a Weinstein $2$-handle along a Legendrian unknot $U$ in $(S^3,\xi_-)$. Moreover, we have one of the following
	\begin{itemize}
		\item $\xi_-$ is isotopic to $\xi_1$, $U$ is the non-loose unknot with $(\tb(U),\rot(U))=(2,\pm1)$ and $\xi$ is isotopic to $\xist$, or
		\item $\xi_-$ is isotopic to $\xi_m$, for $m\in\Z$, $U$ is a loose unknot with $(\tb(U),\rot(U))=(0,2n-1)$, for $n\in\Z$ and $\xi$ is isotopic to $\xi_{m+n(1-n)+1}$, or
		\item $\xi_-$ is isotopic to $\xi_m$, for $m\in\Z$, $U$ is a loose unknot with $(\tb(U),\rot(U))=(2,2n-1)$, for $n\in\Z$ and $\xi$ is isotopic to $\xi_{m+n(1-n)}$.
	\end{itemize}
Conversely, there exists a Stein cobordism with $b_2\leq3$ and no $1$-handles from any overtwisted contact structure on $S^3$ to any other contact structure on $S^3$.
\end{thm}

\begin{proof}
	The first part follows directly from Lemma~\ref{lem:loose} and Proposition~\ref{prop:cobordism1} by taking connected sums.
	
	For the second part, we first observe that the statement is equivalent to finding for every overtwisted contact structure $\xi_-$ on $S^3$ a Legendrian link $L$ in $(S^3,\xi_+)$ with at most $3$ components such that contact $(+1)$-surgery along $L$ yields $(S^3,\xi_-)$. If the difference $\de_3(\xi_+)-\de_3(\xi_-)$ is odd, we can take $L$ to be a link as in Figure~\ref{fig:S3abstract} (i)  contained in a Darboux ball in $(S^3,\xi_+)$. If the difference of the $\de_3$-invariant is even we add another disjoint copy of a Legendrian unknot with $\tb=-2$ in a Darboux ball.	 
\end{proof}



\end{document}